\documentclass[10pt]{amsart}

\usepackage{bbm}
\usepackage{eucal}
\numberwithin{equation}{subsection}
\usepackage{xcolor}
\usepackage[colorlinks=true]{hyperref}
\hypersetup{linkcolor=black, citecolor=black}
\usepackage{amsmath}
\usepackage{amsthm}
\usepackage{amssymb}
\usepackage{mathrsfs}
\usepackage{tikz-cd}
\newtheorem{thm}{Thm}[subsection]

\newtheorem{theorem}[thm]{Theorem}
\newtheorem{proposition}[thm]{Proposition}

\newtheorem{lemma}[thm]{Lemma}
\theoremstyle{definition}

\newtheorem{definition}[thm]{Definition}
\newtheorem{remark}[thm]{Remark}
\newtheorem{example}[thm]{Example}
\newtheorem{construction}[thm]{Construction}

\theoremstyle{remark}

\newcommand{\bC}{\mathbf{C}}

\newcommand{\A}{\mathbb{A}}

\newcommand{\G}{\mathbb{G}}

\newcommand{\N}{\mathbb{N}}
\renewcommand{\P}{\mathbb{P}}

\newcommand{\R}{\mathbb{R}}
\newcommand{\T}{\mathbb{T}}

\newcommand{\Z}{\mathbb{Z}}

\newcommand{\cA}{\mathcal{A}}
\newcommand{\cB}{\mathcal{B}}
\newcommand{\cC}{\mathcal{C}}
\newcommand{\cD}{\mathcal{D}}
\newcommand{\cE}{\mathcal{E}}
\newcommand{\cF}{\mathcal{F}}
\newcommand{\cG}{\mathcal{G}}
\newcommand{\cH}{\mathcal{H}}

\newcommand{\cM}{\mathcal{M}}

\newcommand{\cP}{\mathcal{P}}
\newcommand{\cQ}{\mathcal{Q}}

\newcommand{\cS}{\mathcal{S}}

\newcommand{\cU}{\mathcal{U}}
\newcommand{\cV}{\mathcal{V}}

\newcommand{\cX}{\mathcal{X}}
\newcommand{\cY}{\mathcal{Y}}
\newcommand{\cZ}{\mathcal{Z}}

\newcommand{\rD}{\mathrm{D}}

\newcommand{\sA}{\mathscr{A}}

\newcommand{\sT}{\mathscr{T}}

\DeclareMathOperator{\Hom}{Hom}

\DeclareMathOperator{\Spec}{Spec}

\newcommand{\id}{{\rm id}}

\newcommand{\infSH}{\mathrm{SH}}
\newcommand{\infDA}{\mathrm{DA}}
\newcommand{\infH}{\mathrm{H}}

\newcommand{\Set}{\mathrm{Set}}

\newcommand{\Mon}{\mathrm{Mon}}
\newcommand{\Tri}{\mathrm{Tri}}
\newcommand{\cofib}{\mathrm{cofib}}

\newcommand{\infCat}{\mathrm{Cat}}
\newcommand{\CAlg}{\mathrm{CAlg}}

\newcommand{\Corr}{\mathrm{Corr}}
\newcommand{\PrL}{\mathrm{Pr}^\mathrm{L}}
\newcommand{\PrR}{\mathrm{Pr}^\mathrm{R}}

\newcommand{\infSpc}{\mathrm{Spc}}
\newcommand{\infSpt}{\mathrm{Sp}}
\newcommand{\all}{\mathrm{all}}

\newcommand{\unit}{\mathbf{1}}

\newcommand{\dZar}{\mathrm{dZar}}
\newcommand{\dNis}{\mathrm{dNis}}
\newcommand{\sNis}{\mathrm{sNis}}
\newcommand{\dtau}{\mathrm{d\tau}}

\newcommand{\Sch}{\mathrm{Sch}}
\newcommand{\lSch}{\mathrm{lSch}}
\newcommand{\lSpc}{\mathrm{lSpc}}
\newcommand{\lFan}{\mathrm{lFan}}
\newcommand{\lSmSpc}{\mathrm{lSmSpc}}

\newcommand{\ver}{\mathrm{ver}}

\newcommand{\Cone}{\mathrm{Cone}}

\newcommand{\ol}{\overline}

\newcommand{\Blow}{\mathrm{Bl}}

\newcommand{\lSm}{\mathrm{lSm}}

\newcommand{\gp}{\mathrm{gp}}

\newcommand{\Deform}{\mathrm{D}}
\newcommand{\Normal}{\mathrm{N}}

\newcommand{\Thom}{\mathrm{Th}}

\newcommand{\Sh}{\mathrm{Sh}}
\newcommand{\PSh}{\mathrm{PSh}}

\newcommand{\boxx}{\square}

\newcommand{\colim}{\mathop{\mathrm{colim}}}
\newcommand{\limit}{\mathop{\mathrm{lim}}}
\newcommand{\Fan}{\mathrm{Fan}}
\newcommand{\Pan}{\mathrm{Pan}}
\newcommand{\Stab}{\mathrm{Sp}}
\newcommand{\Ft}{\mathrm{Ft}}

\newcommand{\ul}{\underline}

\DeclareMathOperator{\im}{im}

\newcommand{\Mod}{\mathrm{Mod}}

\begin{document}
\title{Log motivic Gysin isomorphisms}
\author{Doosung Park}
\address{Department of Mathematics and Informatics, University of Wuppertal, Germany}
\email{dpark@uni-wuppertal.de}
\date{\today}
\begin{abstract}
In this paper, we construct the Gysin isomorphisms in the axiomatic motivic setting for fs log schemes.
We formulate the purity transformations for log smooth morphisms.
We show that the purity transformations are isomorphisms for certain specific non-strict morphisms of fs log schemes.
\end{abstract}
\maketitle
\section{Introduction}

\subsection{Gysin isomorphisms in \texorpdfstring{$\A^1$}{A1}-homotopy theory of schemes}
Morel and Voevodsky \cite{MV} introduced $\A^1$-homotopy theory of schemes, which aims to study algebraic geometry by means of algebraic topology.
Many techniques in algebraic topology have been adapted to $\A^1$-homotopy theory.
For example,
Morel and Voevodsky constructed \cite[Theorem 2.23 in \S 3]{MV} the \emph{Gysin isomorphisms} in the motivic setting as follows, which is analogous to the Gysin isomorphisms in algebraic topology:
\begin{theorem}
Let $S$ be a noetherian scheme of finite dimension.
Then for every closed immersion $Z\to X$ of smooth schemes over $S$,
there exists a natural isomorphism
\[
X/(X-Z)
\simeq
\Normal_Z X/(\Normal_Z X-Z)
\]
in the motivic pointed homotopy category $\infH_*(S)$,
where $\Normal_Z X$ is the normal bundle of $Z$ in $X$, and $\Normal_Z X-Z$ is the complement of the zero section.
\end{theorem}
The \emph{Thom space $\Thom(\Normal_Z X)$} is defined to be the right hand side $\Normal_Z X/(\Normal_Z X-Z)$.
The Gysin isomorphism provides a fundamental relation among the motives of $X$, $X-Z$, and $\Thom(\Normal_Z X)$.

\subsection{Gysin isomorphisms in \texorpdfstring{$\A^1$}{A1}-homotopy theory of fs log schemes}

In \cite{logA1}, $\A^1$-homotopy theory of fs log schemes was initiated.
The motivic stable homotopy category $\infSH(S)$ of Morel-Voevodsky was extended to the case when $S$ is a noetherian fs log scheme of finite dimension.
Let us review the construction of $\infSH(S)$.
We begin with the $\infty$-category of dividing Nisnevich sheaves of spectra on the category $\lSm/S$ of log smooth fs log schemes over $S$:
\[
\Sh_{\dNis}(\lSm/S,\infSpt).
\]
The dividing Nisnevich topology in \cite[Definition 3.1.4]{logDM} is a log variant of the Nisnevich topology.
We localize this $\infty$-category by the classes of the projections $X\times \A^1\to X$ and open immersions $X-\partial_S X\to X$ for all $X\in \lSm/S$, where $\partial_S X$ is the vertical boundary \cite[Definition 2.3.5]{logA1}.
After taking $\G_m$-stabilization, we obtain $\infSH(S)$.

This paper is a continuation of \cite{logA1}.
Our first aim is to provide the log version of the Gysin isomorphisms.
We may ask the following question, where $B$ is a fixed base scheme throughout this paper:
\begin{itemize}
\item For every noetherian fs log scheme of finite dimension $S$ over $B$ and closed immersion $Z\to X$ in $\lSm/S$, is there a natural isomorphism
\[
\Sigma^\infty (X/X-Z)
\simeq
\Sigma^\infty (\Normal_Z X/\Normal_Z X-Z)
\]
in $\infSH(S)$?
\end{itemize}
The first problem we encounter is that this is not formulated well since the reasonable open complement $X-Z$ does not exist in the category of fs log schemes if $Z\to X$ is not strict.

To resolve this, one may consider a suitable log blow-up $X'\to X$ such that the projection $Z':=Z\times_X X'\to X'$ is a strict closed immersion.
Since our motivic setting inverts log blow-ups, we may instead construct an isomorphism
\[
\Sigma^\infty(X'/X'-Z')
\simeq
\Sigma^\infty(\Normal_{Z'} X'/\Normal_{Z'} X'-Z')
\]
in $\infSH(S)$.
However, there is a technical limitation of this approach: We do not know that such a log blow-up $X'\to X$ exists globally even though it exists locally as a consequence of \cite[Theorem III.2.6.7]{Ogu}.
Therefore, one has to first consider a Zariski cover of $X$ and then a log blow-up.
Then checking functoriality becomes more and more sophisticated.
This method was taken in \cite{ParThesis}.

Our strategy for this problem is to use divided log spaces over the base scheme $B$ developed in \cite{divspc}, which we review in \S \ref{div}.
A \emph{(noetherian) divided log space over $B$} is a dividing Zariski sheaf of sets on the category $\lSch/B$ of noetherian fs log schemes of finite dimensions over $B$ satisfying certain axioms that are inspired by the axioms of algebraic spaces.
We refer to \cite[Definition 3.1.4]{logDM} for the dividing Zariski topology.
The Yoneda functor
\[
h\colon \lSch/B \to \Sh_{\dZar}(\lSch/B)
\]
makes every object of $\lSch/B$ into a divided log space.
There are notions of open immersions, closed immersions, and log smooth morphisms of divided log spaces.
We can also define the dividing Zariski and dividing Nisnevich topologies on the category of divided log spaces $\lSpc/B$.

An advantage of divided log spaces is that a reasonable definition of the open complement of a closed immersion $a\colon \cZ\to \cX$ in $\lSpc/B$ exists.
Furthermore, if $a$ is a closed immersion of log smooth divided log spaces over some $\cS\in \lSpc/B$, then reasonable definitions of the blow-up $\Blow_{\cZ}\cX$ and the normal bundle $\Normal_{\cZ}\cX$ also exist.

To use divided log spaces in motivic homotopy theory, we need to extend $\infSH$ from fs log schemes to divided log spaces.
For flexibility of our theory, we instead fix a dividing Nisnevich sheaf of presentably symmetric monoidal $\infty$-categories
\[
\sT\in \Sh_{\dNis}(\lSch/B,\CAlg(\PrL))
\]
throughout this paper satisfying the following conditions:
\begin{itemize}
\item
For every morphism $f$ in $\lSch/B$, $f_*$ denotes a right adjoint of $f^*:=\sT(f)$.
\item
For every log smooth morphism $f\colon X\to S$ in $\lSch/B$, $f^*$ admits a left adjoint $f_\sharp$.
We set $M_S(X):=f_\sharp \unit\in \sT(S)$, where $\unit$ denotes the monoidal unit in $\sT(X)$.
\item (Log smooth base change)
For every cartesian square in $\lSch/B$
\[
\begin{tikzcd}
X'\ar[d,"f'"']\ar[r,"g'"]&
X\ar[d,"f"]
\\
S'\ar[r,"g"]&
S
\end{tikzcd}
\]
such that $f$ is log smooth, the Beck-Chevalley transformation
\[
f_\sharp'g'^*\xrightarrow{Ex} g^*f_\sharp
\]
is an isomorphism.
\item (Log smooth projection formula)
For every log smooth morphism $f\colon X\to S$ in $\lSch/B$, the Beck-Chevalley transformation
\[
f_\sharp((-) \otimes f^*(-))
\xrightarrow{Ex}
f_\sharp(-) \otimes (-)
\]
is an isomorphism.
\item (Localization property) For every strict closed immersion $i\colon Z\to S$ in $\lSch/B$ with its open complement $j\colon U\to S$,  $i_*$ is fully faithful, and the pair of functors $(i^*,j^*)$ is conservative.
\item ($\A^1$-invariance)
Let $X\to S$ be a log smooth morphism in $\lSch/B$, and let $p\colon X\times \A^1\to X$ be the projection.
Then the morphism 
\[
M_S(p)\colon M_S(X\times \A^1)\to M_S(X)
\]
is an isomorphism.
\item ($\ver$-invariance)
Let $f\colon X\to S$ be a log smooth morphism in $\lSch/B$, and let $j\colon X-\partial_S X\to X$ be the obvious open immersion.
Then the morphism
\[
M_S(j)
\colon
M_S(X-\partial_S X)\to M_S(X)
\]
is an isomorphism.
\item ($\P^1$-stability) Suppose $S\in \lSch/B$.
Then $\sT(S)$ is stable.
Furthermore, 
\[
\unit(1)
:=
\cofib(M_S(S)\xrightarrow{M_S(i_0)}M_S(S\times \P^1))[-2]
\]
admits a $\otimes$-inverse $\unit(-1)$, where $i_0\colon S\to S\times \P^1$ is the $0$-section.
We set $\unit(d):=\unit(1)^{\otimes d}$ for every integer $d$.
\item Suppose $S\in \lSch/B$.
The family $\{M_S(X)(d)[n]:X\in \lSm/S,\,d,n\in \Z\}$ generates $\sT(S)$.
\end{itemize}
We note that $\infSH$ satisfies all the above conditions for $\sT$, see Example \ref{formulation.2}.

We show in Proposition \ref{extension.7} that there is an equivalence of $\infty$-categories
\[
\Sh_{\dNis}(\lSch/B,\CAlg(\PrL))
\simeq
\Sh_{\dNis}(\lSpc/B,\CAlg(\PrL)),
\]
where $\lSpc/B$ denotes the category of divided log spaces over $B$.
From this, we obtain a naturally defined dividing Nisnevich sheaf
\[
\sT\in \Sh_{\dNis}(\lSpc/B,\CAlg(\PrL)).
\]
In particular, we obtain $\sT(\cS)$ for all $\cS\in \lSpc/B$.
We show in \S \ref{extension} that $\sT$ for divided log spaces satisfies similar conditions.
For this, we need the technique of partial adjoints developed by Liu-Zheng in \cite{LZ}.

If $p\colon \cX\to \cS$ is a log smooth morphism in $\lSpc/B$, we set $M_{\cS}(\cX):=p_\sharp \unit$, where $\unit$ denotes the monoidal unit in $\sT(\cX)$.
If $\cU\to \cX$ is an open immersion of log smooth divided log spaces over $\cS$, we set $M_{\cS}(\cX/\cU):=\cofib(M_{\cS}(\cU)\to M_{\cS}(\cX))$.
Using divided log spaces, our first main result can be formulated as follows.

\begin{theorem}[Theorem \ref{pur.2}]
\label{intro.1}
Let $\cZ\to \cX$ be a closed immersion of log smooth divided log spaces over $\cS$, where $\cS$ is a divided log space.
Then there is a natural isomorphism
\[
M_{\cS}(\cX/\cX-\cZ)
\simeq
M_{\cS}(\Normal_{\cZ}\cX/\Normal_{\cZ} \cX-\cZ)
\]
in $\sT(\cS)$.
\end{theorem}

\subsection{Thom transformations}

As an application of the Gysin isomorphisms, we deduce the following result.

\begin{theorem}[Theorem \ref{thom.10}]
Let $p\colon \cX\to \cS$ be a log smooth morphism in $\lSpc/B$, and let $a\colon \cS\to \cX$ be its section. Assume that $a$ is a closed immersion.
Then the natural transformation $\Thom(p,a):=p_\sharp a_*$, which we call the \emph{Thom transformation}, is an equivalence of $\infty$-categories.
\end{theorem}

We are mainly interested in the case when $p\colon X\to S$ is a separated log smooth morphism in $\lSch/B$ and $a\colon S\to X$ is its section.
Even for this case, we need Theorem \ref{intro.1} not just for fs log schemes but for divided log spaces since the open complement of $a$ does not exist in the category of fs log schemes if $a$ is not strict.

\subsection{Purity for a compactification of \texorpdfstring{$\A_\N^2\to \A_\N$}{AN2 AN}}

Suppose that $f\colon X\to S$ is a proper morphism in $\lSch/B$.
Consider the commutative diagram
\[
\begin{tikzcd}
X\ar[r,"a"]&
X\times_S X\ar[d,"p_1"']\ar[r,"p_2"]&
X\ar[d,"f"]
\\
&
X\ar[r,"f"]&
S,
\end{tikzcd}
\]
where $p_1$ (resp.\ $p_2$) is the first (resp.\ second) projection, and $a$ is the diagonal morphism.
We say that $f$ is \emph{$\sT$-pure} (or \emph{pure} for brevity) if the composition
\[
f_\sharp
\xrightarrow{\simeq}
f_\sharp p_{1*}a_*
\xrightarrow{Ex}
f_*p_{2\sharp}a_*
=
f_*\Thom(p_2,a)
\]
is an isomorphism, see Example \ref{exchange.4} for the notation $Ex$.
As observed in \cite[Theorem 3.4.2]{logA1}, the purity theorem of Voevodsky-R\"ondigs-Ayoub \cite[Theorem 1.7.17]{Ayo071} (see also \cite[Theorems 2.4.26, 2.4.28]{CD19}) implies the following result for strict smooth morphisms:

\begin{theorem}
Any strict separated smooth morphism in $\lSch/B$ is pure.
\end{theorem}

The core part of the purity theorem is the case of the projection $S\times \P^1\to S$.
The general case is a consequence of this and the localization property.
The purity theorem for schemes is not enough to establish the six-functor formalism for fs log schemes since there are lots of non-strict log smooth morphisms of fs log schemes.
Our last main result is to show that some specific non-strict proper morphisms of fs log schemes are pure as follows.
We refer to \cite[Definition III.1.2.3]{Ogu} for the notation $\Spec(P\to A)$.

\begin{theorem}
Let $S\to B\times \A_\N$ be a morphism in $\lSch/B$, and let $W$ be the gluing of
\begin{gather*}
V_1:=\Spec(\N x \oplus \N y\to \Z[x,y]),
\\
V_2:=\Spec(\N (xy) \to \Z[xy,y^{-1}]),
\\
V_3:=\Spec(\N (xy)\to \Z[xy,x^{-1}])
\end{gather*}
along
\begin{gather*}
V_{12}:=\Spec(\N x\to \Z[x,y,y^{-1}]),
\\
V_{13}:=\Spec(\N y\to \Z[x,y,x^{-1}]),
\\
V_{23}:=\Spec(\Z[x,x^{-1},y,y^{-1}]).
\end{gather*}
Consider the proper morphism $W\to \A_\N:=\Spec(\N t\to \Z[t])$ obtained by the formula $t\mapsto xy$.
Then the projection $f\colon W\times_{\A_\N} S\to S$ is pure.
\end{theorem}

This extra abundance of purity is a building block of the motivic six-functor formalism for fs log schemes in the sequel papers \cite{logshriek},\cite{logsix}.

\subsection*{Acknowledgements} This research was conducted in the framework of the DFG-funded research training group GRK 2240: \emph{Algebro-Geometric Methods in Algebra, Arithmetic and Topology}.

\subsection*{Notation and conventions}

Our standard references for log geometry and $\infty$-categories are Ogus's book \cite{Ogu} and Lurie's books \cite{HTT} and \cite{HA}.
We employ the following notation throughout this paper:

\begin{tabular}{l|l}
$B$ & base noetherian scheme of finite dimension
\\
$\Sch$ & category of noetherian schemes of finite dimensions
\\
$\lSch$ &  category of noetherian fs log schemes of finite dimensions
\\
$\lSpc/B$ & category of divided log spaces over $B$
\\
$\lSm$ & class of log smooth morphisms in $\lSch/B$
\\
$\lSmSpc$ & class of log smooth morphisms in $\lSpc/B$
\\
$\PSh(\cC)$ & category of presheaves of sets on $\cC$
\\
$\PSh(\cC,\cV)$ & $\infty$-category of presheaves with values in an $\infty$-category $\cV$ on $\cC$
\\
$\Sh_t(\cC)$ & category of $t$-sheaves of sets on $\cC$
\\
$\Sh_t(\cC,\cV)$ & $\infty$-category of $t$-sheaves with values in an $\infty$-category $\cV$ on $\cC$
\\
$\infCat_\infty$ & $\infty$-category of $\infty$-categories
\\
$\id\xrightarrow{ad}f_*f^*$ & unit of an adjunction $(f^*,f_*)$
\\
$f^*f_*\xrightarrow{ad'}\id$ & counit of an adjunction $(f^*,f_*)$
\end{tabular}

\section{Stable homotopy categories for divided log spaces}
\label{formulation}

Recall from \cite[Definition 3.1.4]{logDM} that a cartesian square
\begin{equation}
\label{formulation.0.1}
Q
:=
\begin{tikzcd}
X'\ar[d,"f'"']\ar[r,"g'"]&
X\ar[d,"f"]
\\
S'\ar[r,"g"]&
S
\end{tikzcd}
\end{equation}
in $\lSch/B$ is a 
\begin{enumerate}
\item[(1)]
\emph{Zariski distinguished square} if $Q\simeq \ul{Q}\times_{\ul{S}}S$ and $\ul{Q}$ is a Zariski distinguished square in $\Sch/B$,
\item[(2)]
\emph{strict Nisnevich distinguished square} if $Q\simeq \ul{Q}\times_{\ul{S}}S$ and $\ul{Q}$ is a Nisnevich distinguished square in $\Sch/B$,
\item[(3)]
\emph{dividing distinguished square} if $S'=X'=\emptyset$ and $f$ is a dividing cover, i.e., a surjective proper log \'etale monomorphism.
\item[(4)]
\emph{dividing Zariski distinguished square} if $Q$ is either a Zariski or dividing distinguished square,
\item[(5)]
\emph{dividing Nisnevich distinguished square} if $Q$ is either a strict Nisnevich or dividing distinguished square.
\end{enumerate}
The \emph{dividing Zariski} (resp.\ \emph{dividing Nisnevich}) \emph{cd-structure} is the collection of dividing Nisnevich distinguished square,
and its associated topology is the \emph{dividing Zariski} (resp.\ \emph{dividing Nisnevich}) \emph{topology}.

Recall that we fix a sheaf
\[
\sT\in \Sh_{\dNis}(\lSch/B,\CAlg(\PrL))
\]
satisfying the conditions in the introduction.
The sheaf condition means that for every family $\{X_i\to X\}_{i\in I}$ in $\lSch/B$ producing a dividing covering sieve,
we have an equivalence
\begin{equation}
\label{formulation.0.2}
\sT(X)
\xrightarrow{\simeq}
\lim
\big(
\prod_{i\in I} \sT(X_i)
\;
\substack{\rightarrow\\[-1em] \rightarrow}
\;
\prod_{i,j\in I} \sT(X_i\times_X X_j)
\;
\substack{\rightarrow\\[-1em] \rightarrow \\[-1em] \rightarrow}
\;
\cdots
\big),
\end{equation}
where the limit is taken in $\CAlg(\PrL)$.
Here, $\PrL$ denotes the subcategory of $\infCat_\infty$ spanned by presentable $\infty$-categories and colimit preserving functors, and $\CAlg(\PrL)$ denotes the $\infty$-category of $\mathbb{E}_\infty$-algebras in $\PrL$.
By \cite[Corollary 3.2.2.4, Lemma 3.2.2.6]{HA}, the forgetful functor $\CAlg(\PrL)\to \PrL$ is conservative and preserves limits.
By \cite[Proposition 5.5.3.13]{HTT}, the inclusion functor $\PrL\to \infCat_\infty$ preserves limits.
Hence the sheaf condition on $\sT$ is equivalent to saying that we have the equivalence \eqref{formulation.0.2} when we forget the symmetric monoidal structure on $\sT$ and the limit is taken in $\infCat_\infty$.

The sheaf condition on $\sT$ is equivalent to saying that the square of $\infty$-categories
\begin{equation}
\sT(Q)
=
\begin{tikzcd}
\sT(S)
\ar[r,"g^*"]\ar[d,"f^*"']&
\sT(S')\ar[d,"f'^*"]
\\
\sT(X)\ar[r,"g'^*"]&
\sT(X')
\end{tikzcd}
\end{equation}
is cartesian for every strict Nisnevich distinguished square $Q$ of the form \eqref{formulation.0.1} and the morphism
\[
f^*\colon \sT(S)\to \sT(X)
\]
is an equivalence for every dividing cover $f\colon X\to S$ in $\lSch/B$.

\begin{lemma}
\label{formation.3}
Let
\[
\sA\in \PSh(\lSch/B,\CAlg(\PrL))
\]
be a presheaf satisfying the localization property and log smooth base change.
If $Q$ is a strict Nisnevich distinguished square of the form \eqref{formulation.0.1} such that the induced functor
\begin{equation}
\label{formation.3.1}
\sA(S)
\to
\lim(\sA(X)\xrightarrow{g'^*}
\sA(X')\xleftarrow{f'^*}
\sA(S')
)
\end{equation}
is fully faithful,
then this functor is an equivalence.
\end{lemma}
\begin{proof}
An object of the limit is a pair of $\cF'\in \sA(S')$ and $\cG\in \sA(X)$ equipped with an isomorphism $\cG':=f'^*\cF'\simeq g'^*\cG$.
We set
\[
\cF:=\colim(g_\sharp \cF'\leftarrow h_\sharp \cG'\to f_\sharp \cG),
\]
where $h:=gf'=fg'$.
Then we have $g^*\cF\simeq \cF'$ since we have $g^*h_\sharp \cG'\simeq g^*f_\sharp \cG$ by log smooth base change.
Hence it suffices to show $g^*\cF\simeq \cF'$ to deduce that \eqref{formation.3.1} is essentially surjective.
Let $i\colon Z\to S$ be the closed complement of $g$ with the reduced scheme structure.
Form the cartesian square
\[
\begin{tikzcd}
X\ar[d,"f"']\ar[r,"i'",leftarrow]&
Z'\ar[d,"f''"]
\\
S'\ar[r,"i",leftarrow]&
Z.
\end{tikzcd}
\]
Since $Q$ is a strict Nisnevich distinguished square,
the induced morphisim $(Z')_{red}\to Z_{red}$ is an isomorphism.
Hence $f''^*$ is an equivalence of $\infty$-categories by the localization property.

By log smooth base change,
we have
$i^*f_\sharp \simeq i'^*f_\sharp''$, $i^*g_\sharp\simeq 0$, and $i'^*g_\sharp'\simeq 0$.
Together with the fact that $f''^*$ is an equivalence,
we obtain $i'^*f^*\cF\simeq i'^*\cG$.
Since $g'^*f^*\cF\simeq g'^*\cG$,
the localization property implies $f^*\cF\simeq \cG$.
Hence \eqref{formation.3.1} is essentially surjective.
\end{proof}

\begin{example}
\label{formulation.2}
Let $\tau$ be one of the topologies: strict Nisnevich, strict \'etale, and Kummer \'etale topologies.
In this case, let $\dtau$ be the dividing Nisnevich, dividing \'etale, and dividing \'etale topologies, see \cite[\S 2.1]{logA1} for a review of these topologies.
There are two fundamental examples of stable $\infty$-categories of motives over fs log schemes in \cite[Definition 2.5.5]{logA1} formulated as
\begin{gather*}
\infSH_\tau(S)
:=
\Stab_{\G_m}((\A^1\cup \ver)^{-1}\Sh_{\dtau}(\lSm/S,\infSpt)),
\\
\infDA_\tau(S,\Lambda)
:=
\Stab_{\G_m}((\A^1\cup \ver)^{-1}\Sh_{\dtau}(\lSm/S,\rD(\Mod_\Lambda))),
\end{gather*}
where $\infSpt$ denotes the $\infty$-category of spectra, $\Lambda$ is a commutative ring, and $\rD(\Mod_\Lambda)$ denotes the $\infty$-category of chain complexes of $\Lambda$-modules.
The class $\ver_S$ consists of open immersion $U\to V$ in $\lSm/S$ such that the induced morphism $U-\partial_S U\to V-\partial_S V$ is an isomorphism, where $\partial_S U$ and $\partial_S V$ are the vertical boundaries \cite[Definition 2.3.5]{logA1}.

Let us show that $\infSH_\tau$ and $\infDA_\tau(-,\Lambda)$ satisfy the above conditions for $\sT$.
We only focus on the case of $\infSH:=\infSH_\sNis$ since the proofs are similar.
For $V\in \lSm/S$, let $\Sigma^\infty V_+$ denote the object of $\infSH(S)$ represented by $V$.

The existence of $f_\sharp$ for $f\in \lSm$, log smooth base change, and log smooth projection formula are encoded in the fact that $\infSH$ is a $\lSm$-premotivic $\infty$-category, see \cite[\S 2.5]{logA1}.
Due to \cite[Theorem 3.3.9]{logA1}, $\infSH$ satisfies the localization property.

For every strict Nisnevich distinguished square $Q$ of the form \eqref{formulation.0.1} and $V\in \lSm/S$,
the induced square $\Sigma^\infty (Q\times V)_+$ is cocartesian in $\infSH(S)$.
Use this to show that \eqref{formation.3.1} is fully faithful.
Lemma \ref{formation.3} implies that $\infSH(Q)$ is cartesian.

If $f\colon X\to S$ is a dividing cover in $\lSch/B$,
then we have $f_\sharp f^*\simeq \id$ and $f^*f_\sharp\simeq \id$ since the projection $X\times_S V\to V$ is a dividing cover for all $V\in \lSm/S$.
Hence $f^*$ is an equivalence.

As a direct consequence of the construction, $\infSH$ satisfies the remaining conditions for $\sT$ 
\end{example}

\subsection{Review of divided log spaces}
\label{div}

Let us briefly review the theory of divided log spaces developed in \cite{divspc}.

For a fan $\Sigma$, let $\ul{\T_\Sigma}$ be the toric variety associated with $\Sigma$, and let $\T_\Sigma$ be the fs log scheme whose underlying scheme is $\ul{\T_\Sigma}$ and whose log structure is the compactifying log structure associated with the open immersion from the dense torus of $\ul{\T_\Sigma}$.

For $X\in \lSch/B$,
a \emph{fan chart of $X$} is a strict morphism $X\to \T_\Sigma$, where $\Sigma$ is a fan.
Let $\lFan/B$ be the full subcategory of $\lSch/B$ consisting of fs log schemes of the form $\amalg_{i\in I} X_i$ with finite $I$ such that each $X_i$ admits a fan chart.

For all $X\in \lSch/B$,
let $h_X\in \Sh_{\dZar}(\lSch/B)$ be the sheaf represented by $X$.
As observed in \cite[\S 7]{divspc}, there is an equivalence of categories
\[
\Sh_{\dZar}(\lFan/B)
\simeq
\Sh_{\dZar}(\lSch/B).
\]

\begin{definition}
\label{div.2}
Suppose that $\cP$ is a class of morphisms in $\lFan/B$ closed under pullbacks.
A morphism of sheaves $\cF\to \cG$ in $\Sh_{\dNis}(\lFan/B)$ is a \emph{representable $\cP$-morphism} if for every morphism $h_X\to \cF$ with $X\in \lFan/B$, there exists a cartesian square
\[
\begin{tikzcd}
h_Y\ar[d]\ar[r,"h_f"]&
h_X\ar[d]
\\
\cG\ar[r]&
\cF
\end{tikzcd}
\]
such that $f$ is a $\cP$-morphism in $\lFan/B$.
\end{definition}

\begin{definition}
\label{div.3}
A \emph{divided log space over $B$} is a sheaf $\cX\in \Sh_{\dZar}(\lFan/B)$ satisfying the following conditions:
\begin{enumerate}
\item[(i)]
The diagonal morphism $\cX\to \cX\times \cX$ is a representable closed immersion.
\item[(ii)]
There exists a representable Zariski cover $h_X\to \cX$ with $X\in \lFan/B$.
\end{enumerate}
A \emph{morphism of divided log spaces over $B$} is a morphism of sheaves.
The category of divided log spaces over $B$ is denoted by $\lSpc/B$.
\end{definition}

For all $X\in \lSch/B$, \cite[Proposition 4.9]{divspc} shows $h_X\in \lSpc/B$.

By \cite[Proposition 4.6]{divspc}, fiber product exists in $\lSpc/B$, and the inclusion functor $\lSpc/B\to \Sh_{\dZar}(\lFan/B)$ preserves fiber products.
It follows that the functor $h\colon \lSch/B\to \lSpc/B$ preserves fiber products.

\begin{proposition}
\label{div.1}
Let $h_U,h_V\to \cX$ be morphisms in $\lSpc/B$ with $U,V\in \lFan/B$.
Then $h_U\times_{\cX}h_V\simeq h_W$ for some $W\in \lFan/B$.
\end{proposition}
\begin{proof}
We refer to \cite[Proposition 4.3]{divspc}.
\end{proof}

\begin{definition}
\label{div.4}
Suppose that $\cP$ is one of the following classes of morphisms in $\lFan/B$: log smooth morphisms, log \'etale morphisms, open immersions, closed immersions, Zariski covers, strict Nisnevich covers, strict \'etale covers, and Kummer \'etale covers.
A morphism $\cY\to \cX$ in $\lSpc/B$ is a \emph{$\cP$-morphism} if there exists a representable Zariski cover $\cU \to \cY$ such that the composition $\cU\to \cX$ is a representable $\cP$-morphism.
We note that the class of $\cP$-morphisms (resp.\ representable $\cP$-morphisms) are closed under pullbacks and compositions by \cite[Proposition 6.5]{divspc} (resp.\ \cite[Example 3.9, Proposition 3.13]{divspc}).

Let $\lSmSpc$ denote the class of log smooth morphisms in $\lSpc/B$.
\end{definition}

\begin{proposition}
\label{div.5}
Let $\cP$ be one of the above classes of morphisms in $\lFan/B$.
Then a morphism $f\colon h_Y\to h_X$ in $\lSpc/B$ with $X,Y\in \lFan/B$ is a $\cP$-morphism if and only if $f$ is a representable $\cP$-morphism.
\end{proposition}
\begin{proof}
This is a consequence of \cite[Lemma 3.4, Proposition 6.3]{divspc}.
\end{proof}

\begin{proposition}
\label{div.6}
A morphism in $\lSpc/B$ is an open immersion (resp.\ closed immersion) if and only if it is a representable open immersion (resp.\ representable closed immersion).
\end{proposition}
\begin{proof}
We refer to \cite[Proposition 6.7]{divspc}.
\end{proof}

\begin{definition}
The \emph{dividing Zariski} (resp.\ \emph{dividing Nisnevich}) \emph{topology on $\lSpc/B$} is the topology generated by the Zariski (resp.\ strict Nisnevich) covers in $\lSpc/B$.
By \cite[Proposition 7.2]{divspc}, it is the topology generated by the representable Zariski (resp.\ representable strict Nisnevich) covers.
\end{definition}

\begin{definition}
Let $\cZ\to \cX$ be a closed immersion in $\lSpc/B$.
The \emph{open complement $\cX-\cZ$ of $\cZ$ in $\cX$} is defined to be a final object of the full subcategory of $\lSpc/\cX$ consisting of $\cY$ such that $\cY\times_{\cX} \cZ=\emptyset$.
By \cite[Theorem 8.6]{divspc}, $\cX-\cZ$ exists, and the canonical morphism $\cX-\cZ\to \cX$ is an open immersion.
\end{definition}

\begin{definition}
Let $\cZ\to \cX$ be a closed immersion in $\lSmSpc/\cS$, where $\cS\in \lSpc/B$.
The \emph{blow-up $\Blow_{\cZ}\cX$ of $\cX$ along $\cZ$} is defined to be a final object of the full subcategory of $\lSpc/\cX$ consisting of $\cY$ such that $\cZ\times_{\cX} \cY$ is an effective log Cartier divisor on $\cY$ in the sense of \cite[Definition 9.1]{divspc}.
By \cite[Theorem 9.12]{divspc}, $\Blow_{\cZ}\cX$ exists, and the canonical morphism $\Blow_{\cZ}\cX\to \cS$ is log smooth.
We set
\[
\Deform_{\cZ} \cX
:=
\Blow_{\cZ\times \{0\}}(\cX\times \boxx)
-
\Blow_{\cZ\times \{0\}}(\cX\times \{0\})
\text{ and }
\Normal_{\cZ} \cX
:=
\Deform_{\cZ}\cX \times_{\boxx}\{0\}.
\]
Here, $\boxx$ denotes the fs log scheme whose underlying scheme is $\P^1$ and whose log structure is the compactifying log structure associated with the open immersion $\A^1\to \P^1$.
\end{definition}

\subsection{Extension of motives to divided log spaces}
\label{extension}

Liu-Zheng \cite{LZ} extended the six-functor formalism from schemes to algebraic spaces (and more) using $\infty$-categories.
Likewise, we explain how to extend $\sT$ from $\lSch/B$ to $\lSpc/B$ in this subsection.
We also explore several fundamental properties of the extension of $\sT$.

\begin{proposition}
\label{extension.7}
For every presentable $\infty$-category $\cV$,
there is an equivalence of $\infty$-categories
\[
\Sh_{\dNis}(\lSch/B,\cV)
\simeq
\Sh_{\dNis}(\lSpc/B,\cV).
\]
\end{proposition}
\begin{proof}
By \cite[Proposition 7.3]{divspc},
there is an equivalence of categories.
\[
\Sh_{\dNis}(\lSch/B)
\simeq
\Sh_{\dNis}(\lSpc/B).
\]
For every $\cX\in \lSpc/B$,
there exists a Zariski cover $h_X\to \cX$ with $X\in \lFan/B$.
By Proposition \ref{div.1},
the fiber product of finite copies of $h_X$ over $\cX$ is isomorphic to $h_{V}$ for some $V\in \lFan/B$.
Since the functor $\lSch/B\to \lSpc/B$ is continuous and cocontinuous,
Theorem \ref{equivalence.8} finishes the proof.
\end{proof}

In particular,
we obtain
\[
\sT\in \Sh_{\dNis}(\lSch/B,\CAlg(\PrL))
\simeq
\Sh_{\dNis}(\lSpc/B,\CAlg(\PrL)).
\]
If $f$ is a morphism in $\lSpc/B$, then we set $f^*:=\sT(f)$.
The sheaf condition means that for every family $\{\cX_i\to \cX\}_{i\in I}$ producing a dividing Nisnevich covering sieve in $\lSpc/B$,
we have an equivalence
\begin{equation}
\label{extension.0.1}
\sT(\cX)
\xrightarrow{\simeq}
\lim
\big(
\prod_{i\in I} \sT(\cX_i)
\;
\substack{\rightarrow\\[-1em] \rightarrow}
\;
\prod_{i,j\in I} \sT(\cX_i\times_\cX \cX_j)
\;
\substack{\rightarrow\\[-1em] \rightarrow \\[-1em] \rightarrow}
\;
\cdots
\big),
\end{equation}
where the limit is taken in $\CAlg(\PrL)$.
If we forget the symmetric monoidal structures,
then the limit can be taken in either $\PrL$ or $\infCat_\infty$ since the forgetful functor $\CAlg(\PrL)\to \PrL$ and the inclusion functor $\PrL\to \infCat_\infty$ preserve limits by \cite[Corollary 3.2.2.4]{HA} and \cite[Proposition 5.5.3.13]{HTT}.

Suppose that $f\colon \cY\to \cX$ is a morphism in $\lSpc/B$ and $\cX_0\to \cX$ and $\cY_0\to \cY\times_{\cX}\cX_0$ are Zariski covers.
Let $\cX_\bullet\to \cX$ and $\cY_\bullet\to \cY$ be \v{C}ech nerves of $\cX_0\to \cX$ and $\cY_0\to \cY$,
Then take limits along the rows in the diagram
\[
\begin{tikzcd}[column sep=small]
\sT(\cX_0)
\ar[r,shift left=0.5ex]
\ar[r,shift right=0.5ex]
\ar[d,"f_0^*"']&
\sT(\cX_1)
\ar[r,shift left=1ex]
\ar[r]
\ar[r,shift right=1ex]
\ar[d,"f_1^*"]&
\cdots
\\
\sT(\cY_0)
\ar[r,shift left=0.5ex]
\ar[r,shift right=0.5ex]&
\sT(\cY_1)
\ar[r,shift left=1ex]
\ar[r]
\ar[r,shift right=1ex]&
\cdots.
\end{tikzcd}
\]
to obtain a functor
\[
\lim_{i\in \Delta} f_i^*
\colon
\lim_{i\in \Delta} \sT(\cX_i)
\to
\lim_{i\in \Delta} \sT(\cY_i).
\]
Observe that there is a commutative square with horizontal equivalences
\begin{equation}
\label{extension.0.2}
\begin{tikzcd}
\sT(\cX)\ar[r,"\simeq"]\ar[d,"f^*"']&
\lim_{i\in \Delta} \sT(\cX_i)\ar[d,"\lim_{i\in \Delta} f_i^*"]
\\
\sT(\cY)\ar[r,"\simeq"]&
\lim_{i\in \Delta} \sT(\cY_i).
\end{tikzcd}
\end{equation}

\begin{proposition}
\label{extension.6}
Let $\cX_0\to \cX$ be a Zariski cover in $\lSpc/B$,
and let $\cX_\bullet\to \cX$ be its \v{C}ech nerve.
Then there is a natural isomorphism
\[
\id
\simeq
\colim_{i\in \Delta^{op}}
f_{i_\sharp}f_i^*,
\]
where $f_i\colon \cX_i\to \cX$ is the induced morphism.
\end{proposition}
\begin{proof}
Since $\sT(\cX)\simeq \lim_{i\in \Delta}\sT(\cX_i)$,
we have natural isomorphisms
\begin{align*}
\Hom_{\sT(\cX)}(\cF,\cG)
\simeq &
\lim_{i\in \Delta} \Hom_{\sT(\cX_i)}(f_i^*\cF,f_i^*\cG)
\\
\simeq &
\lim_{i\in \Delta} \Hom_{\sT(\cX)}(f_{i_\sharp}f_i^*\cF,\cG)
\\
\simeq &
\Hom_{\sT(\cX)}(\colim_{i\in \Delta^{op}} f_{i_\sharp}f_i^*\cF,\cG)
\end{align*}
for all $\cF,\cG\in \sT(X)$.
Hence we obtain the desired natural isomorphism.
\end{proof}

\begin{proposition}
\label{extension.3}
Let $\{f_i\colon \cX_i\to \cX\}_{i\in I}$ be a family of morphisms in $\lSpc/B$ producing a dividing Nisnevich covering sieve.
Then the family of functors $\{f_i^*\}_{i\in I}$ is conservative.
\end{proposition}
\begin{proof}
Assume that $\alpha$ is a morphism in $\sT(\cX)$ such that $f_i^*\alpha$ is an isomorphism for all $i\in I$.
This implies that the image of $\alpha$ in $\sT(\cX_{i_1}\times_\cX \cdots \times_\cX \cX_{i_n})$ is an isomorphism for all $i_1,\ldots,i_n\in I$.
Together with the equivalence \eqref{extension.0.1},
we deduce that $\alpha$ is an isomorphism.
\end{proof}

\begin{proposition}
\label{extension.2}
Let $f\colon \cX\to \cS$ be a log smooth morphism in $\lSpc/B$.
Then $f^*$ admits a left adjoint $f_\sharp$.
\end{proposition}
\begin{proof}
By \cite[Proposition 6.4]{divspc},
there exists a commutative diagram
\[
\begin{tikzcd}
h_X\ar[r,"h_g"]\ar[d]&
h_S\ar[d]
\\
\cX\ar[r,"f"]&
\cS
\end{tikzcd}
\]
such that $g$ is a log smooth morphism in $\lFan/B$ and $h_S\to \cS$ and $h_X\to h_S\times_{\cS}\cX$ are representable Zariski covers.
Let $\cS_\bullet\to \cS$ and $\cX_\bullet\to \cX$ be \v{C}ech nerves of $h_S\to \cS$ and $h_X\to \cX$, and let $f_i\colon \cX_i\to \cS_i$ be the induced morphism for every integer $i\geq 0$.

We claim that $f_i$ is a representable log smooth morphism.
The claim is trivial if $i=0$.
Assume $i>0$.
Then $f_i$ admits a factorization
\[
\cX_i
\xrightarrow{\simeq}
\cX_{i-1}\times_{\cX} \cX_0
\to
\cX_{i-1}\times_{\cS}\cS_0
\to
\cS_{i-1}\times_{\cS}\cS_0
\xrightarrow{\simeq}
\cS_i.
\]
The second arrow is a representable Zariski cover since $\cX_0\to \cS_0\times_{\cS}\cX$ is a representable Zariski cover,
and the third arrow is a representable log smooth morphism by induction.
This completes the induction argument.
As a consequence,
we see that $f_i^*$ preserves limits since $f_i\simeq h_{g_i}$ for some log smooth morphism $g_i$ in $\lFan/B$.

Consider the induced diagram
\begin{equation}
\label{extension.2.1}
\begin{tikzcd}[column sep=small]
\sT(\cX_0) \ar[r,shift left=0.5ex]\ar[r,shift right=0.5ex]\ar[d,"f_0^*"']&
 \sT(\cX_1)\ar[r,shift left=1ex]\ar[r,shift right=1ex]\ar[r]\ar[d,"f_1^*"]&
\cdots
\\
\sT(\cY_0) \ar[r,shift left=0.5ex]\ar[r,shift right=0.5ex]&
\sT(\cY_1)\ar[r,shift left=1ex]\ar[r,shift right=1ex]\ar[r]&
\cdots.
\end{tikzcd}
\end{equation}
We know that the vertical arrows in \eqref{extension.2.1} preserve limits.
Take limits along the rows in \eqref{extension.2.1} and use the fact that the inclusion $\PrR\to \infCat_\infty$ preserves limits \cite[Theorem 5.5.3.18]{HTT} to conclude that $f^*$ is in $\PrR$, i.e., $f^*$ preserves limits.
\end{proof}

\begin{lemma}
\label{extension.4}
Let
\[
\begin{tikzcd}
h_{X'}\ar[r,"g'"]\ar[d,"f'"']&
h_X\ar[d,"f"]
\\
h_{S'}\ar[r,"g"]&
h_{S}
\end{tikzcd}
\]
be a cartesian square in $\lSpc/B$ with $S,S',X,X'\in \lFan/B$.
If $f$ is a representable log smooth morphism,
then the natural transformation
\begin{equation}
\label{extension.4.1}
f_\sharp'g'^*
\xrightarrow{Ex}
g^*f_\sharp
\end{equation}
is an isomorphism.
\end{lemma}
\begin{proof}
By \cite[Lemmas 3.4, 3.5]{divspc},
there exists a commutative diagram with vertical isomorphisms
\[
\begin{tikzcd}
h_{U'}\ar[d,"\simeq"']\ar[r,"h_v"]&
h_U\ar[d,"\simeq"]\ar[r,"h_u",leftarrow]&
h_V\ar[d,"\simeq"]
\\
h_{S'}\ar[r,"g"]&
h_S\ar[r,leftarrow,"f"]&
h_S
\end{tikzcd}
\]
such that $u$ is a log smooth morphism and $v$ is a morphism in $\lFan/B$.
We conclude by log smooth base change for $\sT$ over $\lSch/B$.
\end{proof}

\begin{lemma}
\label{extension.5}
Let $f\colon \cY\to \cX$ be a representable log smooth morphism in $\lSpc/B$,
and let $h_X\to \cX$ be a representable Zariski cover with $X\in \lFan/B$.
Form the \v{C}ech nerve $\cX_\bullet\to \cX$ of $h_X\to \cX$,
and set $\cY_\bullet:=\cX_\bullet\times_{\cX} \cY$.
Let $f_i\colon \cY_i\to \cX_i$ be the induced morphism for every integer $i\geq 0$.
Then there is an induced commutative diagram with horizontal equivalences
\begin{equation}
\label{extension.5.1}
\begin{tikzcd}
\sT(\cY)\ar[d,"f_\sharp"']\ar[r,"\simeq"]&
\lim_{i\in \Delta} \sT(\cY_i)\ar[d,"\lim_{i\in \Delta}f_{i\sharp}"]
\\
\sT(\cX)\ar[r,"\simeq"]&
\lim_{i\in \Delta}\sT(\cX_i).
\end{tikzcd}
\end{equation}
\end{lemma}
\begin{proof}
By Theorem \ref{corr.1}, Proposition \ref{extension.2}, and Lemma \ref{extension.4}, the induced functor
\[
\sT\colon (h(\lFan/B))^{op}\to \infCat_\infty
\]
can be extended to a functor
\[
\sT\colon \Corr(h(\lFan/B),h(\lSm),\all)\to \infCat_\infty.
\]
Together with the diagram $\Delta^{op}\times \Delta^1\to h(\lFan/B)$ given by
\[
\begin{tikzcd}[column sep=small]
\cdots
\ar[r,shift left=1ex]\ar[r]\ar[r,shift right=1ex]&
\cY_1\ar[r,shift left=0.5ex]\ar[r,shift right=0.5ex]\ar[d,"f_1"]&
\cY_0\ar[d,"f_0"]
\\
\cdots
\ar[r,shift left=1ex]\ar[r]\ar[r,shift right=1ex]&
\cX_1\ar[r,shift left=0.5ex]\ar[r,shift right=0.5ex]&
\cX_0,
\end{tikzcd}
\]
we obtain a diagram $\Delta\times \Delta^1\to \infCat_\infty$ given by
\[
\begin{tikzcd}[column sep=small]
\sT(\cY_0)
\ar[r,shift left=0.5ex]
\ar[r,shift right=0.5ex]
\ar[d,"f_{0\sharp}"']&
\sT(\cY_1)
\ar[r,shift left=1ex]
\ar[r]
\ar[r,shift right=1ex]
\ar[d,"f_{1\sharp}"]&
\cdots
\\
\sT(\cX_0)
\ar[r,shift left=0.5ex]
\ar[r,shift right=0.5ex]&
\sT(\cX_1)
\ar[r,shift left=1ex]
\ar[r]
\ar[r,shift right=1ex]&
\cdots.
\end{tikzcd}
\]
Take limits along the rows to obtain a functor
\[
\lim_{i\in \Delta} f_{i\sharp}
\colon
\lim_{i\in \Delta} \sT(\cY_i)
\to
\lim_{i\in \Delta} \sT(\cX_i).
\]

Let $p_i\colon \cX_i\to \cX$ and $q_i\colon \cY_i\to \cY$ be the induced morphisms in $\lSpc/B$,
and let $\alpha\colon \sT(\cX)\to \lim_{i\in \Delta}\sT(\cX_i)$ and $\beta\colon \sT(\cY)\to \lim_{i\in \Delta}\sT(\cY_i)$ be the induced equivalences.
Then we have natural isomorphisms
\begin{align*}
\Hom_{\sT(X)}( \alpha^{-1}\lim_{i\in \Delta}f_{i\sharp} \beta \cG,\cF)
\simeq &
\Hom_{\lim_{i\in \Delta}\sT(\cX_i)}(\lim_{i\in \Delta}f_{i\sharp} \beta \cG,\alpha \cF)
\\
\simeq &
\lim_{i\in \Delta}
\Hom_{\sT(\cX_i)}
(f_{i\sharp} q_i^*\cG,p_i^*\cF)
\\
\simeq &
\lim_{i\in \Delta}
\Hom_{\sT(\cY_i)}
(q_i^*\cG,q_i^*f^*\cF)
\\
\simeq &
\Hom_{\sT(Y)}(\cG,f^*\cF)
\\
\simeq &
\Hom_{\sT(X)}(f_\sharp \cG,\cF)
\end{align*}
for all $\cF\in \sT(X)$ and $\cG\in \sT(Y)$,
which shows the desired commutativity.
\end{proof}

\begin{proposition}
\label{extension.1}
For every cartesian square in $\lSpc/B$
\[
\begin{tikzcd}
\cX'\ar[d,"f'"']\ar[r,"g'"]&
\cX\ar[d,"f"]
\\
\cS'\ar[r,"g"]&
\cS
\end{tikzcd}
\]
such that $f\in \lSmSpc$,
the natural transformation
\begin{equation}
\label{extension.1.1}
f_\sharp'g'^*\xrightarrow{Ex} g^*f_\sharp
\end{equation}
is an isomorphism.
\end{proposition}
\begin{proof}
(I) \emph{Reduction of $f$}.
There exists a representable Zariski cover $u\colon \cU\to \cX$ such that $fu$ is a representable log smooth morphism.
Consider the induced cartesian square
\[
\begin{tikzcd}
\cU'\ar[d,"u'"']\ar[r,"g''"]&
\cU\ar[d,"u"]
\\
\cX'\ar[r,"g'"]&
\cX.
\end{tikzcd}
\]
Assume that the natural transformations
\[
u^*g_*'
\xrightarrow{Ex}
g_*''u'^*
\text{ and }
(fu)^*g_*
\xrightarrow{Ex}
g_*''(f'u')^*
\]
are isomorphisms.
The commutative diagram
\[
\begin{tikzcd}
u^*f^*g_*\ar[r,"Ex"]\ar[d,"\simeq"']&
u^*g_*'f'^*\ar[r,"Ex"]&
g_*''u'^*f'^*\ar[d,"\simeq"]
\\
(fu)^*g_*\ar[rr,"Ex"]&
&
g_*''(f'u')^*
\end{tikzcd}
\]
shows that $u^*f^*g_*\xrightarrow{Ex}u^*g_*'f'^*$ is an isomorphism.
Since $u^*$ is conservative by Proposition \ref{extension.3},
we see that \eqref{extension.1.1} is an isomorphism.
Hence we reduce to the case when $f$ is a representable log smooth morphism.

(II) \emph{Reduction of $\cS'$}.
There exists a representable Zariski cover $p\colon \cS''\to \cS'$ such that $\cS''\simeq h_{S''}$ for some $S''\in \lFan/B$.
Consider the induced cartesian square
\[
\begin{tikzcd}
\cX''\ar[d,"f''"']\ar[r,"p'"]&
\cX'\ar[d,"f'"]
\\
\cS''\ar[r,"p"]&
\cS'.
\end{tikzcd}
\]
Assume that the natural transformations
\[
f_\sharp'' p'^*
\xrightarrow{Ex}
p^*f_\sharp'
\text{ and }
f_\sharp'' (g'p')^*
\xrightarrow{Ex}
(gp)^*f_\sharp
\]
are isomorphisms.
The commutative diagram
\[
\begin{tikzcd}
f_\sharp'' p'^*g'^*
\ar[r,"Ex"]
\ar[d,"\simeq"']&
p^*f_\sharp'g'^*
\ar[r,"Ex"]&
p^*g^*f_\sharp\ar[d,"\simeq"]
\\
f_\sharp'' (g'p')^*
\ar[rr,"Ex"]&
&
(gp)^*f_\sharp
\end{tikzcd}
\]
shows that the natural transformation $p^*f_\sharp'g'^*\xrightarrow{Ex}p^*g^*f_\sharp$ is an isomorphism.
Since $p^*$ is conservative by Proposition \ref{extension.3},
we see that \eqref{extension.1.1} is an isomorphism.
Hence we reduce to the case when $\cS'\simeq h_{S'}$ for some $S'\in \lFan/B$.

(III) \emph{Reduction of $\cS$}.
Let $\cS_0\to \cS$ be a representable Zariski cover such that $\cS_0\simeq h_{S_0}$ for some $S_0\in \lFan/B$,
and let $\cS_\bullet\to \cS$ be its \v{C}ech nerve.
For every integer $i\geq 0$,
consider the induced cartesian square
\[
\begin{tikzcd}
\cX_i'\ar[r,"g_i'"]\ar[d,"f_i'"']&
\cX_i\ar[d,"f_i"]
\\
\cS_i'\ar[r,"g_i"]&
\cS_i,
\end{tikzcd}
\]
where $\cS_i':=\cS'\times_{\cS}\cS_i$,
$\cX_i:=\cX\times_{\cS}\cS_i$,
and $\cX_i':=\cX'\times_{\cS} \cS_i$.
By Proposition \ref{div.1} and Lemma \ref{extension.4},
the natural transformation
\[
f_{i\sharp}'g_i'^*
\xrightarrow{Ex}
g_i^*f_{i\sharp}
\]
is an isomorphism.
Apply $\lim_{i\in \Delta}$ to this and use \cite[Lemma 4.3.7]{LZ} to see that the natural transformation
\[
\lim_{i\in \Delta} f_{i_\sharp}' \lim_{i\in \Delta} g_i'^*
\xrightarrow{Ex}
\lim_{i\in \Delta} g_i^* \lim_{i\in \Delta}f_{i\sharp}
\]
is an isomorphism.
Together with the commutative squares \eqref{extension.0.2} and \eqref{extension.5.1},
we see that \eqref{extension.1.1} is an isomorphism.
\end{proof}

\begin{proposition}
\label{extension.8}
For every log smooth morphism $f\colon \cX\to \cS$ in $\lSpc/B$,
the natural transformation
\[
f_\sharp((-)\otimes f^*(-))
\xrightarrow{Ex}
f_\sharp(-)\otimes (-)
\]
is an isomorphism.
\end{proposition}
\begin{proof}
Let $p\colon h_X\to \cX$ be a representable Zariski cover with $X\in \lFan/B$,
and let $\cX_\bullet\to \cX$ be its \v{C}ech nerve.
Then we have the commutative diagram
\[
\begin{tikzcd}
f_\sharp (p_{i_\sharp} p_i^*\cF\otimes f^*\cG)
\ar[rr,"Ex"]
\ar[d,"Ex"']
&
&
f_\sharp p_{i_\sharp}p_i^*\cF \otimes \cG
\ar[d,"\simeq"]
\\
f_\sharp p_{i_\sharp} (p_i^*\cF\otimes p_i^*f^*\cG)\ar[r,"\simeq"]
&
(fp_i)_\sharp (p_i^*\cF\otimes (fp_i)^*\cG)
\ar[r,"Ex"]
&
(fp_i)_\sharp p_i^*\cF\otimes \cG,
\end{tikzcd}
\]
for all $i\in \Delta$,
where $p_i\colon \cX_i\to \cX$ is the induced morphism.
If we show the claim for $p_i$ and $fp_i$,
then take $\colim_{i\in \Delta^{op}}$ to the upper horizontal arrow and use Proposition \ref{extension.6} to conclude.
Hence we reduce to the case when $\cX=h_X$ with $X\in \lFan/B$.

There exists a representable Zariski cover $g\colon h_S\to \cS$ with $S\in \lFan/B$.
Form the cartesian square
\[
\begin{tikzcd}
\cX'\ar[d,"f'"']\ar[r,"g'"]&
\cX\ar[d,"f"]
\\
h_S\ar[r,"g"]&
\cS.
\end{tikzcd}
\]
Consider the commutative diagram
\[
\begin{tikzcd}
g^*f_\sharp(\cF\otimes f^*\cG)
\ar[d,"Ex"']
\ar[r,"Ex"]
&
f_\sharp' g'^*(\cF\otimes f^*\cG)
\ar[r,"\simeq"]
&
f_\sharp' (g'^*\cF\otimes f'^*g^*\cG)
\ar[d,"Ex"]
\\
g^*(f_\sharp\cF\otimes \cG)
\ar[r,"\simeq"]
&
g^*f_\sharp\cF\otimes g^*\cG
\ar[r,"Ex"]
&
f_\sharp' g'^*\cF\otimes g^*\cG
\end{tikzcd}
\]
for all $\cF\in \sT(\cX)$ and $\cG\in \sT(\cS)$.
By Propositions \ref{div.1} and \ref{extension.1},
the upper left and lower right horizontal arrows are isomorphisms.
Proposition \ref{div.5} implies that
$f'$ is isomorphic to $h_u$ for some log smooth morphism $u\colon U\to S$.
The log smooth projection formula implies that
the right vertical arrow is an isomorphism.
Hence the left vertical arrow is an isomorphism.
To conclude,
observe that $g^*$ is conservative by Proposition \ref{extension.3}.
\end{proof}

\begin{proposition}
Suppose $\cS\in \lSpc/B$.
Let $p\colon \cS\times \A^1\to \cS$ be the projection.
Then $p^*$ is fully faithful.
\end{proposition}
\begin{proof}
There exists a representable Zariski cover $g\colon h_S\to \cS$ with $S\in \lFan/B$.
Consider the induced cartesian square
\[
\begin{tikzcd}
h_S\times \A^1\ar[d,"g'"']\ar[r,"p'"]&
h_S\ar[d,"g"]
\\
\cS\times \A^1\ar[r,"p"]&
\cS.
\end{tikzcd}
\]
We have the commutative diagram
\[
\begin{tikzcd}
p_\sharp'p'^*g^*
\ar[r,"\simeq"]
\ar[rrd,"ad'"']
&
p_\sharp'g'^*p^*
\ar[r,"Ex"]
&
g^*p_\sharp p^*
\ar[d,"ad'"]
\\
&
&
g^*
\end{tikzcd}
\]
By $\A^1$-invariance,
the diagonal arrow is an isomorphism.
Proposition \ref{extension.1} implies that the upper right horizontal arrow is an isomorphism.
Hence the right vertical arrow is an isomorphism.
To conclude,
observe that $g^*$ is conservative by Proposition \ref{extension.3}.
\end{proof}

\begin{proposition}
\label{loc.1}
For every closed immersion $i\colon \cZ\to \cS$ in $\lSpc/B$ with its open complement $j\colon \cU\to \cS$,  $i_*$ is fully faithful, and the pair of functors $(i^*,j^*)$ is conservative.
\end{proposition}
\begin{proof}
Proposition \ref{div.6} implies that $i$ is a representable closed immersion.
By Propositions \ref{extension.3} and  \ref{extension.1}, the question is Zariski local on $\cS$.
Hence we may assume that $\cS=h_S$ for some $S\in \lFan/B$ and $i\simeq h_a$ for some strict closed immersion $a\colon Z\to S$.
Then use the localization property for $\sT$ over $\lSch/B$ to conclude.
\end{proof}

\begin{remark}
One can argue as in \cite[Lemma 2.3.5]{CD19} to show that the localization property for a closed immersion $i\colon \cZ\to \cS$ in $\lSpc/B$ is equivalent to requiring that $i_*$ is conservative and the sequence
\[
j_\sharp j^* \xrightarrow{ad'}
\id
\xrightarrow{ad}
i_*i^*
\]
is a cofiber sequence.
A similar remark applies to the localization property for a strict closed immersion in $\lSch/B$ too.
\end{remark}

For $\cS\in \lSpc/B$ and $\cF\in \sT(\cS)$, we set
\[
\cF(1)[2]
:=
\cF\otimes \cofib(M_{\cS}(\cS) \xrightarrow{M_{\cS}(i_0)} M_{\cS}(\cS\times \P^1))
\in
\sT(\cS),
\]
where $i_0$ is the $0$-section.

\begin{proposition}
Suppose $\cS\in \lSpc/B$.
Let $p\colon \cS\times \P^1\to \cS$ be the projection, and let $i\colon \cS\to \cS\times \P^1$ be the $0$-section. Then $p_\sharp i_*$ is an equivalence of $\infty$-categories, and $\sT(S)$ is stable.
\end{proposition}
\begin{proof}
We have natural isomorphisms
\[
p_\sharp i_*\cF
\simeq
p_\sharp i_*i^* p^*\cF
\simeq
\cF(1)[2]
\]
for $\cF\in \sT(\cS)$,
where the second one is due to Propositions \ref{extension.8} and \ref{loc.1}.
Hence the claim is equivalent to saying that $\unit[1]$ and $\unit(1)[2]$ are $\otimes$-invertible.
Let us focus on the case of $\unit(1)[2]$ since the proofs are similar.
We only need to show that the evaluation morphism
\[
\unit(1)[2]
\otimes
\ul{\Hom}(\unit(1)[2],\unit)
\to
\unit
\]
is an isomorphism,
where $\ul{\Hom}$ denotes the internal Hom in $\sT(\cS)$.

There exists a representable Zariski cover $g\colon h_{S}\to \cS$ with $S\in \lFan/B$.
Since $g^*$ is conservative by Proposition \ref{extension.3},
it suffices to show that the induced morphism
\[
g^*\unit(1)[2]
\otimes
g^*\ul{\Hom}(\unit(1)[2],\unit)
\to
g^*\unit
\]
is an isomorphism.
By adjunction,
Proposition \ref{extension.8} yields a natural isomorphism.
\[
g^*\ul{\Hom}(\unit(1)[2],\unit)
\simeq
\ul{\Hom}(g^*\unit(1)[2],g^*\unit)
\]
Hence we reduce to the case when $\cS=h_S$ with $S\in \lFan/B$.
We finish the proof by $\P^1$-stability.
\end{proof}

We summarize the results in this subsection as follows.

\begin{theorem}
The sheaf
\[
\sT
\in
\Sh_{\dNis}(\lSpc/B,\CAlg(\PrL))
\]
satisfies the following properties:
\begin{itemize}
\item For every morphism $f$ in $\lSpc/B$,
$f_*$ denotes a right adjoint of $f^*:=\sT(f)$.
\item For every log smooth morphism $f$ in $\lSpc/B$, $f^*$ admits a left adjoint $f_\sharp$.
\item (Log smooth base change)
For every cartesian square in $\lSpc/B$
\[
\begin{tikzcd}
\cX'\ar[d,"f'"']\ar[r,"g'"]&
\cX\ar[d,"f"]
\\
\cS'\ar[r,"g"]&
\cS
\end{tikzcd}
\]
such that $f$ is log smooth, the natural transformation
\[
f_\sharp'g'^*\xrightarrow{Ex} g^*f_\sharp
\]
is an isomorphism.
\item (Log smooth projection formula)
For every log smooth morphism $f\colon \cX\to \cS$ in $\lSpc/B$, the natural transformation
\[
f_\sharp((-) \otimes f^*(-))
\xrightarrow{Ex}
f_\sharp(-) \otimes (-)
\]
is an isomorphism.
\item (Localization property) For every closed immersion $i\colon \cZ\to \cS$ in $\lSpc/B$ with its open complement $j\colon \cU\to \cS$,
$i_*$ is fully faithful,
and the pair of functors $(i^*,j^*)$ is conservative.
\item ($\A^1$-invariance) 
Suppose $\cS\in \lSpc/B$.
Let $p\colon \cS\times \A^1\to \cS$ be the projection.
Then $p^*$ is fully faithful.
\item ($\P^1$-stability) Suppose $\cS\in \lSpc/B$.
Let $p\colon \cS\times \P^1\to \cS$ be the projection, and let $i\colon \cS\to \cS\times \P^1$ be the $0$-section.
Then $p_\sharp i_*$ is an equivalence of $\infty$-categories, and $\sT(S)$ is stable.
\end{itemize}
\end{theorem}

\subsection{Miscellaneous properties}

In this subsection, we collect several properties for $\sT$ that are needed later.

For $S\in \lSch/B$, let $\Ft/\ul{S}$ be the full subcategory of $\Sch/\ul{S}$ consisting of schemes of finite type over $\ul{S}$.
As observed in \cite[Theorem 3.4.2]{CD19}, the functor $(\Ft/\ul{S})^{op}\to \Tri^{\otimes}$ sending $X\in \Ft/\ul{S}$ to $\sT(X\times_{\ul{S}}S)$ is a motivic triangulated category in the sense of \cite[Definition 2.4.45]{CD19}, where $\Tri^\otimes$ denotes the $2$-category of closed symmetric monoidal triangulated categories whose $1$-morphisms are the symmetric monoidal functors and whose $2$-morphisms are the symmetric monoidal natural transformations.
In particular, we have the six-functor formalism for strict morphisms, see \cite[Theorem 3.4.2]{logA1} for the details.

\begin{proposition}
\label{fundamental.5}
Let $f\colon X\to S$ be a strict separated smooth morphism in $\lSch/B$, and let $i\colon X\to S$ be its section.
Then $f_\sharp i_*$ is an equivalence of categories.
\end{proposition}
\begin{proof}
By the above argument, we can use \cite[Proposition 2.4.11]{CD19} since $i$ is strict.
\end{proof}

\begin{proposition}
\label{fundamental.3}
Let $i\colon \cZ\to \cS$ be a closed immersion in $\lSpc/B$.
Then $i_*$ admits a right adjoint $i^!$.
\end{proposition}
\begin{proof}
Let $j\colon \cS-\cZ\to \cS$ be the open complement of $i$.
Let $\colim_{a\in I} \cF_a$ be a colimit in $\sT(\cZ)$.
We need to show that the induced morphism
\[
\colim_{a\in I} i_*\cF
\to
i_*\colim_{a\in I}\cF
\]
is an isomorphism.
Since $(i^*,j^*)$ is conservative, it suffices to show that the induced morphisms
\[
i^*\colim_{a\in I} i_*\cF
\to
i^*i_*\colim_{a\in I}\cF
\text{ and }
j^*\colim_{a\in I} i_*\cF
\to
j^*i_*\colim_{a\in I}\cF
\]
are isomorphisms.
To conclude, observe that $i^*$ and $j^*$ preserve colimits, and use $i^*i_*\simeq \id$ and $j^*i_*\simeq 0$. 
\end{proof}

\begin{proposition}
\label{fundamental.2}
For all $X,Y\in \lSch/B$, there is a canonical equivalence
\[
\sT(X)\times \sT(Y) \simeq \sT(X\amalg Y).
\]
A similar claim holds for $\lSpc/B$ too.
\end{proposition}
\begin{proof}
Argue as in \cite[Lemma 2.3.7]{CD19}.
\end{proof}

\begin{proposition}
\label{fundamental.7}
Let
\[
\begin{tikzcd}
X'\ar[r,"g'"]\ar[d,"f'"']&
X\ar[d,"f"]
\\
S'\ar[r,"g"]&
S
\end{tikzcd}
\]
be a cartesian square in $\lSch/B$.
If $f$ is an open immersion and $g$ is strict proper, then the natural transformation
\[
Ex\colon f_\sharp g_*'\to g_*f_\sharp'
\]
is an equivalence.
\end{proposition}
\begin{proof}
This is observed in \cite[\S 3.4]{logA1}.
\end{proof}

\begin{proposition}
\label{loc.2}
Let
\[
\begin{tikzcd}
\cZ'\ar[d,"f'"']\ar[r,"i'"]&
\cX'\ar[d,"f"]
\\
\cZ\ar[r,"i"]&
\cX
\end{tikzcd}
\]
be a cartesian square in $\lSpc/B$ such that $i$ is a closed immersion.
Then the natural transformation
\begin{equation}
\label{loc.2.1}
Ex\colon f^*i_*\to i_*'f'^*
\end{equation}
is an isomorphism.
If $f\in \lSmSpc$, then the natural transformation
\begin{equation}
\label{loc.2.2}
Ex\colon f_\sharp i_*'\to i_*f_\sharp'
\end{equation}
is an isomorphism.
\end{proposition}
\begin{proof}
Consider the cartesian square
\[
\begin{tikzcd}
\cU'\ar[d,"f''"']\ar[r,"j'"]&
\cX'\ar[d,"f"]
\\
\cU\ar[r,"j"]&
\cX,
\end{tikzcd}
\]
where $j$ is the open complement of $i$.
Use the fact that $i_*$ and $i_*'$ are fully faithful to show that the natural transformation
\[
i'^*f^*i_*\xrightarrow{Ex} i'^*i_*'f'^*
\]
is an isomorphism.
Use $j^*i_*=0$ and $j'^*i_*'=0$ to show that the natural transformation
\[
j'^*f^*i_*\xrightarrow{Ex} j'^*i_*'f'^*
\]
is an isomorphism.
Together with the fact that the pair $(i'^*,j'^*)$ is conservative, we deduce that \eqref{loc.2.1} is an isomorphism.

Suppose $f\in \lSmSpc$.
By a similar argument, the natural transformations
\[
i^* f_\sharp i_*'\xrightarrow{Ex} i^*i_*f_\sharp'
\text{ and }
j^* f_\sharp i_*'\xrightarrow{Ex} j^*i_*f_\sharp'
\]
are isomorphisms.
Use the conservativity of $(i^*,j^*)$ to conclude.
\end{proof}

\begin{proposition}
[$\ver$-invariance]
\label{loc.3}
Let $f\colon X\to S$ be a log smooth morphism in $\lSch/B$, and let $j\colon X-\partial_S X\to X$ be the obvious open immersion.
Then the natural transformation
\[
f^*\xrightarrow{ad} j_*j^*f^*
\]
is an isomorphism.
\end{proposition}
\begin{proof}
It suffices to show that the induced morphism
\[
\Hom_{\sT(X)}(M(V)(d)[n],f^*\cF)
\to
\Hom_{\sT(X)}(M(V)(d)[n],j_*j^*f^*\cF)
\]
is an isomorphism for all $V\in \lSm/X$, $\cF\in \sT(S)$, and $d,n\in \Z$.
By adjunction, it suffices to show that the induced morphism
\[
M_S(V\times_X (X-\partial_S X))\to M_S(V)
\]
in $\sT(S)$ is an isomorphism.
To conclude, observe that the open immersion $V\times_X (X-\partial_S X)\to V$ is in $\ver_S$ by \cite[Propositions 2.3.8(6), 2.3.9]{logA1}.
\end{proof}

\section{Gysin isomorphisms}

Recall that we fix a sheaf
\[
\sT\in \Sh_{\dNis}(\lSch/B,\CAlg(\PrL))
\simeq
\Sh_{\dNis}(\lSpc/B,\CAlg(\PrL))
\]
satisfying the conditions in the introduction.

The purpose of this section is to show that there is a natural isomorphism
\[
M_{\cS}(\cX/\cX-\cZ)
\simeq
M_{\cS}(\Normal_{\cZ} \cX/\Normal_{\cZ} \cX-\cZ)
\]
in $\sT(\cS)$
for every closed immersion $a\colon \cZ\to \cX$ in $\lSmSpc/\cS$ with $\cS\in \lSpc/B$.
The outline of the proof is as follows.
We begin with the deformation to the normal cone construction in \cite{divspc}.
Then we reduce to the case when $\cS=h_S$ for some $S\in \lFan/B$ and $\cZ\to \cX$ is equal to $h_i$ for some strict closed immersion $i\colon Z\to X$ in $\lSm/S$.
To reduce to the known case \cite[Theorem 7.5.4]{logDM}, we employ a log smooth morphism $g\colon S'\to S$ in \S \ref{conserv} such that $g^*$ is conservative and the projection $X\times_S S'\to S'$ is saturated.
To construct such a morphism $g$, we need the content of \S \ref{localization}, where we investigate which morphisms are inverted when we invert dividing covers and $\ver$ simultaneously.
This investigation requires the notion of pans in \S \ref{pan}.
As inverting log blow-ups produces a divided log space from an fs log scheme, inverting subdivisions produces a pan from a fan.

\subsection{Pans}
\label{pan}

A fan chart of an fs log scheme is useful for understanding the local behavior.
To obtain an analogous notion of a fan chart for divided log spaces, we first need to invert all subdivisions in the category of fans.
Then only the support and the lattice matter.
This reasoning leads to the following notion.

\begin{definition}
A \emph{pan $\cA$} is a lattice $L_\cA$ together with a subset $\cA$ of $L_{\cA,\R}:=L_\cA\otimes \R$ such that $\cA$ is the support $\lvert \Sigma \rvert$ of a fan $\Sigma$ in $L_{\cA}$.
Here, the support $\lvert \Sigma \rvert$ is defined to be the set of points $x\in L_{\cA,\R}$ such that $rx$ is in a cone of $\Sigma$ for some real number $r>0$.
In this case, we say that the \emph{underlying pan of $\Sigma$} is $\cA$.

A \emph{morphism of pans $f\colon \cB\to \cA$} is a homomorphism of lattices $L_f\colon L_\cB\to L_\cA$ such that $L_{f,\R}:=L_f\otimes \R$ maps $\cB$ into $\cA$.
\end{definition}

Let $\cA_{Fan}$ be the category of fans $\Sigma$ in $L_\cA$ such that the underlying pan of $\Sigma$ is $\cA$.
If $\Sigma,\Sigma'\in \cA_{Fan}$, then the fan
\[
\Sigma'':=\{\sigma\cap \sigma' : \sigma\in \Sigma,\sigma'\in \Sigma'\}
\]
is an object of $\cA_{Fan}$.
This shows that $\cA_{Fan}$ is filtered.
We set
\[
\T(\cA):=\T_{\cA}
:=
\limit_{\Sigma\in \cA_{Fan}}h_{\T_{\Sigma}}
\in \Sh_{\dNis}(\lFan/\Spec(\Z)).
\]
If $\Sigma'\to \Sigma$ is a morphism in $\cA_{Fan}$, then $\T_{\Sigma'}\to \T_{\Sigma}$ is a dividing cover.
Hence we have an isomorphism $\T_{\cA}\simeq h_{\T_{\Sigma}}$ whenever $\Sigma\in \cA_{Fan}$.
In particular, $\T_{\cA}\in \lSpc/\Spec(\Z)$.
We set $\T_{\cA,B}:=\T_{\cA}\times B\in \lSpc/B$.

Let $\Fan$ denote the category of fans, and let $\Pan$ denote the category of pans.
We have the functor
\[
\lvert - \rvert\colon \Fan\to \Pan
\]
sending $\Sigma$ to $\lvert \Sigma \rvert$.
Furthermore, we have the functor
\[
\T\colon \Pan\to \lSpc/\Spec(\Z)
\]
sending $\cA$ to $\T_\cA$.
Observe that the triangle
\[
\begin{tikzcd}
\Fan\ar[rd,"h\circ\T"']\ar[rr,"{\lvert - \rvert}"]&
&
\Pan\ar[ld,"\T"]
\\
&
\lSpc/\Spec(\Z)
\end{tikzcd}
\]
commutes.

\begin{definition}
Let $f\colon \cB\to \cA$ be a morphism of pans.
We say that $\cB$ is a \emph{subpan of $\cA$} if $L_f\colon L_\cB\to L_\cA$ is an isomorphism.
We say that $f$ is an \emph{isogeny} if $f$ is surjective and the kernel of $L_f$ is a finite group.
\end{definition}

\begin{proposition}
\label{pan.1}
Let $\cB_1,\ldots,\cB_n$ be subpans of a pan $\cA$.
Then there exists a fan $\Sigma$ such that the underlying pan of $\Sigma$ is $\cA$ and $\Sigma\cap \cB_i$ is a subfan of $\Sigma$ for all $i=1,\ldots,n$.
\end{proposition}
\begin{proof}
Choose fans $\Sigma' \in \cA_{Fan}$ and $\Delta_i'\in (\cB_i)_{Fan}$ for all $i=1,\ldots,n$.
According to \cite[p.\ 18]{TOda}, there exists a fan $\Delta_i$ such that it contains $\Delta_i'$ as a subfan and $\lvert \Delta_i\rvert=L_\cA$.
We set
\[
\Sigma
:=
\{
\sigma'\cap \delta_1\cap \cdots \cap \delta_n
:
\sigma'\in \Sigma',\delta_1\in \Delta_1,\ldots,\delta_n\in \Delta_n
\}.
\]
This fan satisfies the condition.
\end{proof}

\begin{proposition}
\label{pan.2}
Let $\cB$ and $\cB'$ be subpans of a pan $\cA$.
Then $\cB\cup \cB'$ and $\cB\cap \cB'$ are subpans of $\cA$.
\end{proposition}
\begin{proof}
By Proposition \ref{pan.1}, there exists a fan $\Sigma$ such that its underlying pan is $\cA$ and $\Sigma\cap \cB$ and $\Sigma\cap \cB'$ are subfans of $\Sigma$.
Then $\Sigma\cap (\cB\cup \cB')$ and $\Sigma\cap \cB\cap \cB'$ are subfans of $\Sigma$, and their underlying pans are $\cB\cup \cB'$ and $\cB\cap \cB'$.
\end{proof}

\begin{definition}
\label{pan.3}
Suppose that $\cB_1,\ldots,\cB_n$ are subpans of a pan $\cA$.
We say that $\{\cB_1,\ldots,\cB_n\}$ is a \emph{closed cover} of $\cA$ if $\cB_1\cup \cdots \cup \cB_n=\cA$.
In this case, the associated \emph{\v{C}ech cube} is defined to be the functor
\[
Q\colon (\Delta^1)^n \to \Pan
\]
such that 
\[
Q(a_1,\ldots,a_n):=\cB_1^{a_1}\cap \cdots \cap \cB_n^{a_n}
\]
for $(a_1,\ldots,a_n)\in (\Delta^1)^n$ and the morphism $Q(a_1,\ldots,a_n)\to Q(b_1,\ldots,b_n)$ is the inclusion for every morphism $(a_1,\ldots,a_n)\to (b_1,\ldots,b_n)$ in $(\Delta^1)^n$, where
\[
\cB_n^{a_n}:=\left\{
\begin{array}{ll}
\cB_n & \text{if }a_n=0,
\\
\cA & \text{if }a_n=1.
\end{array}
\right.
\]
\end{definition}

For the notion of the total cofibers of cubes, we refer to \cite[\S A]{logSH}.
A cube in a stable $\infty$-category is \emph{cocartesian} if its total cofiber is isomorphic to $0$.

\begin{proposition}
\label{pan.4}
Let $\cB_1,\ldots,\cB_n$ be subpans of $\cA$ such that $\{\cB_1,\ldots,\cB_n\}$ is a closed cover of $\cA$, and let $Q$ be the associated \v{C}ech cube.
Then the cube
\[
M_{\cS}(\T(Q)\times \cS)
\]
in $\sT(\cS)$ is cocartesian for all $\cS\in \lSpc/B$.
\end{proposition}
\begin{proof}
By Proposition \ref{pan.1}, there exists a fan $\Sigma$ such that the underlying pan of $\Sigma$ is $\cA$ and $\Sigma\cap \cB_i$ is a subfan of $\Sigma$ for all $i=1,\ldots,n$.
Then $\{\T_{\Sigma\cap \cB_1},\ldots,\T_{\Sigma\cap \cB_n}\}$ is a Zariski cover of $\T_{\Sigma}$, so $\{\T_{\cB_1},\ldots,\T_{\cB_n}\}$ is a Zariski cover of $\T_{\cA}$.
We conclude by dividing Nisnevich descent.
\end{proof}

\subsection{Localizations of monoids by critical faces}
\label{localization}

We invert all morphisms belonging to $\ver$.
Since vertical boundaries are not stable under dividing covers, inverting both dividing covers and $\ver$ leads to inverting more morphisms.
For example, consider the diagonal homomorphism $\theta\colon \N\to P:=\N\oplus \N$ and the face $F:=0\oplus \N$ of $P$.
We have the following cones in $\Z^2$:
\begin{gather*}
\sigma_1:=\Cone((1,0),(1,1)),\;
\sigma_2:=\Cone((1,1),(0,1)),
\\
\sigma_3:=\Cone((1,0),(0,1)),\;
\sigma_4:=\Cone((1,0)).
\end{gather*}
Form the four fans $\Sigma_1,\ldots,\Sigma_4$ whose sets of maximal cones are $\{\sigma_1\}$, $\{\sigma_1,\sigma_2\}$, $\{\sigma_3\}$, and $\{\sigma_4\}$.
We have morphisms of fans $\Sigma_1\to \Sigma_2 \to \Sigma_3 \leftarrow \Sigma_4$ as illustrated in the following figure.
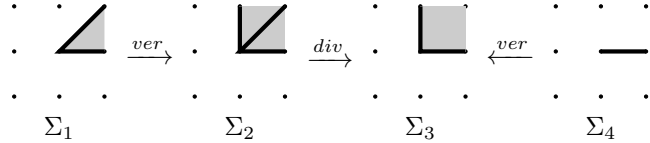
\[
\begin{tikzpicture}[yscale=0.6, xscale=0.6]
\foreach \a in {-1,0,1}
\foreach \b in {-1,0,1}
\filldraw (\a,\b) circle (1pt);
\fill[black!20] (1,1)--(0,0)--(1,0);
\draw[ultra thick] (1,1)--(0,0)--(1,0);
\node at (2,0) {$\xrightarrow{ver}$};
\node at (0,-1.7) {$\Sigma_1$};
\begin{scope}[shift={(4,0)}]
\foreach \a in {-1,0,1}
\foreach \b in {-1,0,1}
\filldraw (\a,\b) circle (1pt);
\fill[black!20] (0,1)--(0,0)--(1,0)--(1,1);
\draw[ultra thick] (1,1)--(0,0)--(1,0);
\draw[ultra thick] (0,1)--(0,0);
\end{scope}
\node at (6,0) {$\xrightarrow{div}$};
\node at (4,-1.7) {$\Sigma_2$};
\begin{scope}[shift={(8,0)}]
\foreach \a in {-1,0,1}
\foreach \b in {-1,0,1}
\filldraw (\a,\b) circle (1pt);
\fill[black!20] (1,0)--(0,0)--(0,1)--(1,1);
\draw[ultra thick] (1,0)--(0,0)--(0,1);
\end{scope}
\node at (10,0) {$\xleftarrow{ver}$};
\node at (8,-1.7) {$\Sigma_3$};
\begin{scope}[shift={(12,0)}]
\foreach \a in {-1,0,1}
\foreach \b in {-1,0,1}
\filldraw (\a,\b) circle (1pt);
\draw[ultra thick] (1,0)--(0,0);
\node at (0,-1.7) {$\Sigma_4$};
\end{scope}
\end{tikzpicture}
\]
Since $\T_{\Sigma_2}\to \T_{\Sigma_3}$ is a dividing cover and $\T_{\Sigma_1}\to \T_{\Sigma_2}$ and $\T_{\Sigma_4}\to \T_{\Sigma_3}$ are in $\ver_{\A_{\N}}$, we deduce that the morphism
\[
M_{\A_\N}(\A_{P_F})
\to
M_{\A_\N}(\A_P)
\]
induced by the open immersion $\A_{P_F}\to \A_P$ is an isomorphism.

The main result of this section is Theorem \ref{pan.7}, which provides a class of morphisms that are inverted in $\sT$ by similar ways.

For an fs monoid $P$ such that $P^\gp$ is torsion free, let $\Spec(P)$ denote the fan with the single maximal cone $P^\vee$, where $P^\vee$ denotes the dual monoid of $P$ in the dual lattice $(P^\gp)^\vee$.

\begin{definition}
\label{pan.5}
A pan $\cA$ is \emph{strongly convex} if $\cA$ is convex and $\cA\cap (-\cA)=0$.
In this case, $\cA$ is a cone, so we can discuss the faces of $\cA$.
Furthermore, we have $\cA\simeq \lvert\Spec(P)\rvert$ for some fs monoid $P$ such that $P^{\gp}$ is torsion free.
\end{definition}

\begin{lemma}
\label{pan.8}
Let $\theta\colon \cB\to \cA$ a morphism of strongly convex pans, and let $\cB'$ be the union of faces $\cV$ of $\cB$ such that $\theta$ maps any nontrivial face of $\cV$ to a nontrivial subpan of $\cA$.
We regard $\cB'$ as a subpan of $\cB$.
If $\cX\to \T_{\cA,B}$ is a morphism in $\lSpc/B$ and the projection $\cX\times_{\T_{\cA}}\T_{\cB}\to \cX$ is log smooth, then the induced morphism
\[
M_{\cX}(\cX\times_{\T_\cA}\T_{\cB'})
\to
M_{\cX}(\cX\times_{\T_\cA}\T_\cB)
\]
is an isomorphism.
\end{lemma}
\begin{proof}
The question is Zariski local on $\cX$.
Hence we may assume $\cX=h_X$ for some $X\in \lFan/B$ and $X\to \T_{\cA,B}$ is equal to $h_p$ for some morphism $p\colon X\to \A_{P,B}$, where $P$ is a sharp fs monoid and $\A_{P,B}:=\A_P\times B$.
We may also assume that $\T_\cB\to \T_\cA$ is equal to $h_{\A_\theta}\colon h_{\A_Q}\to h_{\A_P}$ for some homomorphism $\eta\colon P\to Q$ of sharp fs monoids.

Let $\Sigma$ be the subfan of $\Spec(Q)$ whose underlying pan is $\cB'$.
By \cite[Proposition 2.3.13]{logA1}, the open immersion $\T_\Sigma \to \A_Q$ is contained in $\ver_{\A_P}$.
Together with \cite[Proposition 2.3.9]{logA1}, we obtain an isomorphism
\[
(X\times_{\A_P}\A_Q)-\partial_X(X\times_{\A_P}\A_Q)
\simeq
X\times_{\A_P}\T_{\Sigma}.
\]
Hence the induced morphism
\[
M_X(X\times_{\A_P}\T_\Sigma)
\to
M_X(X\times_{\A_P}\A_Q)
\]
is an isomorphism by $\ver$-invariance, which gives the desired isomorphism.
\end{proof}

\begin{lemma}
\label{pan.6}
Suppose $\cX:=h_X$ with $X\in \lFan/B$.
Let $\cA$ and $\cB$ strongly convex pans, let $\theta\colon \cC\to \cA\times \cB$ be an isogeny of pans such that the torsion part of the kernel of $L_\cC\to L_\cA$ is invertible in $X$, and let $x_1,\ldots,x_n$ be elements of the dual lattice $L_\cC^\vee$.
We set
\[
\cH_i^{\geq 0}:=\{y\in \cC:\langle y,x_i\rangle \geq 0\},
\;
\cH_i^{\leq 0}:=\{y\in \cC:\langle y,x_i\rangle \leq 0\},
\;
\cH_i:=\cH_i^{\geq 0}\cap \cH_i^{\leq 0}
\]
for all $i=1,\ldots,n$.
We assume $\theta^{-1}(\cA\times 0)\subset \cH_i^{\geq 0}$ for all $i=1,\ldots,n$.
If $\cX\to \T_{\cA,B}$ is a morphism in $\lSpc/B$,
then the induced map
\[
M_{\cX}(\cX\times_{\T_\cA}\T_{\theta^{-1}(\cA\times 0)})\to M_{\cX}(\cX\times_{\T_\cA}\T_{\cH_1^{\geq 0}\cap \cdots \cap \cH_n^{\geq 0}})
\]
is an isomorphism.
\end{lemma}
\begin{proof}
We proceed by induction on $n$.
Observe that $\cC$ is a strongly convex pan.
Every face of $\cC$ can be written as $\theta^{-1}(\cF\times \cG)$ for some faces $\cF$ of $\cA$ and $\cG$ of $\cB$.
Furthermore, $\theta$ maps any nontrivial face of $\theta^{-1}(\cF\times \cG)$ to a nontrivial subpan of $\cA$ if and only if $\cG=0$.
Hence Lemma \ref{pan.8} shows that the induced morphism
\begin{equation}
\label{pan.6.1}
M_{\cX}(\cX\times_{\T_\cA}\T_{\theta^{-1}(\cA\times 0)})
\to
M_{\cX}(\cX\times_{\T_\cA}\T_\cC)
\end{equation}
is an isomorphism.
In particular, the claim holds for $n=0$.

Assume $n>0$.
Let $Q$ be the \v{C}ech cube $(\Delta^1)^n\to \Pan$ associated with the closed cover
\[
\{\cH_1^{\geq 0},\ldots, \cH_n^{\geq 0}\}
\]
of $\cH_1^{\geq 0}\cup \cdots \cup \cH_n^{\geq 0}$.
As a consequence of Proposition \ref{pan.4},
the cube $M_{\cX}(\cX\times_{\T_{\cA}}\T(Q))$ is cocartesian.
Furthermore, the induced morphism
\[
M_{\cX}(\cX\times_{\T_\cA}\T_{\theta^{-1}(\cA\times 0)})
\to
M_{\cX}(\cX\times_{\T_\cA}\T_{\cH_{i_1}^{\geq 0}\cap \cdots \cap \cH_{i_r}^{\geq 0}})
\]
is an isomorphism for every subset $I=\{i_1,\ldots,i_r\}\subsetneq \{1,\ldots,n\}$ by induction.
Hence to show that the induced morphism
\[
M_{\cX}(\cX\times_{\T_\cA}\T_{\theta^{-1}(\cA\times 0)})\to M_{\cX}(\cX\times_{\T_\cA}\T_{\cH_1^{\geq 0}\cap \cdots \cap \cH_n^{\geq 0}})=M_{\cX}(\cX\times_{\T_\cA}\T_{Q(0,\ldots,0)})
\]
is an isomorphism, it suffices to show that the induced morphism
\[
M_{\cX}(\cX\times_{\T_\cA}\T_{\theta^{-1}(\cA\times 0)})\to M_{\cX}(\cX\times_{\T_\cA}\T_{\cH_1^{\geq 0}\cup \cdots \cup \cH_n^{\geq 0}})=M_{\cX}(\cX\times_{\T_\cA}\T_{Q(1,\ldots,1)})
\]
is an isomorphism.

Every face of the convex pan $\cH^{\leq 0}:=\cH_1^{\leq 0}\cap \cdots \cap \cH_n^{\leq 0}$ not contained in
\[
\cH:=\bigcup_{i=1}^n
(\cH_1^{\leq 0}\cap \cdots \cap \cH_n^{\leq 0}\cap \cH_i)
\]
is of the form $\theta^{-1}(\cF\times \cG)\cap \cH^{\leq 0}$ for some faces $\cF$ of $\cA$ and $\cG$ of $\cB$.
By the assumption $\theta^{-1}(\cA\times 0)\subset \cH_i^{\geq 0}$ for all $i$, we see that $\cG\neq 0$ in this case.
Furthermore, $\theta$ maps $\theta^{-1}(0\times \cG)\cap \cH^{\leq 0}$ to the trivial subpan of $\cA$.
Hence Lemma \ref{pan.8} shows that the induced morphism
\[
M_{\cX}(\cX\times_{\T_\cA}\T_\cH)\to M_{\cX}(\cX\times_{\T_\cA}\T_{\cH^{\leq 0}})
\]
is an isomorphism.
Since
\[
\cH=\cH^{\leq 0}\cap (\cH_1^{\geq 0}\cup \cdots \cup \cH_n^{\geq 0})
\text{ and }
\cC=\cH^{\leq 0}\cup (\cH_1^{\geq 0}\cup \cdots \cup \cH_n^{\geq 0}),
\]
the induced morphism
\[
M_{\cX}(\cX\times_{\T_\cA}\T_{\cH_1^{\geq 0}\cup \cdots \cup \cH_n^{\geq 0}})
\to
M_{\cX}(\cX\times_{\T_\cA}\T_\cC)
\]
is an isomorphism by Proposition \ref{pan.4}.
Combine with \eqref{pan.6.1} to finish the proof.
\end{proof}

We refer to \cite[Definition I.4.2.10]{Ogu} for the notion of $\theta$-critical faces when $\theta$ is a homomorphism of monoids.

\begin{theorem}
\label{pan.7}
Let $X\to \A_{P,B}:=\A_P\times B$ be a morphism in $\lSch/B$, where $P$ is a sharp fs monoid.
If $\theta\colon P\to Q$ is an injective, local, and locally exact homomorphism of fs monoids such that the torsion part of $\theta^\gp$ is invertible in $X$, then the natural transformation
\[
f_\sharp j_\sharp j^* f^*\xrightarrow{ad'} f_\sharp f^*
\]
is an equivalence for all $\theta$-critical face $G$ of $Q$, where $f\colon X\times_{\A_P}\A_Q\to X$ is the projection, and $j\colon X\times_{\A_P} \A_{Q_G}\to X\times_{\A_P}\A_Q$ is the obvious open immersion.
\end{theorem}
\begin{proof}
Observe that $f$ is log smooth by \cite[Theorem IV.3.1.8]{Ogu}.
We need to show that the induced morphism
\begin{equation}
\label{pan.7.1}
M_X(X'\times_{\A_P}\A_{Q_G})\to M_X(X'\times_{\A_P} \A_Q)
\end{equation}
is an isomorphism for all $X'\in \lSm/X$.

Assume first that $G$ is a maximal $\theta$-critical face of $Q$.
Let $\cA$, $\cB$, and $\cD$ be the underlying pans of $\Spec(P)$, $\Spec(G)$, and $\Spec(Q)$.
The summation homomorphism $P\oplus G\to Q$ gives a morphism of pans $\alpha\colon \cD\to \cA\times \cB$.
We set $\cC:=L_\cD\times_{L_{\cA\times \cB}} (\cA\times \cB)$, which contains $\cD$ as a subpan.
By \cite[Theorem I.4.7.7(2)]{Ogu}, the induced homomorphism $P\to Q/G$ is exact. 
The implication (4)$\Rightarrow$(3) in \cite[Lemma I.4.7.9]{Ogu} shows that the induced morphism $\cC\to \cA\times \cB$ is an isogeny.
Choose elements $x_1,\ldots,x_n\in Q-P$ such that $P\cup \{x_1,\ldots,x_n\}$ generates $Q$.
Then we have an equality
\[
\cD
=
\{
y\in \cC : \langle y,x_i\rangle \geq 0 \;\;\forall i=1,\ldots,n\}.
\]
Furthermore, $\alpha^{-1}(\cA\times 0)$ is the underlying pan of $\Spec(Q_G)$.
Use Lemma \ref{pan.6} to see that \eqref{pan.7.1} is an isomorphism.

For general $G$, let $G'$ be a maximal $\theta$-critical face of $Q$ containing $G$.
The induced morphisms
\[
M_X(X'\times_{\A_P} \A_{Q_{G'}})
\to
M_X(X'\times_{\A_P} \A_Q)
\text{ and }
M_X(X'\times_{\A_P} \A_{Q_{G'}})
\to
M_X(X'\times_{\A_P} \A_{Q_G})
\]
are isomorphisms by what we have proven for maximal $\theta$-critical faces.
We conclude by 2-out-of-3.
\end{proof}

\subsection{Saturations of morphisms}
\label{conserv}

We refer to \cite[Definitions I.3.12, II.2.10]{zbMATH07027475} and \cite[Definitions I.4.8.2, III.2.5.1]{Ogu} for the notions of saturated homomorphisms of integral monoids and saturated morphisms of integral log schemes.
The purpose of this subsection is to show that we can make a morphism of fs log schemes into a saturated morphism after a base change along a suitable log smooth morphism whose pullback functor of motives is conservative.

\begin{lemma}
\label{conserv.4}
Let $F\colon \cC \to \cD$ and $G\colon \cD\to \cE$ be functors of categories.
Assume that for all $X,Y\in \cC$, the function
\[
\Hom_{\cD}(FX,FY)\rightarrow \Hom_{\cE}(GFX,GFY)
\]
given by $f\mapsto Gf$ is bijective.
If $F$ is conservative, then $GF$ is also conservative.
\end{lemma}
\begin{proof}
Let $f\colon X\to Y$ be a morphism in $\cC$ such that $GFf$ is an isomorphism.
There exists a unique morphism $g\colon FY\to FX$ such that $Gg$ is an inverse of $GFf$ by the bijectivity.
Since $G(Ff \circ g)=\id$ and $G(g \circ Ff)=\id$, we see that $g$ is an inverse of $Ff$ by the bijectivity.
The conservativity of $F$ implies that $f$ is an isomorphism.
\end{proof}

\begin{lemma}
\label{conserv.2}
For homomorphisms of saturated monoids $P\to Q,Q'$, there is a canonical isomorphism
\[
\ol{Q\oplus_P Q'}
\simeq
\ol{\ol{Q}\oplus_{\ol{P}}\ol{Q'}}.
\]
\end{lemma}
\begin{proof}
Let $\Mon$ denote the category of saturated monoids.
For a saturated sharp monoid $M$, we have isomorphisms
\begin{align*}
\Hom_{\Mon}(Q\oplus_P Q',M)
\simeq &
\Hom_{\Mon}(Q,M)\times_{\Hom_{\Mon}(P,M)} \Hom_{\Mon}(Q',M)
\\
\simeq &
\Hom_{\Mon}(\ol{Q},M)\times_{\Hom_{\Mon}(\ol{P},M)}\Hom_{\Mon}(\ol{Q'},M)
\\
\simeq &
\Hom_{\Mon}(\ol{Q}\oplus_{\ol{P}}\ol{Q},M).
\end{align*}
We conclude together with the fact that the quotient homomorphism $(-)\to \ol{(-)}$ is the universal homomorphism to a sharp saturated monoid.
\end{proof}

\begin{proposition}
\label{conserv.5}
Let $f\colon X\to S$ be an exact morphism of fs log schemes.
Then there exists a Kummer log smooth morphism $g\colon S'\to S$ such that $u^*$ is conservative for every pullback $u\colon Y\times_S S'\to Y$ of $g$ and the projection $f'\colon X\times_S S'\to S'$ is saturated.
\end{proposition}
\begin{proof}
The question is Zariski local on $S$ by Proposition \ref{extension.3}.
Hence we may assume that $S$ admits a neat chart $P$.
The question is also Zariski local on $X$.
We may also assume that there exists a chart $\theta\colon P\to Q$ of $f$ such that $\theta$ is locally exact.
By \cite[Theorem I.4.9.1]{Ogu}, there exists an integer $n>0$ such that the pushout of $\theta$ along $P\xrightarrow{\cdot n} P$ in the category of saturated monoids is saturated.
Consider the injective and local homomorphisms
\[
P\to P':=P^\gp \oplus P
\text{ and }
\eta\colon P\to P'':=P \oplus P
\]
sending $p\in P$ to $(p,np)$.
Use these homomorphisms to define
\[
Y':=Y\times_{\A_P}\A_{P'}
\text{ and }
Y'':=Y\times_{\A_P}\A_{P''},
\]
where $Y\to S$ is an arbitrary morphism of fs log schemes.
We have a commutative diagram
\[
\begin{tikzcd}
Y\ar[rd,"\id"']\ar[r,"s"]&
Y''\ar[d,"v"]\ar[r,leftarrow,"j"]&
Y',\ar[ld,"u"]
\\
&
Y
\end{tikzcd}
\]
where $u$ and $v$ are the projections, and $s$ (resp.\ $j$) is induced by the first projection $P''\to P$ (resp.\ the inclusion $P''\to P'$).
Let $g\colon S':=S\times_{\A_P}\A_{P'}\to S$ be the projection.
Observe that $u$ is a pullback of $g$.

Since $s^*v^*\simeq \id$, $v^*$ is conservative.
The only maximal $\eta$-critical faces of $P''$ are $P\oplus 0$ and $0\oplus P$.
By the implication (1)$\Leftrightarrow$(5) in \cite[Theorem I.4.7.7]{Ogu}, $\eta$ is locally exact.
Apply Theorem \ref{pan.7} to $\eta$ and its critical face $P\oplus 0$ to obtain an isomorphism $u_\sharp u^*\simeq v_\sharp v^*$.
Together with Lemma \ref{conserv.4}, we see that $u^*$ is conservative.

Combine \cite[Proposition I.4.8.5(4)]{Ogu} and Lemma \ref{conserv.2} to show that the pushout of $\theta$ along $P\to P'$ is saturated too.
Use \cite[Lemma II.3.5]{zbMATH07027475} to deduce that the projection $f'$ is saturated.
The cokernel of $P^\gp \to P'^\gp$ is isomorphic to $P^\gp$, which is torsion free since $P$ is sharp.
Together with \cite[Theorem IV.3.1.8]{Ogu}, we see that $g$ is Kummer log smooth.
\end{proof}

\begin{proposition}
\label{conserv.3}
Let $f\colon X\to S$ be a morphism of fs log schemes.
Then there exists a log smooth morphism $g\colon S'\to S$ such that $u^*$ is conservative for every pullback $u\colon Y\times_S S'\to Y$ of $g$ and the projection $f'\colon X\times_S S'\to S'$ is saturated.
\end{proposition}
\begin{proof}
By \cite[Proposition 4.2.1]{logA1}, there exists a dividing Zariski cover $p\colon U\to S$ such that the projection $X\times_S U\to U$ is exact.
Proposition \ref{extension.3} shows that $p^*$ is conservative.
Apply Proposition \ref{conserv.5} to this projection to conclude.
\end{proof}

For a morphism of fs log schemes $f\colon X\to S$,
we set
\[
\cM_{X/S}
:=
\mathrm{coker}(\cM_{\ul{X}\times_{\ul{S}}S}\to \cM_X),
\]
see \cite[Definition II.1.1.10(2)]{Ogu}.

\begin{remark}
\label{conserv.6}
The projection $f'$ in Propositions \ref{conserv.5} and \ref{conserv.3} is saturated, and this implies that $\cM_{X\times_S S'/S'}^\gp$ is torsion free by \cite[Theorem III.2.5.5]{Ogu}.
\end{remark}

\subsection{Deformation to the normal cone}
\label{deform}
Throughout this subsection, we fix $\cS\in \lSpc/B$ and a closed immersion $a\colon \cZ\to \cX$ in $\lSmSpc/S$.
Let $p\colon \cX\to \cS$ be the structure morphism.
From \cite[Definition 9.13]{divspc}, we have the commutative diagram of the deformation to the normal cone
\begin{equation}
\label{pur.2.2}
\begin{tikzcd}
\cZ\ar[r,"a"]\ar[d,"i_1'"']&
\cX\ar[r,"p"]\ar[d,"l_1"]&
\cS\ar[d,"i_1"]
\\
\cZ\times \boxx\ar[r,"a_d"]&
\Deform_{\cZ} \cX\ar[r,"p_d"]&
\cS\times \boxx\ar[r,"e"]&
\cS
\\
\cZ\ar[r,"a_n"]\ar[u,"i_0'"]&
\Normal_{\cZ} \cX\ar[r,"p_n"]\ar[u,"l_0"']&
\cS,\ar[u,"i_0"']
\end{tikzcd}
\end{equation}
where $e$ is the projection, $i_0$ is the $0$-section, $i_1$ is the $1$-section.

\begin{theorem}
\label{pur.2}
With the above notation, the natural morphisms
\[
M_{\cS}(\cX/\cX-\cZ)
\to
M_{\cS}(\Deform_{\cZ}\cX/\Deform_{\cZ}\cX-\cZ\times \boxx)
\leftarrow
M_{\cS}(\Normal_{\cZ}\cX/\Normal_{\cZ} \cX-\cZ)
\]
are isomorphisms.
\end{theorem}
\begin{proof}
The question is Zariski local on $\cS$ and $\cX$.
Hence as in the proof of \cite[Theorem 9.12]{divspc}, we may assume that $\cS=h_S$ for some $S\in \lFan/B$ and $\cZ\to \cX$ is equal to $h_i$ for some strict closed immersion $i\colon Z\to X$ in $\lSm/S$.
Then it suffices to show that the induced morphisms
\begin{equation}
\label{pur.2.1}
M_{S}(X/X-Z)
\to
M_{S}(\Deform_{Z} X/\Deform_{Z} X-Z\times \boxx)
\leftarrow
M_S(\Normal_{Z}X/\Normal_{Z} X-Z)
\end{equation}
are isomorphisms.

By Proposition \ref{conserv.3} and Remark \ref{conserv.6}, there exists a log smooth morphism $g\colon S'\to S$ such that $g^*$ is conservative and $\cM_{X'/S'}^\gp$ is torsion free, where $X':=X\times_S S'$.
If we apply $g^*$ to \eqref{pur.2.1}, then we get
\[
M_{S'}(X'/X'-Z')
\to
M_{S'}(\Deform_{Z'} X'/\Deform_{Z'}X'-Z'\times \boxx)
\leftarrow
M_{S'}(\Normal_{Z'} X'/\Normal_{Z'} X'-Z')
\]
by \cite[Lemma D.6]{divspc}.
Hence we can replace $S$ by $S'$ to assume that $\cM_{X/S}^\gp$ is torsion free.

Then \cite[Proposition C.1]{divspc} allows us to assume that there exists $Y\in \lSm/S$ and a cartesian square
\[
\begin{tikzcd}
Z\ar[d]\ar[r]&
Y\ar[d,"i_0"]
\\
X\ar[r,"u"]&
Y\times \A^s
\end{tikzcd}
\]
such that $i_0$ is the zero section and $u$ is strict \'etale.
Observe that both $X$ and $Z$ are strict smooth over $Y$.
We can replace $S$ by $Y$ to assume that $X$ and $Z$ are strict smooth over $S$.

Let $p\colon S\to \ul{S}$ be the canonical morphism.
Combine \cite[Theorem 7.5.4]{logDM} and $\ver$-invariance to obtain natural isomorphisms
\[
M_{\ul{S}}(\ul{X}/\ul{X}-\ul{Z})
\xrightarrow{\simeq}
M_{\ul{S}}(\Deform_{\ul{Z}} \ul{X}/\Deform_{\ul{Z}} \ul{X}-\ul{Z}\times \boxx)
\xleftarrow{\simeq}
M_{\ul{S}}(\Normal_{\ul{Z}} \ul{X}/\Normal_{\ul{Z}} \ul{X}-\ul{Z}).
\]
Apply the functor $p^*$ to these to conclude.
\end{proof}

\begin{definition}
With the above notation, the \emph{Gysin isomorphism} is defined to be the natural isomorphism
\[
M_{\cS}(\cX/\cX-\cZ)
\simeq
M_{\cS}(\Normal_{\cZ}\cX/\Normal_{\cZ} \cX-\cZ)
\]
obtained by Theorem \ref{pur.2}.
\end{definition}

\section{Thom transformations}

Recall that we fix a sheaf
\[
\sT\in \Sh_{\dNis}(\lSch/B,\CAlg(\PrL))
\simeq
\Sh_{\dNis}(\lSpc/B,\CAlg(\PrL))
\]
satisfying the conditions in the introduction.
In \S \ref{applications}, we define the Thom transformation $\Thom$ and its variants $\Thom_d$ and $\Thom_n$.
We use the Gysin isomorphisms to show $\Thom\xrightarrow{\simeq} \Thom_d \xleftarrow{\simeq} \Thom_n$.
In \S \ref{exchange} and \S \ref{composition}, we study various functorial properties of Thom transformations as in \cite[\S 1.5.2]{Ayo071}.

\subsection{Application of the Gysin isomorphisms}
\label{applications}
Throughout this subsection, $\cS\in \lSpc/B$, $p\colon \cX\to \cS$ is a log smooth morphism in $\lSpc/B$, and $a\colon \cS\to \cX$ is its section.
We assume that $a$ is a closed immersion.
From \eqref{pur.2.2}, we have a commutative diagram
\[
\begin{tikzcd}
\cS\ar[r,"a"]\ar[d,"i_1"']&
\cX\ar[r,"p"]\ar[d,"l_1"]&
\cS\ar[d,"i_1"]
\\
\cS\times \boxx\ar[r,"a_d"]&
\Deform_{\cS} \cX\ar[r,"p_d"]&
\cS\times \boxx\ar[r,"e"]&
\cS
\\
\cS\ar[r,"a_n"]\ar[u,"i_0"]&
\Normal_{\cS} \cX\ar[r,"p_n"]\ar[u,"l_0"']&
\cS,\ar[u,"i_0"']
\end{tikzcd}
\]
where $e$ is the projection, $i_0$ is the $0$-section, and $i_1$ is the $1$-section.

\begin{definition}
\label{thom.1}
We set
\begin{gather*}
\Thom(p,a)
:=
p_\sharp a_*,
\;
\Thom_d(p,a):=(ep_d)_\sharp a_{d*}e^*,
\;
\Thom_n(p,a):=\Thom(p_n,a_n),
\\
\Thom'(p,a)
:=
a^!p^*,
\;
\Thom_d'(p,a):=
e_*a_d^!(ep_d)^*,
\;
\Thom_n'(p,a)
:=
\Thom'(p_n,a_n),
\end{gather*}
where $a^!$, $a_d^!$, and $a_n^!$ are right adjoints of $a_*$, $a_{d*}$, and $a_{n*}$, which exist by Proposition \ref{fundamental.3}.
Let $T_d\colon \Thom_d'(p,a)\to \Thom'(p,a)$ be the natural transformation given by the composition
\[
e_* a_d^!(ep_d)^*
\xrightarrow{ad}
e_*i_{1*}i_1^*a_d^!(ep_d)^*
\xrightarrow{Ex}
e_*i_{1*}a^!l_1^*(ep_d)^*
\xrightarrow{\simeq}
a^!p^*.
\]
We define $T_n\colon \Thom_d'(p,a)\to \Thom_n'(p,a)$ similarly.
By adjunction, we obtain the natural transformations
\[
T_d\colon \Thom(p,a)\to \Thom_d(p,a)
\text{ and }
T_n \colon \Thom_n(p,a)\to \Thom_d(p,a).
\]
\end{definition}

\begin{proposition}
\label{thom.6}
The natural transformations $T_d$ and $T_n$ are isomorphisms.
\end{proposition}
\begin{proof}
We focus on $T_d$ since the proof for $T_n$ is similar.
We only need to show that the induced morphism
\[
\Hom_{\sT(\cS)}(M_{\cS}(\cS')(d)[n],\Thom_d'(p,a)\cF)
\to
\Hom_{\sT(\cS)}(M_{\cS}(\cS')(d)[n],\Thom'(p,a)\cF)
\]
is an isomorphism for all $\cF\in \sT(\cS)$, $\cS'\in \lSmSpc/\cS$, and $d,n\in \Z$.
By adjunction and the localization property, it suffices to show that the induced morphism
\begin{align*}
& \Hom_{\sT(\cS)}(M_{\cS}(\Deform_{\cS'}\cX'/\Deform_{\cS'}\cX'-\cS'\times \boxx)(d)[n],\cF)
\\
\to &
\Hom_{\sT(\cS)}(M_{\cS}(\cX'/\cX'-\cS')(d)[n],\cF)
\end{align*}
is an isomorphism, where $\cX':=\cX\times_{\cS}\cS'$.
Theorem \ref{pur.2} shows that the induced morphism
\[
M_{\cS'}(\cX'/\cX'-\cS')\to M_{\cS'}(\Deform_{\cS'}\cX'/\Deform_{\cS'}\cX'-\cS'\times \boxx)
\]
is an isomorphism.
Apply $q_\sharp$ to this to conclude, where $q\colon \cS'\to \cS$ is the structure morphism.
\end{proof}

\begin{definition}
Let $T\colon \Thom_n'(p,a)\xrightarrow{\simeq}\Thom'(p,a)$ be the composition
\[
\Thom_n'(p,a)\xrightarrow{T_n^{-1}}\Thom_d'(p,a)\xrightarrow{T_d}\Thom(p,a).
\]
By adjunction, we obtain $T\colon \Thom(p,a)\xrightarrow{\simeq} \Thom_n(p,a)$.
\end{definition}

\subsection{Exchange transformations}
\label{exchange}

Throughout this subsection, we fix a commutative diagram with cartesian squares
\[
\begin{tikzcd}
\cS'\ar[d,"u"']\ar[r,"a'"]&
\cX'\ar[d,"v"]\ar[r,"p'"]&
\cS'\ar[d,"u"]
\\
\cS\ar[r,"a"]&
\cX\ar[r,"p"]&
\cS
\end{tikzcd}
\]
in $\lSpc/B$ such that $p$ and $p'$ are log smooth, $a$ and $a'$ are closed immersions, $pa=\id$, and $p'a'=\id$.

\begin{construction}
\label{thom.8}
We have the natural transformation
\begin{equation}
\label{thom.8.1}
Ex\colon u^*\Thom'(p,a)\to \Thom'(p',a')u^*
\end{equation}
given by the composition
\[
u^*a^!p^*
\xrightarrow{Ex}
a'^!v^*p^*
\xrightarrow{\simeq}
a'^!p'^*u^*,
\]
where the first arrow is defined by \textup{Proposition \ref{loc.2}}.
Observe that if $u\in \lSmSpc$, then the first arrow is an isomorphism, and hence \eqref{thom.8.1} is an isomorphism.

If $u\in \lSmSpc$, then a left adjoint of \eqref{thom.8.1} is
\begin{equation}
\label{thom.8.2}
Ex\colon u_\sharp \Thom(p',a')
\to
\Thom(p,a)u_\sharp.
\end{equation}
\end{construction}

\begin{construction}
\label{thom.13}
We have the natural isomorphism
\begin{equation}
\label{thom.13.1}
Ex\colon u_*\Thom'(p',a')\xrightarrow{\simeq} \Thom'(p,a)u_*
\end{equation}
given by the composition
\[
u_*a'^!p'^*
\xrightarrow{Ex}
a^!v_*p'^*
\xrightarrow{Ex^{-1}}
a^!p^*u_*,
\]
where the first arrow is an isomorphism by Proposition \ref{loc.2}, and the second arrow is defined and an isomorphism by log smooth base change.

A left adjoint of \eqref{thom.13.1} is
\begin{equation}
\label{thom.13.2}
Ex\colon u^*\Thom(p,a)
\xrightarrow{\simeq}
\Thom(p',a') u^*.
\end{equation}
\end{construction}

\begin{construction}
\label{thom.14}
We have the natural transformation
\begin{equation}
\label{thom.14.2}
Ex\colon \Thom(p,a)u_*
\to
u_*\Thom(p',a')
\end{equation}
given by the composition
\[
p_\sharp a_*u_*
\xrightarrow{\simeq}
p_\sharp v_*a_*'
\xrightarrow{Ex}
u_*p_\sharp'.
\]
\end{construction}

\begin{proposition}
\label{thom.9}
The functor $\Thom_n(p,a)$ is an equivalence of $\infty$-categories.
\end{proposition}
\begin{proof}
We need to show that the natural transformations
\[
\id \xrightarrow{ad} \Thom'(p_n,a_n)\Thom(p_n,a_n)
\text{ and }
\Thom(p_n,a_n)\Thom'(p_n,a_n)\xrightarrow{ad'} \id
\]
are isomorphisms.
By \cite[Proposition 9.19]{divspc}, we have a cartesian square
\[
\begin{tikzcd}
h_{S\times \A^n}\ar[d]\ar[r,"h_q"]&
h_S\ar[d,"g"]
\\
\Normal_{\cS}\cX\ar[r]&
\cS
\end{tikzcd}
\]
with $S\in \lFan/B$ and $n\in \N$ such that $g$ is a Zariski cover and $q$ is the projection.
Let $b\colon S\to S\times \A^n$ be the $0$-section.

By Proposition \ref{extension.3} and Constructions \ref{thom.8} and \ref{thom.13}, we reduce to showing that the natural transformations
\[
g^*\xrightarrow{ad} \Thom'(h_q,h_b)\Thom(h_q,h_b)g^*
\text{ and }
\Thom(h_q,h_b)\Thom'(h_q,h_b)g^*\xrightarrow{ad'} g^*
\]
are isomorphisms.
Hence it suffices to show that $\Thom(h_q,h_b)$ is an equivalence of categories.
This follows from Proposition \ref{fundamental.5}.
\end{proof}

\begin{theorem}
\label{thom.10}
The functor $\Thom(p,a)$ is an equivalence of $\infty$-categories.
\end{theorem}
\begin{proof}
Combine Propositions \ref{thom.6} and \ref{thom.9}.
\end{proof}

Since $\Thom'(p,a)$ is a right adjoint of $\Thom(p,a)$, Theorem \ref{thom.10} shows that $\Thom'(p,a)$ is an inverse of $\Thom(p,a)$.

\begin{proposition}
\label{thom.12}
The square
\begin{equation}
\label{thom.12.1}
\begin{tikzcd}
u^*\ar[r,"ad"]\ar[d,"ad"']&
u^*\Thom'(p,a)\Thom(p,a)\ar[d,"Ex"]
\\
\Thom'(p',a')\Thom(p,a)u^*&
\Thom'(p',a')u^*\Thom(p,a)\ar[l,"Ex"']
\end{tikzcd}
\end{equation}
commutes.
\end{proposition}
\begin{proof}
The diagram
\[
\begin{tikzcd}[row sep=small]
u^*\ar[r,"ad"]\ar[rd,"ad"']\ar[dd,"ad"']&
u^*a^!p^*p_\sharp a_*\ar[r,"Ex"]&
a^!v^*p^*p_\sharp a_*\ar[r,"\simeq"]&
a'^!p'^*u^*p_\sharp a_*\ar[d,"Ex^{-1}"]
\\
&
u^*a^!a_*\ar[u,"ad"']\ar[r,"Ex"]&
a'^!v^*a_*\ar[u,"ad"']\ar[lld,"Ex"]\ar[r,"ad"]&
a'^!p'^*p_\sharp'v^*a_*\ar[d,"Ex"]
\\
a'^!a_*'u^*\ar[rrr,"ad"]&
&
&
a'^!p'^*p_\sharp'a_*'u^*
\end{tikzcd}
\]
commutes, which shows the claim.
\end{proof}

\begin{proposition}
\label{thom.7}
The natural transformation \eqref{thom.8.1} is an isomorphism.
If $u\in \lSmSpc$, then the natural transformation \eqref{thom.8.2} is an isomorphism.
\end{proposition}
\begin{proof}
The left vertical and upper horizontal arrows of \eqref{thom.12.1} are isomorphisms by Theorem \ref{thom.10}.
The lower horizontal arrow of \eqref{thom.12.1} is an isomorphism, see \eqref{thom.13.2}.
Hence the right vertical arrow of \eqref{thom.12.1} is an isomorphism.
Theorem \ref{thom.10} finishes the proof.
\end{proof}

\begin{proposition}
\label{thom.11}
The square
\begin{equation}
\label{thom.11.1}
\begin{tikzcd}
\Thom(p,a)u_*\Thom'(p',a')\ar[r,"Ex"]\ar[d,"Ex"']&
\Thom(p,a)\Thom'(p,a)u_*\ar[d,"ad'"]
\\
u_*\Thom(p',a')\Thom'(p',a')\ar[r,"ad'"]&
u_*
\end{tikzcd}
\end{equation}
commutes.
\end{proposition}
\begin{proof}
The diagram
\[
\begin{tikzcd}[row sep=small]
p_\sharp a_* u_* a'^!p'^*\ar[r,"Ex"]\ar[d,"\simeq"']&
p_\sharp a_* a^! v_*p'^*\ar[r,"Ex^{-1}"]\ar[d,"ad'"]&
p_\sharp a_* a^! p^* u_*\ar[d,"ad'"]
\\
p_\sharp v_*a_*'a'^!p'^*\ar[r,"ad'"]\ar[d,"Ex"']&
p_\sharp v_* p'^*\ar[r,"Ex^{-1}"]\ar[d,"Ex"]&
p_\sharp p^* u_*\ar[d,"ad'"]
\\
u_*p_\sharp'a_*'a'^!p'^*\ar[r,"ad'"]&
u_*p_\sharp'p'^*\ar[r,"ad'"]&
u_*
\end{tikzcd}
\]
commutes, which shows the claim.
\end{proof}

\begin{proposition}
\label{thom.18}
The natural transformation \eqref{thom.14.2} is an isomorphism.
\end{proposition}
\begin{proof}
The right vertical and lower horizontal arrows of \eqref{thom.11.1} are isomorphisms by Theorem \ref{thom.10}.
The upper horizontal arrow of \eqref{thom.11.1} is an isomorphism by Proposition \ref{thom.7}.
Hence the left vertical arrow is an isomorphism.
Theorem \ref{thom.10} finishes the proof.
\end{proof}

\begin{construction}
\label{thom.19}
Consider the induced commutative diagram
\[
\begin{tikzcd}[column sep=small, row sep=small]
&
\cS'\ar[ld,"i_1'"']\ar[rr,"a'"]\ar[dd,"u",near start]&
&
\cX'\ar[ld,"l_1'"',near end]\ar[rr,"p'"]\ar[dd,"v",near start]&
&
\cS'\ar[ld,"i_1'"',near end]\ar[dd,"u",near start]
\\
\cS'\times \boxx\ar[rr,"a_d'"',crossing over,near end]\ar[dd,"u'"']&
&
\Deform_{\cS'} \cX'\ar[rr,"p_d'"',crossing over,near end]&
&
\cS' \times \boxx\ar[rr,"e'"',crossing over,near end]&
&
\cS'\ar[dd,"u"]
\\
&
\cS\ar[ld,"i_1"']\ar[rr,"a"',near start]&
&
\cX\ar[ld,"l_1"]\ar[rr,"p"',near start]&
&
\cS\ar[ld,"i_1"]
\\
\cS\times \boxx\ar[rr,"a_d"']&
&
\Deform_{\cS} \cX\ar[rr,"p_d"']\ar[uu,"v'"',crossing over,leftarrow,near start]&
&
\cS \times \boxx\ar[rr,"e"']\ar[uu,"u'"',crossing over,leftarrow,near start]&
&
\cS.
\end{tikzcd}
\]
We have the natural transformation
\[
Ex\colon u^*\Thom_d'(p,a)\to \Thom_d'(p',a')u^*
\]
given by the composition
\[
u^*e_*a_d^!(ep_d)^*
\xrightarrow{Ex}
e_*'u'^*a_d^!(ep_d)^*
\xrightarrow{Ex}
e_*'a_d'^!v'^*(ep_d)^*
\xrightarrow{\simeq}
e_*'a_d'^!(e'p_d')^*u^*.
\]
\end{construction}

\begin{proposition}
\label{thom.15}
The squares
\[
\begin{tikzcd}
u^*\Thom_d'(p,a)\ar[r,"T_d"]\ar[d,"Ex"']&
u^*\Thom'(p,a)\ar[d,"Ex"]
\\
\Thom_d'(p',a')u^*\ar[r,"T_d"]&
\Thom'(p',a')u^*
\end{tikzcd}
\quad
\begin{tikzcd}
u^*\Thom_d'(p,a)\ar[r,"T_n"]\ar[d,"Ex"']&
u^*\Thom_n'(p,a)\ar[d,"Ex"]
\\
\Thom_d'(p',a')u^*\ar[r,"T_n"]&
\Thom_n'(p',a')u^*
\end{tikzcd}
\]
commute.
\end{proposition}
\begin{proof}
Let us use the notation in Construction \ref{thom.19}.
Apply $(-)p_d^*e^*$ to the commutative diagram
\[
\begin{tikzcd}[row sep=small]
u^*e_*a_d^!\ar[r,"ad"]\ar[d,"Ex"']&
u^*e_*i_{1*}i_1^*a_d^!\ar[d,"Ex"]\ar[r,"Ex"]&
u^*e_*i_{1*}a^!l_1^*\ar[d,"Ex"]
\\
e_*'u'^*a_d^!\ar[r,"ad"]\ar[ddd,"Ex"']\ar[rdd,"ad"',bend right=20]&
e_*'u'^*i_{1*}i_1^*a_d^!\ar[d,"Ex"]\ar[r,"Ex"]&
e_*'u'^*i_{1*}a^!l_1^*\ar[d,"Ex"]
\\
&
e_*'i_{1*}'u^*i_1^*a_d^!\ar[r,"Ex"]\ar[d,"\simeq"]&
e_*'i_{1*}'u^*a^!l_1^*\ar[d,"Ex"]
\\
&
e_*'i_{1*}'i_1'^*u'^*a_d^!\ar[d,"Ex"]&
e_*'i_{1*}'a'^!v^*l_1^*\ar[d,"\simeq"]
\\
e_*'a_d'^!v'^*\ar[r,"ad"]&
e_*'i_{1*}'i_1'^*a_d'^!v'^*\ar[r,"Ex"]&
e_*'i_{1*}'a'^!l_1'^*v'^*
\end{tikzcd}
\]
to see that the left square commutes.
The right square commutes similarly.
\end{proof}

\begin{proposition}
\label{thom.16}
The square
\[
\begin{tikzcd}
u_*\Thom_n'(p',a')\ar[r,"T"]\ar[d,"Ex"']&
u_*\Thom'(p',a')\ar[d,"Ex"]
\\
\Thom_n'(p,a)u_*\ar[r,"T"]&
\Thom'(p,a)u_*
\end{tikzcd}
\]
commutes.
\end{proposition}
\begin{proof}
Consider the commutative diagram
\[
\begin{tikzcd}[column sep=tiny, row sep=tiny]
&
u^*\ar[dd,"\id"',near end]\ar[ld,"ad"']\ar[rr,"ad"]&
&
u^*\Thom'(p,a)\Thom(p,a)\ar[ld,"Ex"]\ar[dd,"\simeq"]
\\
\Thom'(p',a')\Thom(p',a')u^*\ar[dd,"\simeq"']\ar[rr,crossing over,leftarrow,"Ex",near end]&
&
\Thom'(p',a')u^*\Thom(p,a)
\\
&
u^*\ar[rr,"ad",near end]\ar[ld,"ad"']&
&
u^*\Thom_n'(p,a)\Thom_n(p,a)\ar[ld,"Ex"]
\\
\Thom_n'(p',a')\Thom_n(p',a')u^*\ar[rr,leftarrow,"Ex"]&
&
\Thom_n'(p',a')u^*\Thom_n(p,a).\ar[uu,crossing over,leftarrow,"\simeq",near end]
\end{tikzcd}
\]
All the arrows with the tag $ad$ (resp.\ $Ex$) are isomorphisms by Theorem \ref{thom.10} (resp.\ Proposition \ref{thom.7}).
The upper and lower floors commute by Proposition \ref{thom.12}.
The right side square commutes by Proposition \ref{thom.15}.
The left and back side squares commute too, so the front side square commutes.

Together with the fact that $\Thom'(p',a')$ and $\Thom_n'(p',a')$ are equivalences by Theorem \ref{thom.10}, we deduce that the square
\[
\begin{tikzcd}
\Thom(p',a')u^*\ar[r,leftarrow,"Ex"]\ar[d,"T"']&
u^*\Thom(p,a)\ar[d,"T"]
\\
\Thom_n(p',a')u^*\ar[r,leftarrow,"Ex"]&
u^*\Thom_n(p,a)
\end{tikzcd}
\]
commutes.
We conclude by adjunction.
\end{proof}

\begin{proposition}
\label{thom.17}
The square
\[
\begin{tikzcd}
\Thom(p,a)u_*\ar[d,"Ex"']\ar[r,"T"]&
\Thom_n(p,a)u_*\ar[d,"Ex"]
\\
u_*\Thom(p',a')\ar[r,"T"]&
u_*\Thom_n(p,a)
\end{tikzcd}
\]
commutes.
\end{proposition}
\begin{proof}
Consider the commutative diagram
\[
\begin{tikzcd}[column sep=tiny, row sep=tiny]
&
\Thom(p,a)u_*\Thom'(p',a')\ar[dd,"\simeq"',near end]\ar[ld,"Ex"']\ar[rr,"Ex"]&
&
\Thom(p,a)\Thom'(p,a)u_*\ar[ld,"ad'"]\ar[dd,"\simeq"]
\\
u_*\Thom(p',a')\Thom'(p',a')\ar[dd,"\simeq"']\ar[rr,crossing over,"ad'",near end]&
&
u_*
\\
&
\Thom_n(p,a)u_*\Thom_n'(p',a')\ar[rr,"Ex",near end]\ar[ld,"Ex"']&
&
\Thom_n(p,a)\Thom_n'(p,a)u_*\ar[ld,"ad'"]
\\
u_*\Thom_n(p',a')\Thom_n'(p',a')\ar[rr,"ad'"]&
&
u_*.\ar[uu,crossing over,leftarrow,"\id",near start]
\end{tikzcd}
\]
All the arrows with the tag $ad'$ (resp.\ $Ex$) are isomorphisms by Theorem \ref{thom.10} (resp.\ Proposition \ref{thom.18}).
The upper and lower side squares commute by Proposition \ref{thom.11}.
The back side square commutes by Proposition \ref{thom.16}.
The front and right side squares commute too, so the left side square commutes.

Together with the fact that $\Thom'(p',a')$ and $\Thom_n'(p',a')$ are equivalences by Theorem \ref{thom.10},
we finish the proof.
\end{proof}

\subsection{Composition transformations}
\label{composition}

Throughout this subsection, we fix a commutative diagram in $\lSpc/B$
\[
\cQ
:=
\begin{tikzcd}
\cS\ar[d,"c"']\ar[rd,"b"]
\\
\cZ\ar[d,"r"']\ar[r,"g"]&
\cY\ar[d,"f"]\ar[rd,"q"]
\\
\cS\ar[r,"a"]&
\cX\ar[r,"p"]&
\cS
\end{tikzcd}
\]
such that the inner square is cartesian, $f,p,q,r\in \lSmSpc$, $pa=\id$, $qb=\id$, and $rc=\id$.
We assume that $a$, $b$, and $c$ are closed immersions.
We have the natural transformation
\[
C
\colon
\Thom(q,b)
\to
\Thom(p,a)\Thom(r,c)
\]
given by the composition
\[
q_\sharp b_*
\xrightarrow{\simeq}
p_\sharp f_\sharp g_* c_*
\xrightarrow{Ex}
p_\sharp a_*r_\sharp c_*.
\]
Apply \cite[Proposition 9.17]{divspc} to $\cS\xrightarrow{c} \cZ\xrightarrow{g} \cY$ to obtain a cartesian square
\[
\begin{tikzcd}
\Deform_{\cS}\cZ\ar[r]\ar[d]&
\Deform_{\cS}\cY\ar[d]
\\
\cZ\times \boxx \ar[r]&
\Deform_{\cZ} \cY.
\end{tikzcd}
\]
By \cite[Lemma 9.7]{divspc}, we have an isomorphism $\Deform_{\cZ} \cY\simeq \Deform_{\cS}\cX\times_{\cX} \cY$.
Hence the inner square in the induced commutative diagram
\[
\cQ_d
:=
\begin{tikzcd}
\cS\times \boxx\ar[d,"c_d"']\ar[rd,"b_d"]
\\
\Deform_{\cS}\cZ\ar[d,"r_d"']\ar[r,"g_d"]&
\Deform_{\cS}\cY\ar[d,"f_d"]\ar[rd,"q_d"]
\\
\cS\times \boxx\ar[r,"a_d"]&
\Deform_{\cS}\cX\ar[r,"p_d"]&
\cS\times \boxx
\end{tikzcd}
\]
is cartesian.
We can similarly use \cite[Corollary 9.18]{divspc} to show that the inner square in the commutative diagram
\[
\cQ_n
:=
\begin{tikzcd}
\cS\ar[d,"c_n"']\ar[rd,"b_n"]
\\
\Normal_{\cS}\cZ\ar[d,"r_n"']\ar[r,"g_n"]&
\Normal_{\cS}\cY\ar[d,"f_n"]\ar[rd,"q_n"]
\\
\cS\ar[r,"a_n"]&
\Normal_{\cS}\cX\ar[r,"p_n"]&
\cS
\end{tikzcd}
\]
is cartesian.
There are induced natural transformations of the diagrams
\[
\cQ \to \cQ_d\leftarrow \cQ_n.
\]

We have the natural transformation
\[
C
\colon
\Thom_d'(p,a)\Thom_d'(r,c)
\to
\Thom_d'(q,b)
\]
given by the composition
\[
e_*c_d^!r_d^*e_d^*e_{d*}a_d^!p_d^*e^*
\xrightarrow{ad'}
e_*c_d^!r_d^*a_d^!p_d^*e^*
\xrightarrow{Ex}
e_*c_d^!g_d^!f_d^*p_d^*e^*
\xrightarrow{\simeq}
e_*b_d^!q_d^*e^*.
\]

\begin{proposition}
The squares
\[
\begin{tikzcd}
\Thom_d'(r,c)\Thom_d'(p,a)\ar[d,"T_dT_d"']\ar[r,"C"]&
\Thom_d'(q,b)\ar[d,"T_d"]
\\
\Thom'(r,c)\Thom'(p,a)\ar[r,"C"]&
\Thom'(q,b)
\end{tikzcd}
\quad
\begin{tikzcd}
\Thom_d'(r,c)\Thom_d'(p,a)\ar[d,"T_nT_n"']\ar[r,"C"]&
\Thom_d'(q,b)\ar[d,"T_n"]
\\
\Thom_n'(r,c)\Thom_n'(p,a)\ar[r,"C"]&
\Thom_n'(q,b)
\end{tikzcd}
\]
commute.
\end{proposition}
\begin{proof}
We only show that the left square commutes since the proofs are similar.
Let $i_1\colon \cS\to \cS\times \boxx$ be the $1$-section, let $l_1\colon \cX\to \Deform_{\cS} \cX$, $m_1\colon \cY\to \Deform_{\cS} \cY$, and $n_1\colon \cZ\to \Deform_{\cS} \cZ$ be the pullbacks of $i_1$, and let $e\colon S\times \boxx\to S$ be the projection.
Consider the commutative diagram
\[
\begin{tikzcd}[column sep=tiny, row sep=tiny]
&
\cS\ar[dd,"c",near start]\ar[rrdd,"b",bend left=40]\ar[ld,"i_1"']
\\
\cS\times \boxx\ar[dd,"c_d"']
\\
&
\cZ\ar[dd,"r"',near end]\ar[rr,"g",near end]\ar[ld,"n_1"']&
&
\cY\ar[dd,"f",near start]\ar[rrdd,"q",bend left=40]\ar[ld,"m_1"',near end]
\\
\Deform_{\cS}\cZ\ar[dd,"r_d"']\ar[rr,"g_d",near end,crossing over]&
&
\Deform_{\cS}\cY\ar[lluu,leftarrow,"b_d"',bend right=40,crossing over]
\\
&
\cS\ar[rr,"a",near start]\ar[ld,"i_1"']&
&
\cX\ar[rr,"p",near end]\ar[ld,"l_1"]&
&
\cS\ar[ld,"i_1"]
\\
\cS\times \boxx\ar[rr,"a_d"']&
&
\Deform_{\cS}\cX\ar[rr,"p_d"']\ar[uu,"f_d"',crossing over,near end,leftarrow]&
&
\cS\times \boxx\ar[lluu,"q_d"',leftarrow,bend right=40,crossing over]\ar[rr,"e"']&
&
\cS.
\end{tikzcd}
\]

It suffices to check that the diagram surrounding $(a)$ in
\[
\begin{tikzcd}[column sep=small, row sep=small]
e_*c_d^!r_d^*e^*e_*a_d^!p_d^*e^*\ar[d,"ad ad"']\ar[r,"ad'"]&
e_*c_d^!r_d^*a_d^!p_d^*e^*\ar[r,"Ex"]\ar[dd,"(a)" description,phantom]&
e_*c_d^!g_d^!f_d^*p_d^*e^*\ar[d,"\simeq"]
\\
e_*i_{1*}i_1^*c_d^!r_d^*e^*e_*i_{1*}i_1^*a_d^!p_d^*e^*\ar[d,"ExEx"']&
&
e_*b_d^!f_d^*p_d^*e^*\ar[d,"ad"]
\\
e_*i_{1*}c^!n_1^*r_d^*e^*e_*i_{1*}a^!l_1^*p_d^*e^*\ar[r,"\simeq"]\ar[dd,"\simeq"]&
e_*i_{1*}c^!r^*a^!l_1^*p_d^*e^*\ar[d,"Ex"']&
e_*i_{1*}i_1^*b_d^!f_d^*p_d^*e^*\ar[d,"Ex"]
\\
&
e_*i_{1*}c^!g^!f^*l_1^*p_d^*e^*\ar[r,"\simeq"']\ar[d,"\simeq"]&
e_*i_{1*}b^!m_1^*f_d^*p_d^*e^*\ar[d,"\simeq"]
\\
c^!r^*a^!p^*\ar[r,"Ex"]&
c^!g^!f^*p^*\ar[r,"\simeq"]&
b^!q^*
\end{tikzcd}
\]
commutes.
The diagram surrounding $(a)$ is obtained by applying $e_*(-)p_d^*e^*$ to the diagram
\[
\begin{tikzcd}[column sep=small, row sep=small]
c_d^!r_d^*e_d^*e_{d*}a_d^!\ar[d,"ad"']\ar[r,"ad'"]&
c_d^!r_d^*a_d^!\ar[d,"ad"]\ar[r,"Ex"]&
c_d^!g_d^!f_d^*\ar[r,"\simeq"]\ar[d,"ad"]&
b_d^!f_d^*\ar[d,"ad"]&
\\
i_{1*}i_1^*c_d^!r_d^*e_d^*e_{d*}a_d^!\ar[r,"ad'"]\ar[d,"ad"']&
i_{1*}i_1^*c_d^!r_d^*a_d^!\ar[r,"Ex"]\ar[d,"ad"]\ar[rd,"ad"]&
i_{1*}i_1^*c_d^!g_d^!f_d^*\ar[r,"\simeq"]\ar[ddddd,bend left=80,"ExEx"']&
i_{1*}i_1^*b_d^!f_d^*\ar[dddddd,"Ex"]
\\
i_{1*}i_1^*c_d^!r_d^*e_d^*e_{d*}i_{1*}i_1^*a_d^!\ar[r,"ad'"]\ar[d,"Ex"']&
i_{1*}i_1^*c_d^!r_d^*i_{1*}i_1^*a_d^!\ar[d,"Ex"]\ar[r,"ExEx"]&
i_{1*}i_1^*c_d^!r_d^*a_d^!l_{1*}l_1^*\ar[d,"Ex"]&
\\
i_{1*}c^!n_1^*r_d^*e_d^*e_{d*}i_{1*}i_1^*a_d^!\ar[r,"ad'"]\ar[d,"Ex"']&
i_{1*}c^!n_1^*r_d^*i_{1*}i_1^*a_d^!\ar[r,"ExEx"]\ar[d,"Ex"]&
i_{1*}c^!n_1^*r_d^*a_d^!l_{1*}l_1^*\ar[d,"Ex"]
\\
i_{1*}c^!n_1^*r_d^*e_d^*e_{d*}i_{1*}a^!l_1^*\ar[ddd,"\simeq"']\ar[r,"ad'"]&
i_{1*}c^!n_1^*r_d^*i_{1*}a^!l_1^*\ar[ru,"Ex"']\ar[d,"Ex"]\ar[rd,"(b)" description,phantom]&
i_{1*}c^!n_1^*g_d^!f_d^*l_{1*}l_1^*\ar[d,"Ex"]\ar[rddd,"(c)" description,very near start,bend left=33,phantom]
\\
&
i_{1*}c^!n_1^*n_{1*}r^*a^!l_1^*\ar[ldd,"ad'"']&
i_{1*}c^!g^!m_1^*f_d^*l_{1*}l_1^*\ar[ld,"\simeq"']
\\
&
i_*c^!g^!f^*l_1^*l_{1*}l_1^*\ar[d,"ad'"]&
i_*c^!g^!m_1^*f_d^*\ar[ld,"\simeq"]\ar[rd,"\simeq"]
\\
i_{1*}c^!r^*a^!l_1^*\ar[r,"Ex"']&
i_{1*}c^!g^!f^*l_1^*\ar[rr,"\simeq"']&
&
i_{1*}b^!m_1^*f_d^*.
\end{tikzcd}
\]
Hence it suffices to show that the diagrams surrounding $(b)$ and $(c)$ commute.

Apply $i_{1*}(-)l_1^*$ to the commutative diagram
\[
\begin{tikzcd}[column sep=small, row sep=small]
c^!n_1^*r_d^*i_{1*}a^!\ar[r,"Ex"]\ar[d,"Ex"']&
c^!n_1^*r_d^*a_d^!l_{1*}\ar[r,"Ex"]&
c^!n_1^*g_d^!f_d^*l_{1*}\ar[d,"Ex"]\ar[r,"Ex"]&
c^!g^!m_1^*f_d^*l_{1*}\ar[d,"\simeq"]\ar[ldd,"Ex"',near start]
\\
c^!n_1^*n_{1*}r^*a^!\ar[r,"Ex"]\ar[d,"ad'"']&
c^!n_1^*n_{1*}g^!f^*\ar[r,"Ex"]\ar[rdd,"ad'"']&
c^!n_1^*g_d^!m_{1*}f^*\ar[d,"Ex"']&
c^!g^!f^*l_1^*l_{1*}\ar[ldd,"ad'"]
\\
c^!r^*a^!\ar[rrd,"Ex"']
&
&
c^!g^!m_1^*m_{1*}f^*\ar[d,"ad'"]&
\\
&
&
c^!g^!f^*
\end{tikzcd}
\]
to show that the diagram surrounding $(b)$ commutes.
Apply $i_{1*}(-)$ to the commutative diagram
\[
\begin{tikzcd}[column sep=small, row sep=small]
i_1^*c_d^!r_d^*a_d^!\ar[d,"ad"']\ar[rr,"Ex"]\ar[rd,"Ex"]&
&
i_1^*c_d^!g_d^!f_d^*\ar[d,"ExEx"]
\\
i_1^*c_d^!r_d^*a_d^!l_{1*}l_1^*\ar[d,"Ex"']&
c^!n_1^*r_d^*a_d^!\ar[ld,"ad"]\ar[r,"ExEx"]&
c^!g^!m_1^*f_d^*\ar[lldd,"ad"]
\\
c^!n_1^*r_d^*a_d^!l_{1*}l_1^*\ar[d,"ExEx"']&
&
c^!g^!f^*l_1^*\ar[ld,"ad"']\ar[u,"\simeq"',leftarrow]
\\
c^!g^!m_1^*f_d^*l_{1*}l_1^*\ar[r,"\simeq"]&
c^!g^!f^*l_1^*l_{1*}l_1^*\ar[r,"ad'"]&
c^!g^!f^*l_1^*\ar[u,"\id"',leftarrow]
\end{tikzcd}
\]
to show that the diagram surrounding $(c)$ commutes.
\end{proof}

\section{Purity}\label{Sec4}

Recall that we fix a sheaf
\[
\sT\in \Sh_{\dNis}(\lSch/B,\CAlg(\PrL))
\simeq
\Sh_{\dNis}(\lSpc/B,\CAlg(\PrL))
\]
satisfying the conditions in the introduction.

In \S \ref{puritytransformations}, we define purity transformations and prove their basic properties.
In \S \ref{nonstrictpurity}, we prove that a certain compactification of the morphism $\A_{\N^2}\to \A_\N$ induced by the diagonal homomorphism $\N\to \N^2$ is pure.

\subsection{Purity transformations}
\label{puritytransformations}

\begin{definition}
Let $f\colon X\to S$ be a log smooth morphism in $\lSch/B$.
We have a commutative diagram
\[
\begin{tikzcd}
X\ar[r,"a"]&
X\times_S X\ar[d,"p_1"']\ar[r,"p_2"]&
X\ar[d,"f"]
\\
&
X\ar[r,"f"]&
S,
\end{tikzcd}
\]
where $a$ is the diagonal morphism, and $p_1$ (resp.\ $p_2$) is the first (resp.\ second) projection.
We set
\[
\Sigma_f
:=
p_{2\sharp} a_*
\text{ and }
\Omega_f
:=
a^!p_2^*.
\]
If we assume also that $f$ is proper, then the \emph{purity transformation associated with $f$} is defined to be the natural transformation
\begin{equation}
\mathfrak{p}_f
\colon
f_\sharp
\to
f_*\Sigma_f
\end{equation}
given by the composition
\[
f_\sharp
\xrightarrow{\simeq}
f_\sharp p_{1*}a_*
\xrightarrow{Ex}
f_*p_{2\sharp}a_*
=
f_*\Sigma_f.
\]
\end{definition}

\begin{lemma}
\label{N2.4}
Let $f\colon X\to S$ be a log smooth morphism in $\lSch/B$, and let $j\colon U\to X$ be an open immersion in $\lSch/B$.
Then there is a natural isomorphism
\begin{equation}
\label{N2.4.1}
\Sigma_f j_*
\xrightarrow{\simeq}
j_*\Sigma_{fj}.
\end{equation}
\end{lemma}
\begin{proof}
Consider the induced commutative diagram
\[
\begin{tikzcd}
&U\times_S U\ar[d,"u"]\ar[rd,"r_2"]
\\
U\ar[ru,"c"]\ar[r,"b"]\ar[d,"j"']&
X\times_S U\ar[d]\ar[r,"q_2"]&
U\ar[d,"j"]
\\
X\ar[r,"a"]&
X\times_S X\ar[r,"p_2"]&
X,
\end{tikzcd}
\]
where $a$, $b$, and $c$ are the graph morphisms, and $p_2$, $q_2$, and $r_2$ are the second projections.
We have the natural transformations
\[
\Sigma_f j_*
\xrightarrow{Ex}
j_*\Thom(q_2,b)
=
j_*q_{2\sharp} b_*
\xrightarrow{Ex^{-1}}
j_*q_{2\sharp}u_\sharp c_*
\xrightarrow{\simeq}
j_*r_{2\sharp}c_*
=
j_*\Sigma_{fj}.
\]
The first arrow is an isomorphism by Proposition \ref{thom.18}.
The third arrow is defined and an isomorphism by Proposition \ref{loc.2}.
The composition gives the desired natural isomorphism.
\end{proof}

\begin{lemma}
\label{N2.15}
Suppose that $f\colon X\to S$ is a log smooth morphism in $\lSch/B$ with a section $i$.
We have the natural isomorphisms
\begin{gather*}
\Sigma_f i_*
\xrightarrow{Ex}
i_*\Thom(f,i),
\;
\Thom(f,i)f_*
\xrightarrow{Ex}
f_*\Sigma_f,
\\
i^*\Sigma_f
\xrightarrow{Ex}
\Thom'(f,i)i^*,
\;
f^*\Thom'(f,i)
\xrightarrow{Ex}
\Omega_f f^*.
\end{gather*}
\end{lemma}
\begin{proof}
Consider the commutative diagram with cartesian squares
\[
\begin{tikzcd}
S\ar[r,"i"]\ar[d,"i"']&
X\ar[r,"f"]\ar[d]&
S\ar[d,"i"]
\\
X\ar[r,"a"]\ar[d,"f"']&
X\times_S X\ar[r,"p_2"]\ar[d,"p_1"]&
X\ar[d,"f"]
\\
S\ar[r,"i"]&
X\ar[r,"f"]&
S,
\end{tikzcd}
\]
where $a$ is the diagonal embedding, and $p_1$ and $p_2$ are the first and second projections.
Use Propositions \ref{thom.7} and \ref{thom.18} for the upper and lower parts of this diagram to conclude.
\end{proof}

\begin{lemma}
\label{N2.5}
Suppose $S\in \lSch/B$.
Let $f\colon \A_S^1\to S$ be the projection, and let $i\colon S\to \A_S^1$ be the zero section.
Then the natural transformations
\[
f_* \Sigma_f f^*
\xrightarrow{ad}
f_*\Sigma_f i_*i^* f^*
\text{ and }
f_\sharp i_*i^! \Omega_f f^*
\xrightarrow{ad'}
f_\sharp \Omega_f f^*
\]
are isomorphisms.
\end{lemma}
\begin{proof}
Consider the commutative square
\[
\begin{tikzcd}
\Thom(f,i)f_* f^*\ar[r,"ad"]\ar[d,"Ex"']&
\Thom(f,i)f_*i_*i^*f^*\ar[d,"Ex"]
\\
f_*\Sigma_f f^* \ar[r,"ad"]&
f_*\Sigma_f i_*i^*f^*,
\end{tikzcd}
\]
where the vertical arrows are isomorphisms by Lemma \ref{N2.15}.
The upper horizontal arrow is an isomorphism by $\A^1$-invariance.
It follows that the lower horizontal arrow is an isomorphism.

Consider the commutative square
\[
\begin{tikzcd}
f_\sharp  i_*i^! f^*\Thom'(f,i)\ar[r,"ad'"]\ar[d,"Ex"']&
f_\sharp f^*\Thom'(f,i)\ar[d,"Ex"]
\\
f_\sharp i_*i^!\Omega_f f^* \ar[r,"ad'"]&
f_\sharp  \Omega_f f^*,
\end{tikzcd}
\]
where the vertical arrows are isomorphisms by Lemma \ref{N2.15}.
Since $f_\sharp i_*=\Thom(f,i)$, we have an isomorphism $f_\sharp i_*i^!f^*\simeq \id$.
Hence the upper horizontal arrow is an isomorphism by $\A^1$-invariance.
It follows that the lower horizontal arrow is an isomorphism.
\end{proof}

\begin{lemma}
\label{N2.16}
Let $f\colon X\to S$ be a proper log smooth morphism in $\lSch/B$.
If the two natural transformations
\[
f_\sharp f^*\xrightarrow{\mathfrak{p}_f f^*} f_*\Sigma_f f^*
\text{ and }
f_\sharp \Omega_f f^*
\xrightarrow{\mathfrak{p}_f \Omega_f f^*} 
f_*\Sigma_f \Omega_f f^*
\]
are isomorphisms, then the natural transformation
\[
\mathfrak{p}_f\colon f_\sharp \to f_*\Sigma_f
\]
is an isomorphism.
\end{lemma}
\begin{proof}
Let $\alpha$ be the composition
\[
f_*\Sigma_{f}
\xrightarrow{ad}
f_*\Sigma_{f} f^*f_\sharp
\xrightarrow{(\mathfrak{p}_f f^*)^{-1}}
f_\sharp f^*f_\sharp
\xrightarrow{ad'}
f_\sharp.
\]
The diagram
\[
\begin{tikzcd}
f_\sharp\ar[d,"\mathfrak{p}_f"']\ar[r,"ad"]&
f_\sharp f^*f_\sharp \ar[rrd,"\simeq",bend left]\ar[d,"\mathfrak{p}_ff^*f_\sharp"']
\\
f_*\Sigma_f\ar[r,"ad"]&
f_*\Sigma_{f} f^*f_\sharp\ar[rr,"(\mathfrak{p}_f f^*f_\sharp)^{-1}"]&&
f_\sharp f^*f_\sharp\ar[r,"ad'"]&
f_\sharp
\end{tikzcd}
\]
commutes, which shows that $\alpha$ is a left quasi-inverse of $\mathfrak{p}_f$.
Let $\beta$ be the composition
\[
f_*\Sigma_f
\xrightarrow{ad}
f_*\Sigma_f \Omega_f f^*f_*\Sigma_{f}
\xrightarrow{(\mathfrak{p}_f\Omega_f f^*f_*\Sigma_f)^{-1}}
f_\sharp \Omega_{f} f^*f_*\Sigma_{f}
\xrightarrow{ad'}
f_\sharp,
\]
where we use the adjunction pair $(\Omega_f f^*,f_*\Sigma_f)$ for the first and third natural transformations.
The diagram
\[
\begin{tikzcd}
f_*\Sigma_{f}\ar[r,"ad"]&
f_*\Sigma_f \Omega_f f^*f_*\Sigma_{f}\ar[rrr,"(\mathfrak{p}_f\Omega_ff^*f_*\Sigma_f)^{-1}"]\ar[rrrd,"\simeq"']&&&
f_\sharp \Omega_{f} f^*f_*\Sigma_{f}\ar[r,"ad'"]\ar[d,"\mathfrak{p}_f\Omega_ff^*f_*\Sigma_f"]&
f_\sharp\ar[d,"\mathfrak{p}_f"]
\\
&
&&&
f_* \Sigma_f \Omega_f f^*f_*\Sigma_{f}\ar[r,"ad'"]&
f_*\Sigma_{f}
\end{tikzcd}
\]
commutes, which shows that $\beta$ is a right quasi-inverse of $\mathfrak{p}_f$.
Hence $\mathfrak{p}_f$ has both left quasi-inverse and right quasi-inverse, so $\mathfrak{p}_f$ is an isomorphism.
\end{proof}

\begin{lemma}
\label{N2.1}
Let $f\colon X\to S$ be a proper log smooth morphism in $\lSch/B$, and let $i$ be its section.
Then the natural transformation
\[
f_\sharp i_{*}
\xrightarrow{\mathfrak{p}_f}
f_*\Sigma_f i_{*}
\]
is an isomorphism.
\end{lemma}
\begin{proof}
Consider the induced commutative diagram with cartesian squares
\[
\begin{tikzcd}
S\ar[r,"i"]\ar[d,"i"']&
X\ar[r,"f"]\ar[d,"g"']&
S\ar[d,"i"]
\\
X\ar[r,"a"]&
X\times_S X\ar[r,"p_2"]\ar[d,"p_1"']&
X\ar[d,"f"]
\\
&
X\ar[r,"f"]&
S.
\end{tikzcd}
\]
The commutative diagram
\[
\begin{tikzcd}
f_\sharp i_*\ar[r,"\simeq"]\ar[rd,"\simeq"']&
f_\sharp p_{1*}a_*i_*\ar[r,"Ex"]\ar[d,"\simeq"]&
f_*p_{2\sharp}a_*i_*\ar[d,"\simeq"]
\\
&
f_\sharp p_{1*}g_*i_*\ar[r,"Ex"]\ar[rd,"\simeq"']&
f_*p_{2\sharp}g_*i_*\ar[d,"Ex"]
\\
&
&
f_*i_*f_\sharp i_*,
\end{tikzcd}
\]
commutes.
The lower right vertical arrow is an isomorphism by Proposition \ref{loc.2}.
It follows that the composition $f_\sharp i_*\to f_*p_{2\sharp}a_*i_*$ is an isomorphism.
\end{proof}

\subsection{Purity for a compactification of \texorpdfstring{$\mathbb{A}_{\mathbb{N}^2}\rightarrow \mathbb{A}_\mathbb{N}$}{AN2AN}}
\label{nonstrictpurity}

Throughout this subsection, we fix the following notation.
We set
\begin{gather*}
V_1:=\Spec(\N x\oplus \N y \to \Z[x,y]),
\\
V_2:=\Spec(\N (xy)\to \Z[xy,y^{-1}]),
\\
V_3:=\Spec(\N (xy) \to \Z[xy,x^{-1}]),
\\
V_4:=\Spec(\N (xy) \oplus \N (x^{-1})\to \Z[xy,x^{-1}]),
\\
V_5:=\Spec(\N x\oplus \N (xy) \to \Z[x,xy]),
\\
V_{12}:=V_{14}:=\Spec(\N x\to \Z[x,y,y^{-1}]),
\\
V_{13}:=V_{35}:=\Spec(\N y \to \Z[x,x^{-1},y]),
\\
V_{23}:=V_{34}:=\Spec(\Z[x,x^{-1},y,y^{-1}]).
\end{gather*}
Let $W$ (resp.\ $W_5$, $W_6$) be the gluing of $V_1$, $V_2$, and $V_3$ (resp.\ $V_1$, $V_3$, and $V_4$, resp.\ $V_3$ and $V_5$) along $V_{12}$, $V_{13}$, and $V_{23}$ (resp.\ $V_{13}$, $V_{14}$, and $V_{34}$, resp.\ $V_{35}$).
Let $W_2$ (resp.\ $W_4$) be the gluing of $V_1$ and $V_2$ (resp.\ $V_1$ and $V_3$) along $V_{12}$ (resp.\ $V_{13}$). We set $W_1:=V_2$ and $W_3:=V_3$.
These are fs log schemes over $\A_\N:=\Spec(\N t\to \Z[t])$, where the morphisms $W\to \A_\N$ and $W_i\to \A_\N$ for $i=1,\ldots,6$ are obtained by the formula $t\mapsto xy$.
We also have
\[
W-W_2
\simeq
V_3-V_{13}
\simeq
\Spec(\N(xy)\to \Z[xy,x^{-1}]/(x^{-1})),
\]
where the closed complements are defined with the reduced scheme structures.
Observe that the induced morphism $W-W_2\to \A_\N$ is an isomorphism.

Let $S\to \A_{\N,B}:=B\times \A_\N$ be a morphism in $\lSch/B$.
We set $X:=S\times_{\A_\N} W$ and $X_i:=S\times_{\A_\N} W_i$ for $i=1,\ldots,6$.
If $u$ and $v$ are integers, let
\[
j_{uv}\colon X_u\to X_v,
\;
j_u\colon X_u\to X,
\;
i_{uv}\colon X-X_u\to X_v,
\text{ and }
i_u\colon X-X_u\to X
\]
be the induced morphisms whenever they are meaningful.
Let $f\colon X\to S$ and $f_i\colon X_i\to S$ for $i=1,\ldots,6$ be the projections.
We have the commutative diagram
\[
\begin{tikzcd}
&
X-X_2\ar[ld,"i_{23}"']\ar[d,"i_{24}"]\ar[rd,"i_{25}"]\ar[rrd,"i_{26}",bend left]
\\
X_3\ar[r,"j_{34}"]\ar[rrd,"j_3"',near start]\ar[rrdd,"f_3"',bend right]&
X_4\ar[rdd,"f_4"',crossing over,near end]\ar[rd,"j_4"]\ar[r,"j_{45}"]&
X_5\ar[d,"j_5"']\ar[dd,bend left,"f_5"]\ar[r,"j_{56}"]&
X_6\ar[ldd,"f_6",bend left]
\\
&
&
X\ar[d,"f"']\\
&
&
S.
\end{tikzcd}
\]
Let us explain the properties of the fs log schemes and morphisms between them constructed above:
\begin{itemize}
\item $i_{23}$, $i_{24}$, $i_{25}$, $i_{26}$, and $i_2$ are closed immersions.
\item $j_{34}$, $j_{45}$, $j_3$, and $j_4$ are open immersions.
\item $j_{56}$ is a dividing cover.
\item The underlying morphism of schemes $\ul{j_5}$ is an isomorphism.
\item $X_3\simeq S\times \A^1$ and $X_6\simeq S\times \boxx$.
\item $f$ is proper.
\item The composition $X-X_2\to S$ is an isomorphism.
\end{itemize}

\begin{lemma}
\label{N2.7}
The natural transformations
\[
f_5^*\xrightarrow{ad}j_{45*}j_{45^*}f_5^*
\text{ and }
f_6^*\xrightarrow{ad}j_{36*}j_{36^*}f_6^*
\]
are isomorphisms.
\end{lemma}
\begin{proof}
There are isomorphisms $X_5-\partial_S X_5\simeq X_4$ and $X_6-\partial_S X_6\simeq X_3$.
Proposition \ref{loc.3} finishes the proof.
\end{proof}

\begin{lemma}
\label{N2.6}
The natural transformation
\[
\id
\xrightarrow{ad}
f_{4*}f_4^*
\]
is an isomorphism.
\end{lemma}
\begin{proof}
We have natural transformations
\[
f_{5*}f_5^*
\xrightarrow{ad}
f_{5*}j_{45*}j_{45}^*f_5^*
\xrightarrow{\simeq}
f_{4*}f_4^*,
\]
and the first arrow is an isomorphism by Lemma \ref{N2.7}.
We similarly have an isomorphism
\[
f_{6*}f_6^*
\simeq
f_{3*}f_3^*.
\]
Since $h_{j_{56}}$ is an isomorphism, it suffices to show that the natural transformation
\[
\id \xrightarrow{ad} f_{3*}f_3^*
\]
is an isomorphism.
This follows from $\A^1$-invariance.
\end{proof}

\begin{lemma}
\label{N2.3}
The natural transformation
\[
f_*\Sigma_f j_{4*}j_4^*f^*
\xrightarrow{ad}
f_*\Sigma_f j_{4*}i_{24*}i_{24}^*j_4^*f^*
\]
is an isomorphism.
\end{lemma}
\begin{proof}
We have a commutative diagram
\begin{equation}
\label{N2.3.3}
\begin{tikzcd}[row sep=small]
f_*\Sigma_f j_{4*}j_4^*f^*\ar[d,"\simeq"']\ar[r,"ad"]&
f_*\Sigma_f j_{4*}i_{24*}i_{24}^*j_4^*f^*\ar[d,"\simeq"]
\\
f_*j_{4*}\Sigma_{f_4}j_4^*f^*\ar[d,"\simeq"']\ar[r,"ad"]&
f_*j_{4*}\Sigma_{f_4}i_{24*}i_{24}^*j_4^*f^*\ar[d,"\simeq"]
\\
f_{4*}\Sigma_{f_4}f_4^*\ar[r,"ad"]&
f_{4*}\Sigma_{f_4}i_{24*}i_{24}^*f_4^*,
\end{tikzcd}
\end{equation}
where the upper vertical arrows are obtained by Lemma \ref{N2.4}.
We have a commutative diagram
\begin{equation}
\label{N2.3.1}
\begin{tikzcd}[row sep=small]
f_{4*}\Sigma_{f_4}f_4^*\ar[d,"\simeq"']\ar[r,"ad"]&
f_{4*}\Sigma_{f_4}i_{24*}i_{24}^*f_4^*\ar[d,"\simeq"]
\\
f_{5*}j_{45*}\Sigma_{f_4}j_{45}^*f_5^*\ar[d,"\simeq"']\ar[r,"ad"]&
f_{5*}j_{45*}\Sigma_{f_4}i_{24*}i_{24}^*j_{45}^*f_5^*\ar[d,"\simeq"]
\\
f_{5*}\Sigma_{f_5}j_{45*}j_{45}^*f_5^*\ar[d,"\simeq"',"ad^{-1}"]\ar[r,"ad"]&
f_{5*}\Sigma_{f_5}j_{45*}i_{24*}i_{24}^*j_{45}^*f_5^*\ar[d,"\simeq"]
\\
f_{5*}\Sigma_{f_5}f_5^*\ar[r,"ad"]&
f_{5*}\Sigma_{f_5}i_{25*}i_{25}^*f_5^*,
\end{tikzcd}
\end{equation}
where the middle vertical arrows are obtained by Lemma \ref{N2.4}, and the left lower vertical arrow is defined and an isomorphism by Lemma \ref{N2.7}. 

We similarly have a commutative diagram
\begin{equation}
\label{N2.3.2}
\begin{tikzcd}[row sep=small]
f_{3*}\Sigma_{f_3}f_3^*\ar[r,"ad"]\ar[d,"\simeq"']&
f_{3*}\Sigma_{f_3}i_{23*}i_{23}^*f_3^*\ar[d,"\simeq"]
\\
f_{6*}\Sigma_{f_6}f_6^*\ar[r,"ad"]&
f_{6*}\Sigma_{f_6}i_{26*}i_{26}^*f_6^*
\end{tikzcd}
\end{equation}
Since $h_{j_{56}}$ is an isomorphism, the bottom row of \eqref{N2.3.1} is an isomorphism if and only if the bottom row of \eqref{N2.3.2} is an isomorphism.
Hence it suffices to show that the top row of \eqref{N2.3.2} is an isomorphism.
This follows from Lemma \ref{N2.5}.
\end{proof}

\begin{lemma}
\label{N2.14}
The natural transformation
\[
f_\sharp j_{4\sharp}i_{24*}i_{24}^!j_4^*\Omega_f f^*
\xrightarrow{ad'}
f_\sharp j_{4\sharp}j_4^*\Omega_f f^*
\]
is an isomorphism.
\end{lemma}
\begin{proof}
We have a commutative diagram
\begin{equation}
\label{N2.14.3}
\begin{tikzcd}[row sep=small]
f_\sharp j_{4\sharp}i_{24*}i_{24}^!j_4^* \Omega_f f^*\ar[d,"\simeq"']\ar[r,"ad'"]&
f_\sharp j_{4\sharp}j_4^*\Omega_f f^*\ar[d,"\simeq"]
\\
f_\sharp j_{4\sharp}i_{24*}i_{24}^!\Omega_{f_4}j_4^* f^*
\ar[d,"\simeq"']\ar[r,"ad'"]&
f_\sharp j_{4\sharp}\Omega_{f_4}j_4^*f^*
\ar[d,"\simeq"]
\\
f_{4\sharp}i_{24*}i_{24}^!\Omega_{f_4}f_4^*
\ar[r,"ad'"]&
f_{4\sharp}\Omega_{f_4}f_4^*,
\end{tikzcd}
\end{equation}
where the upper vertical arrows are obtained by a left adjoint of \eqref{N2.4.1}.
We have a commutative diagram
\begin{equation}
\label{N2.14.1}
\begin{tikzcd}[row sep=small]
f_{4\sharp}i_{24*}i_{24}^!\Omega_{f_4}f_4^*
\ar[d,"\simeq"']\ar[r,"ad'"]&
f_{4\sharp}\Omega_{f_4}f_4^*
\ar[d,"\simeq"]
\\
f_{5\sharp}j_{45\sharp}i_{24*}i_{24}^!\Omega_{f_4} j_{45}^*f_5^*
\ar[d,"\simeq"']\ar[r,"ad'"]&
f_{5\sharp}j_{45\sharp}\Omega_{f_4}j_{45}^*f_5^*
\ar[d,"\simeq"]
\\
f_{5\sharp}j_{45\sharp}i_{24*}i_{24}^!j_{45}^*\Omega_{f_5}f_5^*
\ar[d,"\simeq"',"ad'"]\ar[r,"ad'"]&
f_{5\sharp}j_{45\sharp}j_{45}^*\Omega_{f_5}f_5^*
\ar[d,"\simeq"]
\\
f_{5\sharp}i_{25*}i_{25}^!\Omega_{f_5}f_5^*
\ar[r,"ad'"]&
f_{5\sharp}\Omega_{f_5}f_5^*,
\end{tikzcd}
\end{equation}
where the middle vertical arrows are obtained by the left adjoint of \eqref{N2.4.1}, the left below vertical arrow is an isomorphism by 
Proposition \ref{fundamental.7}, and the right below vertical arrow is defined and an isomorphism by Lemma \ref{N2.7}.

We similarly have a commutative diagram
\begin{equation}
\label{N2.14.2}
\begin{tikzcd}[row sep=small]
f_{3\sharp}i_{23*}i_{23}^!\Omega_{f_3}f_3^*
\ar[r,"ad'"]\ar[d,"\simeq"']&
f_{3\sharp}\Omega_{f_3}f_3^*
\ar[d,"\simeq"]
\\
f_{6\sharp}i_{26*}i_{26}^!\Omega_{f_6}f_6^*
\ar[r,"ad'"]&
f_{6\sharp}\Omega_{f_6}f_6^*.
\end{tikzcd}
\end{equation}
Since $h_{j_{56}}$ is an isomorphism, the bottom row of \eqref{N2.14.1} is an isomorphism if and only if the bottom row of \eqref{N2.14.2} is an isomorphism.
Hence it suffices to show that the top row of \eqref{N2.14.2} is an isomorphism.
This follows from Lemma \ref{N2.5}.
\end{proof}

Observe that $fi_{4}\colon X-X_4\to S$ is an isomorphism.

\begin{lemma}
\label{N2.8}
The sequence
\[
f_\sharp i_{4*}i_4^!f^*
\xrightarrow{ad'}
f_\sharp f^*
\xrightarrow{ad}
f_\sharp i_{2*}i_2^*f^*
\]
is a cofiber sequence.
\end{lemma}
\begin{proof}
The localization property implies that the sequence
\[
f_\sharp j_{2\sharp}j_2^* f^*
\xrightarrow{ad'}
f_\sharp f^*
\xrightarrow{ad}
f_\sharp i_{2*}i_2^*f^*
\]
is a cofiber sequence.
Hence it suffices to show that the composition
\[
f_\sharp i_{4*}i_4^!f^*
\xrightarrow{\simeq}
f_\sharp j_{4\sharp}i_{24*}i_{24}^!j_4^*f^*
\xrightarrow{ad'}
f_\sharp j_{4\sharp}j_4^*f^*
\xrightarrow{\simeq}
f_{4\sharp}f_4^*
\]
is an isomorphism, where the first arrow is defined and an isomorphism by Proposition \ref{fundamental.7}.

Since $f_\sharp i_{4*}=\Thom(f,i_4)$ is an equivalence of $\infty$-categories by Theorem \ref{thom.10}, we have a natural isomorphism $f_\sharp i_{4*}i_4^!f^*\simeq \id$.
Lemma \ref{N2.6} finishes the proof.
\end{proof}

\begin{lemma}
\label{N2.9}
The sequence
\[
f_* \Sigma_f i_{4*}i_4^!f^*
\xrightarrow{ad'}
f_* \Sigma_f  f^*
\xrightarrow{ad}
f_* \Sigma_f i_{2*}i_2^*f^*
\]
is a cofiber sequence.
\end{lemma}
\begin{proof}
The localization property implies that the sequence
\[
f_*\Sigma_f i_{4*}i_4^!f^*
\xrightarrow{ad'}
f_*\Sigma_f f^*
\xrightarrow{ad}
f_*\Sigma_f j_{4*}j_4^*f^*
\]
is a cofiber sequence.
We conclude together with Lemma \ref{N2.3}.
\end{proof}

\begin{lemma}
\label{N2.10}
The sequence
\[
f_\sharp i_{4*}i_4^!\Omega_f f^*
\xrightarrow{ad'}
f_\sharp \Omega_f f^*
\xrightarrow{ad}
f_\sharp i_{2*}i_2^* \Omega_f f^*
\]
is a cofiber sequence.
\end{lemma}
\begin{proof}
After switching the coordinates $x$ and $y$ in the formulations of $V_1$, $V_2$, and $V_3$,
we can instead show that the sequence
\[
f_\sharp i_{2*}i_2^!\Omega_f f^*
\xrightarrow{ad'}
f_\sharp \Omega_f f^*
\xrightarrow{ad}
f_\sharp i_{4*}i_4^* \Omega_f f^*
\]
is a cofiber sequence.

The localization property implies that the sequence
\[
f_\sharp j_{4\sharp}j_4^*\Omega_f f^*
\xrightarrow{ad'}
f_\sharp \Omega_f f^*
\xrightarrow{ad}
f_\sharp i_{4*}i_4^* \Omega_f f^*
\]
is a cofiber sequence.
Hence it suffices to show that the composition
\[
f_\sharp i_{2*}i_2^!\Omega_f f^*
\xrightarrow{\simeq}
f_\sharp j_{4\sharp}i_{24*}i_{24}^! j_4^*\Omega_f f^*
\xrightarrow{ad'}
f_\sharp j_{4\sharp}j_4^* \Omega_f f^*
\]
is an isomorphism, where the first arrow is defined and an isomorphism by Proposition \ref{fundamental.7}.
This follows from Lemma \ref{N2.14}.
\end{proof}

\begin{lemma}
\label{N2.11}
The sequence
\[
f_*\Sigma_f i_{4*}i_4^!\Omega_f f^*
\xrightarrow{ad'}
f_*\Sigma_f \Omega_f f^*
\xrightarrow{ad}
f_*\Sigma_f i_{2*}i_2^*\Omega_f f^*
\]
is a cofiber sequence.
\end{lemma}
\begin{proof}
The localization sequence implies that the sequence
\[
f_*\Sigma_f i_{4*}i_4^!\Omega_f f^*
\xrightarrow{ad'}
f_*\Sigma_f \Omega_f f^*
\xrightarrow{ad}
f_*\Sigma_f j_{4*}j_4^*\Omega_f f^*
\]
is a cofiber sequence.
Hence it suffices to show that the composition
\[
f_*\Sigma_f j_{4*}j_4^*\Omega_f f^*
\xrightarrow{ad}
f_*\Sigma_f j_{4*}i_{24*}i_{24}^*j_4^*\Omega_f f^*
\xrightarrow{\simeq}
f_*\Sigma_f i_{2*}i_2^*\Omega_f f^*
\]
is an isomorphism.

Lemma \ref{N2.4} gives a natural isomorphism $\Sigma_f j_{4*}\simeq j_{4*}\Sigma_{f_4}$, whose left adjoint is $\Omega_{f_4}j_4^*\simeq j_4^*\Omega_f$.
Lemma \ref{N2.15} gives natural isomorphisms $\Sigma_f i_{2*}\simeq i_{2*}\Thom(f,i_2)$ and $i_2^*\Omega_f\simeq \Thom'(f,i_2)i_2^*$.
Hence it suffices to show that the composition
\[
f_*j_{4*}j_4^* f^*
\xrightarrow{ad}
f_*j_{4*}i_{24*}i_{24}^*j_4^* f^*
\xrightarrow{\simeq}
f_* i_{2*}i_2^* f^*
\]
is an isomorphism.
This follows from Lemma \ref{N2.6} since $fi_2\colon X-X_2\to S$ is an isomorphism.
\end{proof}

\begin{lemma}
\label{N2.12}
The natural transformation
\[
f_\sharp f^* \xrightarrow{\mathfrak{p}_f} f_* \Sigma_{f} f^*
\]
is an isomorphism.
\end{lemma}
\begin{proof}
Consider the commutative diagram
\begin{equation*}
\begin{tikzcd}
f_\sharp i_{4*}i_4^!f^*\arrow[r,"ad'"]\arrow[d,"\mathfrak{p}_f"']&
f_\sharp f^*\arrow[r,"ad"]\arrow[d,"\mathfrak{p}_f"]&
f_\sharp i_{2*}i_2^*f^*\arrow[d,"\mathfrak{p}_f"]
\\
f_* \Sigma_{f} i_{4*}i_4^!f^*\arrow[r,"ad'"]&
f_*\Sigma_{f} f^*\arrow[r,"ad"]&
f_* \Sigma_{f} i_{2*}i_2^* f^*
\end{tikzcd}
\end{equation*}
whose rows are cofiber sequences by Lemmas \ref{N2.8} and \ref{N2.9}.
The left and right vertical arrows are isomorphisms by Lemma \ref{N2.1}.
It follows that the middle vertical arrow is an isomorphism.
\end{proof}

\begin{lemma}
\label{N2.13}
The natural transformation
\[
f_\sharp \Omega_{f} f^* \xrightarrow{\mathfrak{p}_f} f_*\Sigma_{f}\Omega_f  f^*
\]
is an equivalence.
\end{lemma}
\begin{proof}
Consider the commutative diagram
\[
\begin{tikzcd}
f_\sharp  i_{4*}i_4^!\Omega_{f}f^*\arrow[r,"ad'"]\arrow[d,"\mathfrak{p}_f"']&
f_\sharp \Omega_{f} f^*\arrow[r,"ad"]\arrow[d,"\mathfrak{p}_f"]&
f_\sharp  i_{2*}i_2^*\Omega_{f}f^*\arrow[d,"\mathfrak{p}_f"]
\\
f_* \Sigma_f i_{4*}i_4^!\Omega_{f}f^*\arrow[r,"ad'"]&
f_* \Sigma_f \Omega_{f} f^*\arrow[r,"ad"]&
f_*\Sigma_f i_{2*}i_2^* \Omega_{f} f^*
\end{tikzcd}
\]
whose rows are cofiber sequences by Lemmas \ref{N2.10} and \ref{N2.11}.
The left and right vertical arrows are isomorphisms by Lemma \ref{N2.1}.
It follows that the middle vertical arrow is an isomorphism.
\end{proof}

\begin{theorem}
\label{N2.17}
The natural transformation
\[
\mathfrak{p}_f\colon f_\sharp \to f_*\Sigma_{f}
\]
is an isomorphism.
\end{theorem}
\begin{proof}
Combine Lemmas \ref{N2.16}, \ref{N2.12}, and \ref{N2.13}.
\end{proof}

The following result is needed in \cite{logsix}.
We are still using the notation $f\colon X\to S$ constructed at the beginning of this subsection.

\begin{proposition}
\label{N2.18}
Let $d\colon S'\to S$ be a strict closed immersion with its open complement $u\colon S''\to S$.
Consider the induced commutative diagram with cartesian squares
\[
\begin{tikzcd}
X'\ar[d,"f'"']\ar[r,"d'"]&
X\ar[d,"f"]\ar[r,leftarrow,"u'"]&
X''\ar[d,"f''"]
\\
S'\ar[r,"d"]&
S\ar[r,leftarrow,"u"]&
S''.
\end{tikzcd}
\]
Then we have $d^*f_*u_\sharp'\simeq 0$.
\end{proposition}
\begin{proof}
Let $a''\colon X''\to X\times_S X''$ be the graph morphism, and let $p''\colon X\times_S X''\to X''$ be the projection.
A left adjoint of \eqref{thom.13.2} gives an isomorphism
\begin{equation}
\label{N2.18.1}
\Omega_f u_\sharp'
\simeq
u_\sharp' \Thom'(p'',a'').
\end{equation}
Theorem \ref{N2.17} gives an isomorphism
\begin{equation}
\label{N2.18.2}
d^*f_*u_\sharp'
\simeq
d^*f_\sharp \Omega_f u_\sharp'.
\end{equation}
Combine \eqref{N2.18.1} and \eqref{N2.18.2}, and use $d^*f_\sharp u_\sharp'\simeq d^*u_\sharp f_\sharp''\simeq 0$ to conclude.
\end{proof}

\appendix

\section{Categorical toolbox}

\subsection{Exchange transformations}

\begin{definition}
\label{exchange.1}
Let
\begin{equation}
\label{exchange.1.1}
Q
:=
\begin{tikzcd}
\cC\ar[d,"f^*"']\ar[r,"g^*"]&
\cC'\ar[d,"f'^*"]
\\
\cD\ar[r,"g'^*"]&
\cD'
\end{tikzcd}
\end{equation}
be a commutative square in $\infCat_\infty$.
Assume that $f^*$ and $f'^*$ admit left adjoints $f_\sharp$ and $f_\sharp'$.
The \emph{Beck-Chevalley transformation} (or \emph{exchange transformation}) is defined to be the composition
\[
Ex\colon f_\sharp' g'^*
\xrightarrow{ad}
f_\sharp' g'^* f^*f_\sharp
\xrightarrow{\simeq}
f_\sharp' f'^*g^*f_\sharp
\xrightarrow{ad'}
g^*f_\sharp.
\]
\end{definition}

\begin{definition}
\label{exchange.2}
Let $Q$ be a commutative square in $\infCat_\infty$ of the form \eqref{exchange.1.1}.
We say that \emph{$Q$ admits a left adjoint} if the following conditions are satisfied:
\begin{enumerate}
\item[(i)]
$f^*$ and $f'^*$ admit left adjoints $f_\sharp$ and $f_\sharp'$.
\item[(ii)]
The exchange transformation $Ex\colon f_\sharp' g'^* \to g^*f_\sharp$ is an isomorphism.
\end{enumerate}
In this case, we say that the commutative square
\[
\begin{tikzcd}
\cD\ar[d,"f_\sharp"']\ar[r,"g'^*"]&
\cD'\ar[d,"f_\sharp'"]
\\
\cC\ar[r,"g^*"]&
\cC'
\end{tikzcd}
\]
is a \emph{left adjoint of $Q$}.
\end{definition}

\begin{remark}
\label{exchange.3}
We are following Liu-Zheng \cite[Definition 2.2.1]{LZ} for the above terminology of a left adjoint of a square.
Lurie's terminology \cite[Definition 7.3.1.2]{HTT} is that $Q$ is \emph{left adjointable}.
\end{remark}

\begin{example}
\label{exchange.4}
Let $Q$ be a commutative square in $\infCat_\infty$ of the form \eqref{exchange.1.1}.
Assume that $f^*$ and $f'^*$ admit left adjoints $f_\sharp$ and $f_\sharp'$ and $g^*$ and $g'^*$ admit right adjoints $g_*$ and $g_*'$.
If the exchange transformation $Ex\colon f_\sharp'g'^* \to g^*f_\sharp$ is an isomorphism, then we have another exchange transformation
\[
Ex\colon f_\sharp g_*'
\xrightarrow{ad}
g_*g^*f_\sharp g_*'
\xrightarrow{Ex^{-1}}
g_*f_\sharp' g'^*g_*'
\xrightarrow{ad'}
g_*f_\sharp'.
\]
\end{example}

\subsection{Extensions of functors to the category of correspondences}

The purpose of this section is to review the technique of partial adjoints developed by Liu-Zheng \cite{LZ}.

We first recall the definition of the $\infty$-category of correspondences in \cite[\S 6.1]{LZ} as follows.
For every integer $n\geq 0$, let $\bC(\Delta^n)$ be the full subcategory of $\Delta^n \times (\Delta^n)^{op}$ spanned by $(i,j)$ such that $i\leq j$.
This construction defines a colimit preserving functor $\bC\colon \Set_\Delta\to \Set_\Delta$.
Let $\Corr\colon \Set_\Delta\to \Set_\Delta$ be its right adjoint.

Observe that $\bC(\Delta^n)$ is the nerve of the ordinary category associated with the diagram
\[
\begin{tikzcd}[column sep=small, row sep=small]
(0,n)
\ar[r]
\ar[d]
&
(0,n-1)
\ar[r]
\ar[d]
&
\cdots
\ar[r]
&
(0,1)
\ar[r]
\ar[d]
&
(0,0)
\\
(1,n)
\ar[r]
\ar[d]
&
(1,1)
\ar[d]
\ar[r]
&
\cdots
\ar[r]
&
(n-1,n-1)
\\
\vdots
\ar[d]
&
\vdots
\ar[d]
\\
(n-1,n)
\ar[d]
\ar[r]
&
(n-1,n-1)
\\
(n,n).
\end{tikzcd}
\]
A morphism of $\bC(\Delta^n)$ is \emph{vertical} (resp.\ \emph{horizontal}) if its image in the second (resp.\ first) factor degenerates.
A square in $\bC(\Delta^n)$ is \emph{exact} if it is cartesian and cocartesian, i.e., it is a square of the form
\[
\begin{tikzcd}
(m_1,n_1)\ar[d]\ar[r]&
(m_1,n_2)\ar[d]
\\
(m_2,n_1)\ar[r]&
(m_2,n_2)
\end{tikzcd}
\]
for some integers $m_1$, $m_2$, $n_1$, and $n_2$.

Suppose that $\cC$ is an $\infty$-category and $\cP$ and $\cQ$ are class of morphisms in $\cC$.
The \emph{simplicial set of correspondences} $\Corr(\cC,\cP,\cQ)$ is the simplicial subset of $\Corr(\cC)$ whose $n$-cells $\bC(\Delta^n)\to \cC$ sends vertical (resp.\ horizontal) morphisms into $\cP$ (resp.\ $\cQ$) and exact squares to pullback squares.
This is an $\infty$-category if $\cP$ and $\cQ$ are closed under pullbacks along one another and compositions by \cite[Lemma 6.1.2]{LZ}.

A morphism in $\Corr(\cC,\cP,\cQ)$ can be expressed as a diagram
\[
\begin{tikzcd}
X_{(0,1)}
\ar[d]
\ar[r]&
X_{(0,0)}
\\
X_{(1,1)}
\end{tikzcd}
\]
in $\cC$ such that the vertical morphism is in $\cP$ and the horizontal morphism is in $\cQ$.
A composition of morphisms can be expressed as a diagram
\[
\begin{tikzcd}
X_{(0,2)}
\ar[d]
\ar[r]
&
X_{(0,1)}
\ar[d]
\ar[r]
&
X_{(0,0)}
\\
X_{(1,2)}
\ar[d]
\ar[r]
&
X_{(1,1)}
\\
X_{(2,2)}
\end{tikzcd}
\]
in $\cC$ such that the vertical morphisms are in $\cP$, the horizontal morphisms are in $\cQ$, and the inner square is cartesian.
To simplify notation, we set
\[
f:=
\begin{tikzcd}
X'\ar[d,"\id"']\ar[r,"f"]&
X
\\
X'
\end{tikzcd}
\text{ and }
g^\dagger:=
\begin{tikzcd}
Y'\ar[d,"g"']\ar[r,"\id"]&
Y'
\\
Y
\end{tikzcd}
\]
for every $\cP$-morphism $f\colon X'\to X$ and $\cQ$-morphism $g\colon Y'\to Y$.

\begin{theorem}[Liu-Zheng]
\label{corr.1}
Let $\cP$ be a class of morphisms in an $\infty$-category $\cC$ closed under compositions and pullbacks, and let $\cF\colon \cC^{op}\to \infCat_\infty$ be a functor.
Assume that $\cF(f)$ admits a left adjoint for every $f\in \cP$ and $\cF(Q)$ admits a left adjoint for every cartesian square
\begin{equation}
\label{corr.1.1}
Q
:=
\begin{tikzcd}
X'\ar[d,"f'"']\ar[r,"g'"]&
X\ar[d,"f"]
\\
S'\ar[r,"g"]&
S
\end{tikzcd}
\end{equation}
in $\cC$ with $f\in \cP$.
Then $\cF$ can be naturally extended to a functor
\[
\cF
\colon
\Corr(\cC,\cP,\all)
\to
\infCat_\infty
\]
satisfying the following conditions:
\begin{enumerate}
\item[\textup{(i)}]
$\cF(f^\dagger)$ is a left adjoint of $\cF(f)$ for every $f\in \cP$.
\item[\textup{(ii)}]
For every cartesian square $Q$ of the form \eqref{corr.1.1} with $f\in \cP$, consider the induced commutative square 
\[
Q^\dagger
:=
\begin{tikzcd}
S'\ar[r,"g"]\ar[d,"f'^\dagger"']&
S\ar[d,"f^\dagger"]
\\
X'\ar[r,"g'"]&
X
\end{tikzcd}
\]
in $\Corr(\cC,\cP,\all)$.
Then $\cF(Q^\dagger)$ is a left adjoint of $\cF(Q)$.
\end{enumerate}
\end{theorem}
\begin{proof}
Let $\cC_{\cP,\all}^\mathrm{cart}$ be the cartesian bisimplicial nerve of $(\cC,\cP,\all)$ in the sense of \cite[Definition 1.3.16]{LZ}.
The $(m,n)$-simplicies of $\cC_{\cP,\all}^\mathrm{cart}$ are the maps
\[
\Delta^m\times\Delta^n\to \cC
\]
sending $e\times v$ to a morphism in $\cP$ and $e\times f$ to a cartesian square for every vertex $v$ in $\Delta^n$, edge $e$ in $\Delta^m$, and edge $f$ in $\Delta^n$.
We employ the functor
\[
\delta_{2,\{2\}}\colon \Set_{\Delta\times \Delta}\to \Set_{\Delta}
\]
in \cite[Definition 1.3.3(4)]{LZ}.

By \cite[Example 1.4.29]{LZ} (see also \cite[\S 6.1]{LZ}),
there is a categorical equivalence
\[
\delta_{2,\{2\}}^*
\cC_{\cP,\all}^\mathrm{cart}
\to
\Corr(\cC,\cP,\all).
\]
On the other hand, \cite[Proposition 2.2.4]{LZ} provides a functor
\[
\delta_{2,\{2\}}^*
\cC_{\cP,\all}^\mathrm{cart}
\to
\infCat_\infty.
\]
Combine these two to obtain
$
\cF
\colon
\Corr(\cC,\cP,\all)
\to
\infCat_\infty
$.
The conclusions (1) and (4) in loc.\ cit.\ imply that this is a natural extension of the original $\cF\colon \cC^{op}\to \infCat_\infty$.
The conclusions (2) and (3) in loc.\ cit.\ imply our conditions (i) and (ii).
\end{proof}

\subsection{Equivalences of \texorpdfstring{$\infty$}{00}-topoi}

Let $\cC$ be an ordinary category.
Consider the functor $\iota \colon \Set\to \infSpc$ sending a set to the corresponding discrete space.
This induces a functor
\[
\iota
\colon
\PSh(\cC)
\to
\PSh(\cC,\infSpc).
\]
For $X\in \cC$, let
\[
p_X\in \PSh(\cC)
\text{ and }
\mathfrak{p}_X\in \PSh(\cC,\infSpc)
\]
be the presheaves represented by $X$.
These two are related by
\begin{equation}
\mathfrak{p}_X
\simeq
\iota p_X.
\end{equation}

A presheaf of spaces $\cF\in \PSh(\cC,\infSpc)$ is a \emph{sheaf} if $\cF$ satisfies \v{C}ech descent,
i.e.,
we have
\[
\cF(X)
\simeq
\lim
\big(
\prod_{i\in I} \cF(X_i)
\;
\substack{\rightarrow\\[-1em] \rightarrow}
\;
\prod_{i,j\in I} \cF(X_i\times_X X_j)
\;
\substack{\rightarrow\\[-1em] \rightarrow \\[-1em] \rightarrow}
\;
\cdots
\big)
\]
for every covering $\{X_i\to X\}_{i\in I}$.
This definition of sheaves of spaces is equivalent to that in \cite[Definition 6.2.2.6]{HTT} by \cite[Theorem A5]{MR2034012}.

Let $\Sh(\cC,\infSpc)$ be the full subcategory of $\PSh(\cC,\infSpc)$ consisting of sheaves,
which is the localization of $\PSh(\cC,\infSpc)$ with respect to the class of morphisms
\[
\colim
\big(
\cdots
\;
\substack{\rightarrow\\[-1em] \rightarrow \\[-1em] \rightarrow}
\;
\coprod_{i,j\in I} \mathfrak{p}_{X_i\times_X X_j}
\;
\substack{\rightarrow\\[-1em] \rightarrow}
\;
\coprod_{i\in I} \mathfrak{p}_{X_i}
\big)
\to
\mathfrak{p}_X
\]
for every covering $\{X_i\to X\}_{i\in I}$.

Let $\cC$ be a site.
We have an adjunction
\[
a^*:
\PSh(\cC)\rightleftarrows \Sh(\cC):a_*,
\]
where $a^*$ is the sheafification functor,
and $a_*$ is the inclusion functor.
We have an adjunction
\[
L^*
:
\PSh(\cC,\infSpc)
\rightleftarrows
\Sh(\cC,\infSpc)
:
L_*,
\]
where $L^*$ is the localization functor,
and $L_*$ is the inclusion functor.
For $X\in \cC$,
let
\[
h_X\in \Sh(\cC)
\text{ and }
\mathfrak{h}_X\in \Sh(\cC,\infSpc)
\]
be the sheaves represented by $X$.
There are relations
\[
h_X\simeq a^*p_X
\text{ and }
\mathfrak{h}_X\simeq L^*\mathfrak{p}_X.
\]
Let
$
\iota\colon \Sh(\cC)\to \Sh(\cC,\infSpc)
$
be the composition $L^*\iota a_*$.

\begin{proposition}
\label{equivalence.2}
Let $\cC$ be a site.
Then the square
\[
\begin{tikzcd}
\PSh(\cC)\ar[d,"a^*"']\ar[r,"\iota"]&
\PSh(\cC,\infSpc)\ar[d,"L^*"]
\\
\Sh(\cC)\ar[r,"\iota"]&
\Sh(\cC,\infSpc)
\end{tikzcd}
\]
commutes.
In particular,
$\iota h_X\simeq \mathfrak{h}_X$ for all $X\in \cC$.
\end{proposition}
\begin{proof}
The natural transformation
\[
L^*\iota 
\xrightarrow{ad}
L^*\iota a_*a^*
\]
is an isomorphism by \cite[Proposition A2]{MR2034012},
which shows the claim.
\end{proof}

Let $f\colon \cC\to \cD$ be a functor of categories.
By \cite[Proposition I.5.1]{SGA4},
there is an adjunction
\[
f^*:
\PSh(\cC)
\rightleftarrows
\PSh(\cD)
:
f_*
\]
such that $f_*\cG(X):=\cG(f(X))$ for all $\cG\in \PSh(\cD)$ and $X\in \cC$.
By \cite[Proposition 5.2.6.3]{HTT},
there is an adjunction
\[
f^*:
\PSh(\cC,\infSpc)
\rightleftarrows
\PSh(\cD,\infSpc)
:f_*
\]
such that $f_*\cG(X):=\cG(f(X))$ for all $\cG\in \PSh(\cD,\infSpc)$ and $X\in \cC$.
Since this $f_*$ preserves colimits, there is an adjunction
\[
f_*
:
\PSh(\cD,\infSpc)
\rightleftarrows
\PSh(\cC,\infSpc)
:
f^!.
\]

If $f\colon \cC\to \cD$ is a continuous functor of sites,
then by \cite[Proposition III.1.2]{SGA4},
there is an adjunction
\[
f^*:
\Sh(\cC)
\rightleftarrows
\Sh(\cD)
:
f_*
\]
such that $f_*\cG(X):=\cG(f(X))$ for all $\cG\in \Sh(\cD)$ and $X\in \cC$.

\begin{proposition}
\label{equivalence.3}
Let $f\colon \cC\to \cD$ be a functor of categories.
Then the square
\[
\begin{tikzcd}
\PSh(\cD)
\ar[d,"f_*"']
\ar[r,"\iota"]&
\PSh(\cD,\infSpc)\ar[d,"f_*"]
\\
\PSh(\cC)\ar[r,"\iota"]&
\PSh(\cC,\infSpc)
\end{tikzcd}
\]
commutes.
\end{proposition}
\begin{proof}
Observe that we have natural isomorphisms
\[
\iota f_*\cG(Y)\simeq \cG(f(Y))\simeq f_*\iota \cG(Y)
\]
for all $\cG\in \PSh(\cD)$ and $Y\in \cD$.
\end{proof}

\begin{definition}
We say that a functor of sites $f\colon \cC\to \cD$ \emph{preserves \v{C}ech coverings} if the following conditions are satisfied for every covering $\{X_i\to X\}_{i\in I}$ in $\cC$:
\begin{enumerate}
\item[(i)]
$\{f(X_i)\to f(X)\}_{i\in I}$ is a covering in $\cD$.
\item[(ii)]
$f(X_{i_1})\times_{f(X)}\cdots \times_{f(X)}f(X_{i_n})$ is representable and the induced morphism
\[
f(X_{i_1}\times_X \cdots \times_X X_{i_n})
\to
f(X_{i_1})\times_{f(X)}\cdots \times_{f(X)}f(X_{i_n})
\]
is an isomorphism for all $i_1,\ldots,i_n\in I$.
\end{enumerate}
\end{definition}

\begin{proposition}
\label{equivalence.9}
Let $f\colon \cC\to \cD$ be a functor of sites preserving \v{C}ech coverings.
Then there is an adjunction
\begin{equation}
\label{equivalence.9.1}
f^*
:
\Sh(\cC,\infSpc)
\rightleftarrows
\Sh(\cD,\infSpc)
:
f_*,
\end{equation}
where $f_*\cG(X):=\cG(f(X))$ for all $\cG\in \Sh(\cD,\infSpc)$ and $X\in \cC$.
Furthermore,
the square
\begin{equation}
\label{equivalence.9.2}
\begin{tikzcd}
\PSh(\cC,\infSpc)\ar[r,"L^*"]\ar[d,"f^*"']&
\Sh(\cC,\infSpc)\ar[d,"f^*"]
\\
\PSh(\cD,\infSpc)\ar[r,"L^*"]&
\Sh(\cD,\infSpc)
\end{tikzcd}
\end{equation}
commutes.
\end{proposition}
\begin{proof}
Let us show that $f_*L_*\cG$ is a sheaf for all $\cG\in \Sh(\cC,\infSpc)$.
We have
\[
f_*L_*\cG(X)
\simeq
\cG(f(X))
\]
for all $X\in \cC$.
Let $\{X_i\to X\}_{i\in I}$ be a covering in $\cC$.
Then $\{f(X_i)\to f(X)\}_{i\in I}$ is a covering in $\cD$ since $f$ preserves \v{C}ech covering.
Furthermore,
\[
\cG(f(X))
\simeq
\lim
\big(
\prod_{i\in I} \cG(f(X_i))
\;
\substack{\rightarrow\\[-1em] \rightarrow}
\;
\prod_{i,j\in I} \cG(f(X_i)\times_{f(X)} f(X_j))
\;
\substack{\rightarrow\\[-1em] \rightarrow \\[-1em] \rightarrow}
\;
\cdots
\big)
\]
implies
\[
f_*L_*\cF(X)
\simeq
\lim
\big(
\prod_{i\in I} f_*L_*\cF(X_i)
\;
\substack{\rightarrow\\[-1em] \rightarrow}
\;
\prod_{i,j\in I} f_*L_*\cF(X_i\times_X X_j)
\;
\substack{\rightarrow\\[-1em] \rightarrow \\[-1em] \rightarrow}
\;
\cdots
\big).
\]
Hence $f_*L_*\cG$ is a sheaf.
It follows that the natural transformation
\begin{equation}
\label{equivalence.9.3}
f_*L_*\xrightarrow{ad}
L_*L^*f_*L_*
\end{equation}
is an isomorphism.

For all $\cG\in \Sh(\cC,\infSpc)$, consider the morphisms
\begin{align*}
\Hom_{\Sh(\cD,\infSpc)}(L^*f^*L_*\cF,\cG)
\xrightarrow{\simeq} &
\Hom_{\PSh(\cC,\infSpc)}(L_*\cF,f_*L_*\cF)
\\
\to &
\Hom_{\Sh(\cC,\infSpc)}(L^*L_*\cF,L^*f_*L_*\cG)
\\
\xrightarrow{\simeq} &
\Hom_{\Sh(\cC,\infSpc)}(\cF,L^*f_*L_*\cG).
\end{align*}
The second arrow is an isomorphism since \eqref{equivalence.9.3} is an isomorphism.
It follows that $L^*f_*L_*$ is a right adjoint of $L^*f^*L_*$.
Since
\[
L^*f_*L_*\cG(X)
\simeq
f_*L_*\cG(X)
\simeq
\cG(f(X)),
\]
we have
\begin{equation}
\label{equivalence.9.4}
f_*\simeq L^*f_*L_*.
\end{equation}
We set $f^*:=L^*f^*L_*$ to obtain \eqref{equivalence.9.1}.
From \eqref{equivalence.9.3},
we have $f_*L_*\simeq L_*f_*$.
Take left adjoints to see that \eqref{equivalence.9.2} commutes.
\end{proof}

\begin{proposition}
\label{equivalence.7}
Let $f\colon \cC\to \cD$ be a continuous and cocontinuous functor of sites preserving \v{C}ech coverings.
Assume that
$
f_*
\colon
\Sh(\cD)
\to
\Sh(\cC)
$
is an equivalence of categories.
Then $f_*\colon \Sh(\cD,\infSpc)\to \Sh(\cC,\infSpc)$ preserves colimits,
and $f^*\colon \Sh(\cC,\infSpc)\to \Sh(\cD,\infSpc)$ is fully faithful.
\end{proposition}
\begin{proof}
Recall from \eqref{equivalence.9.4} that we have $f_*\simeq L^*f_*L_*$.
Let us first show that $L^*f^!L_*$ is a right adjoint of $L^*f_*L_*$.
For this,
it suffices to show that $f^!L_*\cF$ is a sheaf for all $\cF\in \Sh(\cC,\infSpc)$ since this implies that the second arrow in
\begin{align*}
\Hom_{\Sh(\cC,\infSpc)}(L^*f_*L_*\cG,\cF)
\xrightarrow{\simeq} &
\Hom_{\PSh(\cD,\infSpc)}(L_*\cG,f^!L_*\cF)
\\
\to &
\Hom_{\Sh(\cD,\infSpc)}(L^*L_*\cG,L^*f^!L_*\cF)
\\
\xrightarrow{\simeq} &
\Hom_{\Sh(\cD,\infSpc)}(\cG,L^*f^!L_*\cF)
\end{align*}
is an isomorphism for all $\cG\in \Sh(\cD,\infSpc)$.
We have isomorphisms
\[
L^*f_*\mathfrak{p}_Y
\simeq
L^*f_* \iota p_Y
\simeq
L^* \iota f_* p_Y
\simeq
\iota a^* f_* p_Y
\simeq
\iota f_* a^* p_Y
\simeq
\iota f_* h_Y.
\]
for all $Y\in \cD$,
where we need Proposition \ref{equivalence.3} for the second isomorphism, Proposition \ref{equivalence.2} for the third isomorphism, and \cite[Proposition III.2.3(3)]{SGA4} for the fourth isomorphism.
Hence we have a natural isomorphism
\[
f^!L_*\cF(Y)
\simeq
\Hom_{\Sh(\cC,\infSpc)}(\iota f_* h_Y,\cF).
\]
We need to show
\[
f^!L_*\cF(Y)
\simeq
\lim
\big(
\prod_{i\in I} f^! L_*\cF(Y_i)
\;
\substack{\to\\[-1em] \to}
\;
\prod_{i,j\in I} f^!L_*\cF(Y_i\times_Y Y_j)
\;
\substack{\rightarrow\\[-1em] \rightarrow \\[-1em] \rightarrow}
\;
\cdots
\big)
\]
for every covering $\{Y_i\to Y\}_{i\in I}$ in $\cD$.
To show this,
it suffices to show
\[
\colim
\big(
\cdots
\;
\substack{\rightarrow\\[-1em] \rightarrow \\[-1em] \rightarrow}
\;
\coprod_{i,j\in I} \iota f_*h_{Y_i\times_Y Y_j}
\;
\substack{\rightarrow\\[-1em] \rightarrow}
\;
\coprod_{i\in I} \iota f_*h_{Y_i}
\big)
\simeq
\iota f_* h_Y.
\]
By \cite[Corollary A3]{MR2034012},
it suffices to show that
\[
\cdots
\;
\substack{\rightarrow\\[-1em] \rightarrow \\[-1em] \rightarrow}
\;
\coprod_{i,j\in I}  f_*h_{Y_i\times_Y Y_j}
\;
\substack{\rightarrow\\[-1em] \rightarrow}
\;
\coprod_{i\in I}  f_*h_{Y_i}
\to
f_*h_Y
\]
is a \v{C}ech nerve and $f_*h_{Y_i}\to f_*h_Y$ is an epimorphism.
This is a consequence of the assumption that
$
f_*
\colon
\Sh(\cD)
\to
\Sh(\cC)
$
is an equivalence.
Hence we have shown that $f^!L_*\cF$ is a sheaf for all $\cF\in \Sh(\cC,\infSpc)$ and $f_*\colon \Sh(\cD,\infSpc)\to \Sh(\cC,\infSpc)$ preserves colimits.

It follows that the natural transformation
\[
f^!L_*
\xrightarrow{ad}
L_*L^*f^!L_*
\]
is an isomorphism.
Take left adjoints to see that the right square in the diagram
\[
\begin{tikzcd}
\Sh(\cD)\ar[d,"f_*"']\ar[r,"a_*"]&
\PSh(\cD)\ar[d,"f_*"]\ar[r,"\iota"]&
\PSh(\cD,\infSpc)\ar[r,"L^*"]\ar[d,"f_*"]&
\Sh(\cD,\infSpc)\ar[d,"f_*"]
\\
\Sh(\cC)\ar[r,"a_*"]&
\PSh(\cC)\ar[r,"\iota"]&
\PSh(\cC,\infSpc)\ar[r,"L^*"]&
\Sh(\cC,\infSpc).
\end{tikzcd}
\]
commutes.
The middle square commutes by Proposition \ref{equivalence.3}.
The left square commutes since $f$ is a continuous functor of sites.
It follows that the square
\begin{equation}
\label{equivalence.7.1}
\begin{tikzcd}
\Sh(\cD)\ar[d,"f_*"']\ar[r,"\iota"]&
\Sh(\cD,\infSpc)\ar[d,"f_*"]
\\
\Sh(\cC)\ar[r,"\iota"]&
\Sh(\cC,\infSpc)
\end{tikzcd}
\end{equation}
commutes.
From this, we have the exchange transformation
\[
\iota f^* \xrightarrow{Ex} f^* \iota.
\]

We claim that the morphism
\begin{equation}
\label{equivalence.7.2}
\iota f^* h_X
\xrightarrow{Ex}
f^*\iota h_X
\end{equation}
is an isomorphism for all $X\in \cC$.
Consider the induced commutative diagram
\[
\begin{tikzcd}
\iota a^* f^*p_X\ar[r]\ar[d]&
\iota f^* a^*p_X\ar[r,"Ex"]&
f^* \iota a^*p_X\ar[d]
\\
L^* \iota f^*p_X\ar[r]&
L^* f^* \iota p_X\ar[r]&
f^* L^* \iota p_X.
\end{tikzcd}
\]
The left vertical arrow is an isomorphism since \eqref{equivalence.7.1} commutes.
The lower right horizontal arrow is an isomorphism by Proposition \ref{equivalence.9}.
From
\begin{gather*}
\iota a^* p_X\simeq \mathfrak{h}_X\simeq L^*\iota p_X,
\\
a^*f^*p_X\simeq h_{f(X)}\simeq f^*a^*p_X,
\\
\iota f^* p_X\simeq \mathfrak{p}_{f(X)}\simeq f^*\iota p_X,
\end{gather*}
we see that the right vertical, upper left horizontal, and lower left horizontal arrows are isomorphisms.
It follows that the upper left horizontal arrow is an isomorphism,
i.e., \eqref{equivalence.7.2} is an isomorphism.

To show that $f^*\colon \Sh(\cC,\infSpc)\to \Sh(\cD,\infSpc)$ is fully faithful,
it suffices to show that the morphism
\begin{equation}
\label{equivalence.7.3}
\mathfrak{h}_X \xrightarrow{ad} f_*f^* \mathfrak{h}_X
\end{equation}
is an isomorphism for all $X\in \cC$ since every sheaf of spaces is isomorphic to a colimit of representable sheaves of spaces and $f^*$ and $f_*$ preserve colimits.
Consider the commutative diagram
\[
\begin{tikzcd}
\iota h_X\ar[r,"ad"]\ar[rrd,"ad"',bend right=10]&
\iota f_*f^* h_X\ar[r,"\simeq"]&
f_* \iota f^* h_X\ar[d,"Ex"]
\\
&
&
f_*f^*\iota h_X
\end{tikzcd}
\]
The right vertical arrow is an isomorphism since \eqref{equivalence.7.2} is an isomorphism.
The upper left horizontal arrow is an isomorphism since $f_*\colon \Sh(\cD)\to \Sh(\cC)$ is an equivalence.
It follows that the diagonal arrow is an isomorphism,
i.e., \eqref{equivalence.7.3} is an isomorphism.
\end{proof}

\begin{theorem}
\label{equivalence.8}
Let $f\colon \cC\to \cD$ be a continuous and cocontinuous functor of sites preserving \v{C}ech coverings.
Assume the following conditions:
\begin{enumerate}
\item[\textup{(i)}]
For all $Y\in \cD$, there exists a covering $\{Y_i\to Y\}_{i\in I}$ in $\cD$ such that
$
Y_{i_1}\times_Y \cdots \times_Y Y_{i_n}\in \im f
$
for all $i_1,\ldots,i_n\in I$.
\item[\textup{(ii)}]
$
f_*
\colon
\Sh(\cD)
\to
\Sh(\cC)
$
is an equivalence,
\end{enumerate}
Then $f_*\colon \Sh(\cD,\cV)\to \Sh(\cC,\cV)$ is an equivalence of $\infty$-categories for every $\infty$-category $\cV$ admitting limits.
\end{theorem}
\begin{proof}
By \cite[Proposition 1.3.1.7]{SAG},
there is an equivalence between $\Sh(\cE,\cV)$ and the $\infty$-category of functors from $\Sh(\cE,\infSpc)$ to $\cV$ sending colimits to limits for every site $\cE$.
Hence we reduce to the case when $\cV=\infSpc$.

Proposition \ref{equivalence.7} means that $f^*$ is fully faithful and $f_*$ preserves colimits.
Hence it suffices to show that the morphism
\[
f^*f_*\cG\xrightarrow{ad'} \cG
\]
is an isomorphism for all $\cG\in \Sh(\cD,\infSpc)$.
Since $f^*$ preserves colimits and $\cG$ is a colimit of representable sheaves of spaces,
it suffices to show that the morphism
\[
f^*f_*\mathfrak{h}_Y\xrightarrow{ad'} \mathfrak{h}_Y
\]
is an isomorphism for all $Y\in \cD$.
Consider the covering $\{Y_i\to Y\}_{i\in I}$ in $\cD$ in the assumption.
Since
\[
\colim
\big(
\cdots
\;
\substack{\rightarrow\\[-1em] \rightarrow \\[-1em] \rightarrow}
\;
\coprod_{i,j\in I}  \mathfrak{h}_{Y_i\times_Y Y_j}
\;
\substack{\rightarrow\\[-1em] \rightarrow}
\;
\coprod_{i\in I}  \mathfrak{h}_{Y_i}
\big)
\simeq
\mathfrak{h}_Y,
\]
we only need to show that the morphism
\[
f^*f_*\mathfrak{h}_{Y_{i_1}\times_Y \cdots \times_Y Y_{i_n}}\xrightarrow{ad'} \mathfrak{h}_{Y_{i_1}\times_Y \cdots \times_Y Y_{i_n}}
\]
is an isomorphism for all $i_1,\ldots,i_n\in I$.
There exists $X\in \cC$ such that $f(X)\simeq Y_{i_1}\times_Y \cdots \times_Y Y_{i_n}$.
This implies $\mathfrak{h}_{Y_{i_1}\times_Y \cdots \times_Y Y_{i_n}}\simeq f^*\mathfrak{h}_X$.
To conclude,
observe that the composition
\[
f^*\mathfrak{h}_X\xrightarrow{ad}f^*f_*f^*\mathfrak{h}_X
\xrightarrow{ad'} f^*\mathfrak{h}_X
\]
is the identity and the first arrow is an isomorphism since $f^*$ is fully faithful.
\end{proof}

\bibliography{bib}
\bibliographystyle{siam}

\end{document}